\date{\today}
\documentclass[10pt]{amsart}
\usepackage{amssymb,amsfonts,amsthm,amsmath,verbatim,lscape,color,tensor,url,todonotes}
\usepackage{stmaryrd}
\usepackage[margin=1.5in]{geometry}
\usepackage[enableskew]{youngtab}
\usepackage[all]{xy}
\usepackage{enumerate}
\usepackage{mathrsfs}
\usepackage{caption}
\def\biblio{\bibliography{bibliography}\bibliographystyle{alpha}}

\usepackage[pagebackref]{hyperref}

\usepackage{subfiles}
\usepackage{graphicx}

\usepackage[nameinlink,capitalise,noabbrev]{cleveref}
\newtheorem{thm}{Theorem}[section]
\newtheorem*{thm*}{Theorem}
\newtheorem{cor}[thm]{Corollary}
\newtheorem*{cor*}{Corollary}
\newtheorem{prop}[thm]{Proposition}
\newtheorem{lem}[thm]{Lemma}

\theoremstyle{definition}
\newtheorem{defn}[thm]{Definition}

\newtheorem{ex}[thm]{Example}

\theoremstyle{remark}
\newtheorem{rem}[thm]{Remark}
\DeclareMathOperator{\map}{map}

\DeclareMathOperator{\Ho}{Ho}
\DeclareMathOperator{\Aut}{Aut}

\DeclareMathOperator{\Alg}{Alg}
\DeclareMathOperator{\CAlg}{CAlg}

\DeclareMathOperator{\Hom}{Hom}
\DeclareMathOperator{\Ext}{Ext}
\DeclareMathOperator{\colim}{colim}

\DeclareMathOperator{\End}{End}

\DeclareMathOperator{\Id}{Id}
\DeclareMathOperator{\Tot}{Tot}
\DeclareMathOperator{\Ind}{Ind}

\DeclareMathOperator{\Mod}{\mathrm{Mod}}

\DeclareMathOperator{\TwArr}{TwArr}

\DeclareMathOperator{\Pic}{Pic}
\DeclareMathOperator{\pic}{\mathfrak{pic}}

\DeclareMathOperator{\Sp}{Sp}
\DeclareMathOperator{\Loc}{Loc}

\DeclareMathOperator{\Comod}{Comod}

\DeclareMathOperator{\Cat}{Cat_{\infty}}

\DeclareMathOperator{\Catomega}{Cat_{\infty}^{\omega}}
\DeclareMathOperator{\Fun}{Fun}
\DeclareMathOperator{\Pre}{Pre}
\DeclareMathOperator{\Map}{Map}
\DeclareMathOperator{\Top}{Top}
\DeclareMathOperator{\Thick}{Thick}

\DeclareMathOperator{\op}{op}

\DeclareMathOperator{\alg}{alg}

\DeclareMathOperator{\Ch}{Ch}

\DeclareMathOperator{\comod}{Comod}
\DeclareMathOperator{\Franke}{Fr}

\newcommand{\cA}{\mathcal{A}}
\newcommand{\Fr}{\Franke_{n}}
\newcommand{\Frnp}{\Franke_{n,p}}

\newcommand{\noloc}{\;\mathord{:}\,}
\newcommand{\N}{\mathbb{N}}
\newcommand{\Z}{\mathbb{Z}}
\newcommand{\F}{\mathbb{F}}
\newcommand{\Q}{\mathbb{Q}}

\newcommand{\G}{\mathbb{G}}
\newcommand{\E}{E_{n,p}}
\newcommand{\bA}{\mathbb{A}}
\newcommand{\bE}{\mathbb{E}}

\newcommand{\pS}{S^{0}_{p}}
\newcommand{\hS}{\hat{S}}
\newcommand{\roundbrackets}[1]{(\!(#1)\!)}

\newcommand{\cC}{\mathcal{C}}
\newcommand{\cD}{\mathcal{D}}
\newcommand{\cF}{\mathcal{F}}

\newcommand{\cL}{\mathcal{L}}

\newcommand{\cO}{\mathcal{O}}
\newcommand{\cP}{\mathcal{P}}
\newcommand{\cS}{\mathcal{S}}

\newcommand{\lra}[1]{\overset{#1}{\longrightarrow}}

\newcommand{\doubleflat}{\flat\kern-1.0pt\flat}

\newcommand{\Prod}[1]{{\prod}_{#1}}
\newcommand{\Oplus}[1]{\underset{#1}{\bigoplus}}

\newcommand{\Coprod}[1]{\underset{#1}{\coprod}}
\newcommand{\Colim}[1]{\underset{#1}{\colim}}

\makeatletter
\let\c@equation\c@thm
\makeatother

\numberwithin{equation}{section}
\Crefname{figure}{Figure}{Figures}

\DeclareMathOperator{\Cell}{Cell}
\DeclareMathOperator{\DimCell}{DimCell}

\DeclareMathOperator{\Set}{Set}
\DeclareMathOperator{\sSet}{sSet}
\DeclareMathOperator{\Sh}{Sh}

\newcommand{\cE}{\mathcal{E}}

\newcommand{\Mfg}{{\mathcal{M}_{\mathrm{fg}}}}
\newcommand{\Mfgn}{\mathcal{M}_{\mathrm{fg}}^{\leq n}}
\renewcommand{\mod}{\text{mod}}

\newcommand{\Fins}{N(\mathcal{F}\mathrm{in}_*)}

\newcommand{\gr}{\mathrm{graded}}
\newcommand{\Op}{\mathrm{Op}_{\infty}}
\newcommand{\POp}{\mathrm{POp}_{\infty}}

\newcommand{\sCat}{\mathrm{Cat}_{\infty}^{\otimes}}

\newcommand{\LM}{\mathrm{M}^{\mathrm{loc}}}
\newcommand{\MS}{\sSet^{+}}

\begin{document}
\title[Chromatic homotopy is asymptotically algebraic]{Chromatic homotopy theory is\\ asymptotically algebraic}
\author{Tobias Barthel}
\address{Max Planck Institute for Mathematics\\ Bonn, Germany}
\email{tbarthel@mpim-bonn.mpg.de}
\author{Tomer M. Schlank}
\address{The Hebrew University of Jerusalem \\ Jerusalem, Israel}
\email{tomer.schlank@mail.huji.ac.il}
\author{Nathaniel Stapleton}
\address{University of Kentucky\\ Lexington, Kentucky, USA}
\email{nat.j.stapleton@uky.edu}


\begin{abstract}
Inspired by the Ax--Kochen isomorphism theorem, we develop a notion of categorical ultraproducts to capture the generic behavior of an infinite collection of mathematical objects. We employ this theory to give an asymptotic solution to the approximation problem in chromatic homotopy theory. More precisely, we show that the ultraproduct of the $E(n,p)$-local categories over any non-principal ultrafilter on the set of prime numbers is equivalent to the ultraproduct of certain algebraic categories introduced by Franke. This shows that chromatic homotopy theory at a fixed height is asymptotically algebraic. 
\end{abstract}

\maketitle

\begin{figure}[h!]
  \includegraphics[scale=.40]{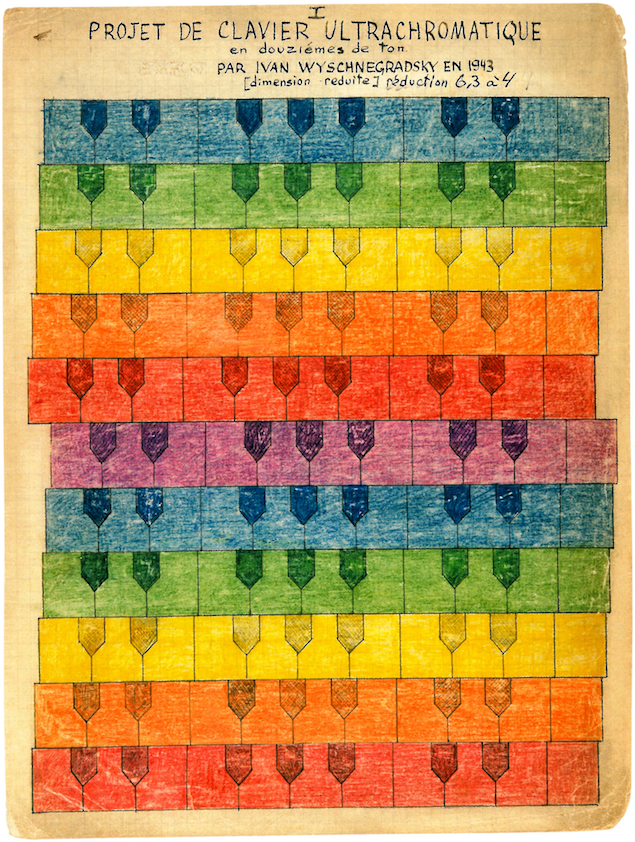}
\end{figure}

\setcounter{tocdepth}{1}
{\hypersetup{linkcolor=black}\tableofcontents}

\def\biblio{}

\section{Introduction}

\paragraph{\bf{Motivation and background}}
The guiding problem in stable homotopy theory is the computation of the stable homotopy groups of spheres $\pi_*S^0$. The first attempts at this calculation, for example via a Postnikov filtration, are of limited use and only provide coarse information about the large scale structure of $\pi_*S^0$. Motivated by patterns seen in the Adams spectral sequence, chromatic homotopy theory yields a more efficient filtration of $\pi_*S^0$ through localizations $L_{n,p}S^0$ of the sphere spectrum at the chromatic primes $(n,p)$. These localizations fit into the chromatic tower
\[
\xymatrix{\ldots \ar[r] & L_{n,p}S^0 \ar[r] & \ldots \ar[r] & L_{1,p}S^0 \ar[r] & L_{0,p}S^0 \simeq S_{\Q_p}^0}
\]
and the chromatic convergence theorem of Hopkins and Ravenel implies that the resulting filtration on $\pi_*S^0$ is exhaustive. In fact, this tower arises from an ascending filtration of the ($p$-local) stable homotopy category $\Sp$,
\[
\Sp_{\Q} \simeq \Sp_{0,p} \subset \Sp_{1,p} \subset \ldots \subset \Sp_{n,p} = L_{n,p}\Sp \subset \ldots \subset \Sp,
\]
with filtration quotients equivalent to the category of $K(n,p)$-local spectra, where $K(n,p)$ is Morava $K$-theory. This filtration is exhaustive when restricted to finite spectra, but not in general. The chromatic approach thus divides the computation of $\pi_*S^0$ into two main problems:
 \begin{enumerate}
  \item The study of the categories $\Sp_{n,p}$ and the calculation of the local homotopy groups $\pi_*L_{{n,p}}S^0$ for each $n \ge 0$ and every prime $p$.
  \item The question of how to assemble these local pieces to reconstruct $\pi_*S^0$.
 \end{enumerate}
 
The goal of this paper and its sequel is to show that, asymptotically in the prime $p$, both problems are controlled entirely by algebraic geometry. More generally, our main result provides a solution to the longstanding open question of finding a good algebraic approximation to $\Sp_{n,p}$ for $n< \infty$. Serre's work addresses this problem when $n=0$ and shows that this case is governed entirely by the theory of rational vector spaces. However, for every $n>0$ and any prime $p$, these categories do not admit an algebraic model. 

It is nevertheless possible to algebraically model partial information about $\Sp_{n,p}$. For instance, Bousfield gave a purely algebraic classification of all homotopy types in this category when $n=1$. As $n$ goes to infinity, the complexity of $\Sp_{n,p}$ increases rapidly and all known algebraic approximations grow coarser. In the extremal case $n=\infty$, the ``Mahowald uncertainty principle" states that even the homotopy types of objects cannot be modeled algebraically.  

To describe the algebraic approximations we will use, recall that Quillen's work on complex cobordism reveals a close connection between stable homotopy theory and the moduli stack of formal groups $\Mfg$: the cohomology of the tensor powers of the canonical line bundle on the moduli stack forms the $E_2$-page of the Adams--Novikov spectral sequence converging to $\pi_*S^0$. The $p$-local moduli stack $(\Mfg)_p$ admits an increasing filtration by the open substacks $(\Mfg)_{n,p}$ consisting of formal groups of height $\leq n$ at the prime $p$ that mirrors the chromatic filtration as observed by Morava. Our approximation to $\Sp_{n,p}$ will be a category $\Frnp$ of twisted complexes of quasi-coherent sheaves on $(\Mfg)_{n,p}$ introduced and first studied systematically by Franke. 

We can now state a first version of our main result.
\begin{thm*}
For any $n \geq 0$, there is a symmetric monoidal equivalence
\[
\lim_{p \to \infty} \Sp_{n,p} \simeq \lim_{p \to \infty} \Frnp.
\]
\end{thm*}

The limit notation is justified as we are capturing the asymptotic behavior of these categories, however it does not stand for the categorical limit (or the topological limit). Indeed, there are no natural functors between the categories as $p$ varies. Instead, to produce a limiting object out of the collections of categories $\Sp_{n,p}$ and $\Frnp$ as $p \to \infty$, we construct a categorical analogue of the model-theoretic notion of ultraproducts. A key feature of the ultraproduct construction is that it captures the generic behavior of a collection of objects. Thus the theorem above allows one to study questions about the generic behavior of chromatic homotopy theory by purely algebraic means. \\

\paragraph{{\bf In more detail.}} It has long been understood that chromatic homotopy theory at a fixed height $n$ simplifies as the prime tends towards infinity. This simplification manifests itself as sparseness in various spectral sequences leading to certain topological constructions (the existence of Smith--Toda complexes, Picard groups, homotopy groups of finite complexes) being completely controlled by algebra. However the size of the prime needed for these constructions to be purely algebraic depends on the construction. 

In the early 1990s, Franke \cite{Franke_exotic} introduced categories $\Frnp$ of quasi-periodic chain complexes of comodules whose homotopy theory was intended to converge to $\Sp_{n,p}$ in the limit $p \to \infty$. However, as observed by Patchkoria \cite{patchkoria}, his work remains incomplete due to the difficulty of directly comparing the algebraic categories to the topological categories. In contrast, our approach to circumvent this problem is based on and inspired by concepts from mathematical logic. 

Our use of ultraproducts resembles the use in the celebrated Ax--Kochen isomorphism theorem (\cite{AxKoc1,AxKoc2,AxKoc3}).  Their theorem produces an astounding bridge between local fields of characteristic $0$ and local fields of characteristic $p$, which is a non-canonical isomorphism between ultraproducts
\[
\Prod{\cF} \Q_p \cong \Prod{\cF} \F_p\roundbrackets{x},
\]
for every non-principal ultrafilter $\cF$ on the prime numbers. Thus, {\L}os's theorem implies that a first order statement in the theory of fields holds for $\F_p\roundbrackets{x}$ for almost all $p$ if and only if it holds for $\Q_p$ for almost all $p$. For example, Lang proved the existence of non-trivial zeros for all degree $d$ homogeneous polynomials in at least $d^2+1$ variables over $\F_p\roundbrackets{x}$. Since the existence of such a zero is a first order statement, for any prime $p$ different from a finite number of exceptional primes, any degree $d$ homogeneous polynomial in at least $d^2+1$ variables with coefficients in $\Q_p$ has a non-trivial zero. 

More generally, let $\Prod{\cF} M_i$ be the ultraproduct of a collection of models $(M_i)_{i\in I}$ for some first order theory with respect to an ultrafilter $\cF$ on $I$. {\L}os's theorem states that a first order statement is true for $M_i$ for almost all $i \in I$ if and only if it is true for $\Prod{\cF} M_i$ for all non-principal $\cF$. Thus {\L}os's theorem demonstrates that ultraproducts at non-principal ultrafilters can be used to capture the asymptotic behavior of a collection of models. At the same time, the ultraproducts $\Prod{\cF} M_i$ for non-principal $\cF$ often exhibit interesting new features. For instance, in the case of the isomorphism above, the ultraproduct of characteristic $p$ fields $\Prod{\cF} \F_p\roundbrackets{x}$ is a characteristic $0$ field. 

We modify the definition of the ultraproduct to function in the homotopical world. Given a collection of objects $(c_p)_{p \in \cP}$ in an $\infty$-category $\cC$ indexed by a set $\cP$ and an ultrafilter $\cF$ on $\cP$, we define the ultraproduct to be
\[
\Prod{\cF} c_p = \Colim{U \in \cF} \Prod{i \in U} c_p,
\]
where the colimit is along reverse inclusions. Let $\Sp_{n,p}$ be the $E(n,p)$-local category, let $\Frnp$ be Franke's algebraic category, and let $\cF$ be a non-principal ultrafilter on $\cP$. The main result is a symmmetric monoidal equivalence of $\infty$-categories 
\[
\Prod{\cF} \Sp_{n,p} \simeq \Prod{\cF} \Frnp,
\]
where the ultraproduct is taken in a suitable subcategory of the $\infty$-category of symmetric monoidal $\infty$-categories.

The following consequence of {\L}os's theorem (see \cite[Theorem 4.1.9]{changkeisler}, for instance) is known as the transfer principle: given two collections of objects $(a_p)_{p \in \cP}$ and $(b_p)_{p \in \cP}$ indexed by a set $\cP$ such that the ultraproducts $\Prod{\cF}a_p$ and $\Prod{\cF}b_p$ are isomorphic for every non-principal ultrafilter on $\cP$, then a first order statement is true for all but finitely many elements in $(a_p)$ if and only if it is true for all but finitely many objects in $(b_p)$. To extract results from our equivalences, one would like an analogous transfer principle in the $\infty$-categorical setting. In lieu of such an $\infty$-categorical transfer principle, we provide arguments that establish the transfer principle for specific problems that we are interested in. \\

\paragraph{\bf{Applications}} Our main theorem has a variety of applications. In order to display the utility of the result, we establish the first of these applications in this paper. We indicate further applications below, but leave the details to the sequels.

Recall that a $p$-local Smith--Toda complex of type $k+1$ is a spectrum $V_p(k)$ such that $BP_*(V_p(k)) \cong BP_*/(p,\ldots, v_k)$; their existence and non-existence is a major open problem in stable homotopy theory \cite{nave}. Further, the existence of multiplicative structure on these Smith--Toda complexes is subtle: The mod $p$ Moore spectrum $V_p(0) = S^0/p$ admits the structure of an $\bA_n$-algebra if and only if $n<p$, which is intimately related to Schwede's proof of the rigidity of the stable homotopy category \cite{schwede}. More generally, Nave also showed that, for $p \geq 7$, if $V_p(\frac{p-3}{2})$ exists, then it does not admit the structure of an $\bA_2$-algebra, where $\bA_2$ is the $2$nd operad in Stasheff's hierarchy of associative operads.

A consequence of our equivalence is that locally any given obstruction to the existence of Smith--Toda complexes or multiplicative structure on them vanishes for almost all primes. More precisely, let $\E$ be a height $n$ Morava $E$-theory. For any $0 \leq k < n$ and $p$ large enough (depending on $n$ and $k$), we show that there exists an $E_{n,p}$-local spectrum $V_{n,p}(k)$ such that $(E_{n,p})_*(V_{n,p}(k)) \cong (E_{n,p})_*/(p, \ldots, v_{k})$. Further, since the algebraic analogues of the local Smith--Toda complexes in $\Frnp$ are $\bE_{\infty}$-algebras, part of this multiplicative structure can be transferred through our equivalence. Let $\bA_m$ be the $m$th associative operad.

\begin{cor*}
For any $m > 0$ and $n > k \geq 0$, there exists a prime $p_0$ such that, for all primes $p>p_0$, there exists an $E_{n,p}$-local $\bA_m$-algebra spectrum $V_{n,p}(k)$ such that $(E_{n,p})_*(V_{n,p}(k)) \cong (E_{n,p})_*/(p, \ldots, v_{k})$.
\end{cor*}

The existence of local analogues of the Smith--Toda complexes is a problem that is amenable to classical methods, but the multiplicative structure guaranteed in the corollary above appears to be more difficult to obtain. Besides the equivalence of the main theorem, the key ingredient in the proof of the corollary is the fact that $\bA_m$ is compact as an $\infty$-operad. In fact, the corollary applies to local generalized Moore spectra and any compact $\infty$-operad.

Information also flows through our equivalence from stable homotopy theory to algebraic geometry. Hovey and Sadofsky~\cite{hoveysadofskyinv} have shown that the Picard group of $\Sp_{n,p}$ is isomorphic to $\Z$ for $2p-2 > n^2+n$. Since the equivalence is symmetric monoidal, this implies that the Picard group of $\Frnp$ is isomorphic to $\Z$ for large enough primes generalizing a result of Barnes and Roitzheim~\cite{barnesroitzheim_monoidality} for $n=1$.

Hopkins' chromatic splitting conjecture~\cite{hoveycsc} describes how the sphere spectrum $S^0$ can be assembled from its local pieces $L_{n,p}S^0$. More precisely, the conjecture stipulates that the bottom map in the chromatic pullback square
\[
\xymatrix{L_{n,p}S^0 \ar[r] \ar[d] & L_{K(n,p)}S^0 \ar[d] \\ L_{n-1,p}S^0  \ar[r] & L_{n-1,p}L_{K(n,p)}S^0}
\]
is split and proposes a precise decomposition of the cofiber. The known cases of this conjecture give another example of the asymptotic behavior of chromatic homotopy theory. The chromatic splitting conjecture is known for $n =1$ and all $p$ and also $n=2$ and $p \geq 3$ \cite{ghmr,ghm}. At $n=2$ and $p=2$ the conjecture is false \cite{beaudry,bgh_csc}, which suggests that, at each height, the conjecture may only hold for all but a finite set of primes. However, the current approaches appear to be infeasible at higher heights. Using a $K(n)$-local refinement of the equivalence of the main theorem, we reduce the chromatic splitting conjecture for large enough primes to a purely algebro-geometric question, thereby offering a novel attack on the problem. \\ 

\paragraph{\bf{Outline of the results and proof}} Let $I$ be a set, let $(\cC_i)_{i \in I}$ be a collection of compactly generated $\infty$-categories, and consider an ultrafilter $\cF$ on $I$. We define the ultraproduct to be
\[
\Prod{\cF}^{\omega} \cC_i = \Colim{U \in \cF} \Prod{i \in U} \cC_i,
\]
where the colimit is along reverse inclusions and taken in the $\infty$-category of compactly generated $\infty$-categories. Note that there is a canonical functor
\[
\Prod{I} \cC_i \to \Prod{\cF}^{\omega} \cC_i
\]
which is surjective on compact objects. 

The ultraproduct shares many properties with and can be understood in terms of the input compactly generated $\infty$-categories. For instance, if $c$ and $d$ are compact objects in the ultraproduct and $(c_i)_{i \in I}$ and $(d_i)_{i \in I}$ are preimages of $c$ and $d$ in $\Prod{I}\cC_i$, then
\[
\Map_{\Prod{\cF}^{\omega}\cC_i}(c,d) \simeq \Prod{\cF}\Map_{\cC_i}(c_i,d_i),
\]
where the ultraproduct on the right is taken in the $\infty$-category of spaces. Also, if the categories $\cC_i$ are stable and equipped with a symmetric monoidal structure, then so is the ultraproduct.

For any compactly generated symmetric monoidal $\infty$-category $\cC$, we may implement a familiar version of Whitehead's theorem by localizing with respect to maps $f \colon c \to c'$ such that $[u,f] \colon [u,c] \xrightarrow{\cong} [u,c']$ is an isomorphism for all invertible objects $u$ in $\cC$. If the compactly generated $\infty$-categories $\cC_i$ are symmetric monoidal, we obtain the ``$\Pic$-generated protoproduct"
\[
\Prod{\cF}^{\Pic} \cC_i
\]
by localizing the ultraproduct $\Prod{\cF}^{\omega} \cC_i$ with respect to these equivalences. Informally speaking, this process enforces a suitable finiteness condition on the ultraproduct.

To state the main theorem, we must describe the algebraic approximation to $\Sp_{n,p}$ that we are going to use. In \cite{Franke_exotic}, Franke introduces a category of twisted complexes of $(\E)_0\E$-comodules. Consider the category with objects chain complexes of $(\E)_0\E$-comodules equipped with a chosen isomorphism
\[
X \lra{\cong} (X \otimes_{\pi_0 \E} (\pi_2 \E))[2]
\]
between the complex and the double suspension of the complex tensored with the invertible comodule $\pi_2 \E$ and morphisms maps of complexes compatible with the chosen isomorphism. We establish several key features (at large enough primes) of a symmetric monoidal model structure on this category defined by Hovey \cite{hoveyhcomodules} and Barnes--Roitzheim \cite{barnesroitzheim_monoidality} in which weak equivalences are quasi-isomorphisms of the underlying complexes of comodules. Let $\Frnp$ be the compactly generated symmetric monoidal $\infty$-category associated to this symmetric monoidal model category.

Let $I = \cP$, the set of prime numbers, and let $\cF$ be a non-principal ultrafilter on $\cP$. The main theorem can be stated as follows:
\begin{thm*}
For any $n \geq 0$, there is a symmetric monoidal equivalence of $\Q$-linear stable $\infty$-categories
\[
\Prod{\cF}^{\Pic} \Sp_{n,p} \simeq \Prod{\cF}^{\Pic} \Frnp.
\]
\end{thm*}

The proof of this theorem passes through a descent result on each side. Mathew \cite{mathew_galois}, building on Lurie's homotopical descent theory and classical work of Hopkins and Ravenel \cite{ravbook2}, produces an equivalence
\[
\Sp_{n,p} \simeq \lim \Mod_{E_{n,p}^{\otimes \bullet+1}},
\]
where the limit is taken over the cosimplicial diagram of $\infty$-categories induced by the Amitsur complex $\E^{\otimes \bullet+1}$ of $S^0 \to \E$. 

We prove a similar result on the algebraic side. Let $\mathrm{GrAb}$ be the category of graded abelian groups and let $H \colon \mathrm{GrAb} \to \Sp$ be the Eilenberg--MacLane functor. For a spectrum $X$, we write $X_\star$ for $H\pi_*X$. Since $H$ is lax symmetric monoidal, applying $(-)_{\star}$ to a cosimplicial $\bE_{\infty}$-ring spectrum produces a cosimplicial $\bE_{\infty}$-ring spectrum. We produce an equivalence
\[
\Frnp \simeq \lim \Mod_{(\E^{\otimes \bullet+1})_{\star}}.
\]

Using these equivalences as well as the generic uniform bound on the cohomological dimension of the Morava stabilizer group, we study the analogous descent questions at a non-principal ultrafilter $\cF$. We produce equivalences
\[
\Prod{\cF}^{\Pic} \Sp_{n,p} \simeq \Loc \Pic \lim \Prod{\cF}^{\Pic} \Mod_{\E^{\otimes \bullet+1}}
\]
and
\[
\Prod{\cF}^{\Pic} \Frnp \simeq \Loc \Pic \lim \Prod{\cF}^{\Pic} \Mod_{(\E^{\otimes \bullet+1})_{\star}},
\]
where the right hand side is the localizing subcategory (closure under all colimits) on the invertible objects in the limit.

Thus it is crucial to understand the cosimplicial compactly generated $\infty$-categories
\[
\Prod{\cF}^{\Pic}\Mod_{E_{n,p}^{\otimes \bullet+1}} \text{ and } \Prod{\cF}^{\Pic} \Mod_{(\E^{\otimes \bullet+1})_{\star}}.
\]

Using Morita theory, we show that
\[
\lim \Prod{\cF}^{\Pic}\Mod_{E_{n,p}^{\otimes \bullet+1}} \simeq \lim \Mod_{\Prod{\cF}E_{n,p}^{\otimes \bullet + 1}} 
\]
and
\[
\lim \Prod{\cF}^{\Pic} \Mod_{(\E^{\otimes \bullet+1})_{\star}} \simeq \lim \Mod_{\Prod{\cF} (E_{n,p}^{\otimes \bullet + 1})_{\star}},
\]
where $\Prod{\cF} E_{n,p}^{\otimes k}$ and $\Prod{\cF} (E_{n,p}^{\otimes \bullet + 1})_{\star}$ are the ultraproducts in the $\infty$-category of $\bE_{\infty}$-ring spectra. It suffices to gain a good understanding of the cosimplicial $\bE_{\infty}$-ring spectra
\[
\Prod{\cF} E_{n,p}^{\otimes \bullet+1} \text{ and } \Prod{\cF} (E_{n,p}^{\otimes \bullet + 1})_{\star}.
\]
This is the purpose of the following theorem:

\begin{thm*} \label{intro:formality}
There is an equivalence of cosimplicial $\bE_{\infty}$-ring spectra
\[
\Prod{\cF} E_{n,p}^{\otimes \bullet+1} \simeq \Prod{\cF} (E_{n,p}^{\otimes \bullet + 1})_{\star}.
\]
\end{thm*}

In other words, the cosimplicial $\bE_{\infty}$-ring spectrum $\Prod{\cF} E_{n,p}^{\otimes \bullet+1}$ is formal. Several ingredients go into the proof of this theorem. The arithmetic fracture square is used to reduce the result to the rational case and the case where the tensor product is relative to the $p$-complete sphere spectrum $\hS$. The proof in the rational case is an application of obstruction theory. The proof in the case relative to $\hS$ is more difficult. We develop a functorial theory of weights for spectra equipped with a naive $C_{p-1}$-action and apply it to the cosimplicial $\bE_{\infty}$-ring spectrum $E_{n,p}^{\otimes_{\hS} \bullet+1}$. This gives a weight decomposition of cosimplicial spectra of the form
\[
E_{n,p}^{\otimes_{\hS} \bullet+1} \simeq \Oplus{\chi \in \hom(C_{p-1},\Z_{p}^{\times})} \big ( E_{n,p}^{\otimes_{\hS} \bullet+1} \big )_{\chi},
\]
indexed by characters of $C_{p-1}$. This decomposition reflects the fact that the non-trivial $k$-invariants of $E_{n,p}$ grow sparser as $p$ increases. Applying the ultraproduct over a non-principal ultrafilter, we find that the cosimplicial spectrum is formal.\\

\paragraph{\bf{Acknowledgements}} We would like to thank Mark Behrens, David Gepner, Paul Goerss, Rune Haugseng, Lars Hesselholt, Mike Hopkins, Irakli Patchkoria, and the Homotopy Theory chat room for useful discussions and would all like to thank the MPIM for its hospitality. We are grateful to the referees for many helpful suggestions and corrections. The first author was supported by the DNRF92 and the European Union’s Horizon 2020 research and innovation programme under the Marie Sklodowska-Curie grant agreement No.~751794. The second author is supported by the Alon fellowship and ISF 1588/18. The third author was supported by SFB 1085 \emph{Higher Invariants} funded by the DFG and NSF grant No.~ DMS-1906236. \\

\paragraph{\bf{Conventions}}
Throughout this paper we will employ the following conventions:
\begin{itemize}
	\item We write $\Map$ for mapping spaces in $\infty$-categories and $\Hom$ for mapping spectra in stable $\infty$-categories.
	\item The $\infty$-category of commutative monoids in a symmetric monoidal $\infty$-category $\cC$ will be denoted by $\CAlg(\cC)$ and we refer to its objects as commutative algebras in $\cC$. For $\cC = \Sp$ equipped with its natural symmetric monoidal structure, we usually say $\mathbb{E}_{\infty}$-ring spectrum or $\mathbb{E}_{\infty}$-ring instead of commutative algebra.
	\item A symmetric monoidal presentable $\infty$-category $\cC = (\cC,\otimes)$ is called presentably symmetric monoidal if the monoidal structure $\otimes$ preserves colimits separately in each variable.
	\item By symmetric monoidal compactly generated $\infty$-category we always mean a compactly generated $\infty$-category equipped with a symmetric monoidal structure $\otimes$ such that $\otimes$ commutes with all colimits and restricts to a symmetric monoidal structure on the full subcategory $\cC^{\omega}$ of compact objects in $\cC$. In particular, the unit object is assumed to compact.
	\item If $\cC$ is a presentably symmetric monoidal stable $\infty$-category and $A$ is a commutative algebra in $\cC$, then $\Mod_A(\cC)$ denotes the stable $\infty$-category of modules over $A$ in $\cC$. In the case $\cC = \Sp$, we will write $\Mod_A$ instead of $\Mod_A(\cC)$ for simplicity. Similarly, we write $\CAlg_A(\cC)$ for the $\infty$-category of commutative $A$-algebras in $\cC$ and omit the $\infty$-category $\cC$ when it is clear from context and in particular whenever $\cC = \Sp$.
	\item  If $\cC$ is a presentable stable $\infty$-category and $\cS$ is a collection of objects in $\cC$, then we will write $\Loc_{\cC}(\cS)$ for the smallest localizing (stable) subcategory of $\cC$ containing $\cS$.
	\item The totalization of a cosimplicial diagram of compactly generated $\infty$-categories $\cE^{\bullet}$, denoted $\Tot(\cE^{\bullet})$, will always refer to the limit taken in the $\infty$-category of compactly generated $\infty$-categories.
	\item Let $G$ be a finite group. The category of naive $G$-spectra is by definition the functor category $\Fun(BG,\Sp)$, where the classifying space $BG$ of $G$ is considered as an $\infty$-groupoid. 
\end{itemize}

\section{Recollections}

\subsection{Ultrafilters}\label{ssec:ultrafilters}

In this subsection we explain the basics of ultrafilters and ultraproducts. Our goal is to give the background necessary for the paper and a brief introduction for the working homotopy theorist. More details may be found in many textbooks, e.g., \cite{bellslomson} or \cite{changkeisler}; we will primarily follow \cite{schoutensbook}.

The basic definition is the following:
\begin{defn}
An ultrafilter $\cF$ on a set $I$ is a nonempty collection of subsets of $I$ satisfying:
	\begin{enumerate}
		\item The empty set is not an element of $\cF$.
		\item If $A \subseteq B \subseteq I$ and $A \in \cF$, then $B \in \cF$.
		\item If $A,B \in \cF$, then $A \cap B \in \cF$.
		\item If $A \subseteq I$, then either $A \in \cF$ or $I \setminus A \in \cF$.
	\end{enumerate}
\end{defn}
A filter is a subset of the power set of $I$ satisfying all but the last axiom. A filter may be completed to an ultrafilter in many ways, assuming the axiom of choice.  

\begin{lem} \label{lem:completetoultrafilter}
If $\cF$ is a filter on $I$, then there exists an ultrafilter $\overline{\cF}$ on $I$ containing $\cF$. 
\end{lem}
\begin{proof}
This is an application of Zorn's lemma. The union of a chain of filters is a filter and a maximal filter is an ultrafilter.
\end{proof}

The following lemma is useful:

\begin{lem}\label{lem:finitepartition}
Suppose $I$ is a set and $\cF$ is an ultrafilter on $I$. If $I_1 \sqcup I_2 \sqcup \ldots \sqcup I_n = I$ is a finite partition of $I$, then there exists exactly one $1\le i \le n$ such that $I_i \in \cF$. 
\end{lem}
\begin{proof}
If $I_i \notin \cF$ for all $i$, then $\varnothing = \bigcap_{i=1}^n (I \setminus I_i) \in \cF$, a contradiction. If there exist $i \ne j$ such that $I_i,I_j \in \cF$, then $\varnothing = I_i \cap I_j \in \cF$. The claim follows. 
\end{proof}

\begin{ex}
Given an element $x \in I$, the set of subsets of $I$ containing $x$ is an ultrafilter denoted $\cF_x$. The ultrafilters of this form are called principal ultrafilters. Because of this, the ultrafilters on a set may be considered as generalized elements of the set. 
\end{ex}

\begin{lem} \label{finiteprincipal}
An ultrafilter $\cF$ that contains a finite set is principal.
\end{lem}
\begin{proof}
We may partition $I$ into the points of the finite set and the complement of the finite set. Since $\cF$ contains a finite set, it does not contain the complement, so one of those points must be in $\cF$ by the previous lemma. 
\end{proof}

It is reasonably easy to construct non-principal filters. For instance, the collection of cofinite subsets of an infinite set $I$ is a filter, known as the Frechet filter $\cF_{\infty}$, but it is not an ultrafilter. By \cite{B}, the existence of a non-principal ultrafilter is independent of $\mathrm{ZF}$ so it is impossible to explicitly describe non-principal ultrafilters.

\begin{lem}
An ultrafilter $\cF$ is non-principal if and only if it contains $\cF_{\infty}$.
\end{lem}
\begin{proof}
This follows immediately from \Cref{finiteprincipal}.
\end{proof}

\begin{lem}\label{lem:ultrafiltersupply} 
If $A \subseteq I$ is infinite, then there exists a non-principal ultrafilter $\cF$ on $I$ such that $A \in \cF$. 
\end{lem}
\begin{proof}
Consider the collection of subsets of $I$ that contain all but a finite number of elements in $A$. This is a filter that contains $\cF_{\infty}$ and by \Cref{lem:completetoultrafilter} it can be completed to an ultrafilter.
\end{proof}

\subsection{Set-theoretic ultraproducts}

For the rest of this section, $I$ will denote some indexing set, for example the set of prime numbers $\cP = \{2,3,5,7,\ldots\}$.

\begin{defn}
Let $(A_i)_{i \in I}$ be a collection of nonempty sets and let $\cF$ be an ultrafilter on $I$. The ultraproduct of the sets $(A_i)_{i \in I}$ over the ultrafilter $\cF$ is the quotient of the product $\Prod{i \in I} A_i$ defined by the relation
\[
(a_i)_{i \in I} \sim (b_i)_{i \in I} \text{ if and only if } \{i \in I| a_i = b_i\} \in \cF.
\]
We will denote the quotient $(\Prod{i \in I} A_i)_{/\sim}$ by $\Prod{\cF} A_i$. 
\end{defn}

By definition, there is a quotient map from the product to the ultraproduct
\[
\Prod{i \in I} A_i \rightarrow \Prod{\cF}A_i.
\]
We will denote the image of $(a_i) \in \Prod{i \in I}A_i$ in the ultraproduct by $[a_i]$. If the $A_i$'s are all the same set $A$, then we will refer to an element $[a_i] \in \Prod{\cF}A$ as constant if it is the image of $(a)_{i \in I}$ for some $a \in A$.

The ultraproduct of a collection of sets of bounded finite cardinality is particularly simple.

\begin{ex}\label{ex:finite}
Let $X$ be a finite set and let $\cF$ be an ultrafilter on $I$. There is an isomorphism
\[
\Prod{\cF} X \cong X,
\]
where the ultraproduct is taken over the constant collection $A_i = X$: Let $(x_i)_{i \in I} \in \Prod{i \in I}X$. We may may produce a finite partition of $I$ indexed by the elements of $X$ by setting
\[
I_{x} = \{i \in I \colon x_i = x\}.
\]
By \cref{lem:finitepartition}, only one of these sets can be in $\cF$, thus $[x_i]$ is constant in $\Prod{\cF}X$.
\end{ex}

Ultraproducts preserve many algebraic structures, for instance the structure of being an abelian group, commutative ring, field, and so on. These are all special cases of a result due to {\L}o{\'s}, which is often referred to as the fundamental theorem of ultraproducts. 

\begin{thm}[{\L}o{\'s}]\label{thm:los} (\cite[Theorem 4.1.9]{changkeisler})
Let $\cL$ be a language and let $\cF$ be an ultrafilter on a set $I$. Suppose $(X_i)_{i \in I}$ is  a collection of $\cL$-structures with ultraproduct $X=\Prod{\cF}X_i$. Let
\[
(x_{i,1})_{i \in I}, \ldots, (x_{i,n})_{i \in I} \in \Prod{i \in I}X_i
\]
be $n$ elements in the product. Then for any formula $\phi$ in $n$ unbounded variables, $\phi([x_{i,1}], \ldots, [x_{i,n}])$ holds in $X$ if and only if 
\[
\{i\in I\colon \phi(x_{i,1}, \ldots, x_{i,n}) \text{ holds in } X_i)\} \in \cF.
\]
\end{thm}

Informally speaking, the content of this theorem can be summarized by saying that a first order statement holds for the ultraproduct if and only if it holds on a set in the ultrafilter.  

\begin{ex}\label{ex:pseudofinite}
Let $A_p = \F_p$ be the finite field of order $p$. Given an ultrafilter $\cF$ on the set of primes $\cP$, we may form the ultraproduct 
\[
\F_{\cF} = \Prod{\cF} \F_p. 
\]
By {\L}o{\'s}'s theorem, $\F_{\cF}$ is a field which behaves much like finite fields. For instance, the absolute Galois groups of $\F_{\cF}$ is $\hat{\Z}$. The fields obtained in this way are known as pseudo-finite fields \cite{pseudofinite}.

If $\cF = \cF_p$ is principal then the ultraproduct is just $\F_p$. Otherwise the ultraproduct is a characteristic $0$ field. The reason for this is because multiplication by $n$ on $\F_p$ is an isomorphism for all but finitely many $p$. Since $\cF$ is non-principal this means that it is an isomorphism on a set in the ultrafilter. This implies that multiplication by $n$ induces an automorphism of the ultraproduct. 

The properties of these fields depend on the ultrafilter chosen. For example, by {\L}o{\'s}'s theorem $-1$ is a square in $\F_{\cF}$ if and only if $\cF$ contains the set of primes that are congruent to $1$ mod $4$.
\end{ex}

\begin{ex}\label{ex:fadics}
Consider the ultraproduct $\Z_{\cF} = \Prod{\cF}\Z_p$ of the $p$-adic integers with respect to a non-principal ultrafilter $\cF$ on $\cP$. By an argument similar to the one used in \Cref{ex:pseudofinite}, $\Z_{\cF}$ is a commutative algebra over $\Q$.
\end{ex}

\begin{ex}
We may let $A_i = \N$ and take the ultraproduct 
\[
\N^{\cF} = \Prod{\cF} \N. 
\]
For $\cF$ non-principal this is a semiring of cardinality $2^{\aleph_0}$. If the sequence $(n_i) \in \Prod{i \in I}\N$ is bounded on a set in the ultrafilter then $[n_i]$ is constant.  
\end{ex}

\begin{ex}\label{ex:zf}
We will let $\Z^{\cF}$ be the ultraproduct $\Prod{\cF}\Z$. The canonical maps $\Z \rightarrow \Z/n$, induce a surjection
\[
\Z^{\cF} \twoheadrightarrow \hat{\Z}.
\]
The kernel of this map is an uncountable rational vector space. 
\end{ex}

As we will show in the next example, the ultraproduct does not necessarily send polynomial rings to polynomial rings. 

\begin{ex}
Let $A_p = \F_p[x]$ and let $\cF$ be an ultrafilter on $\cP$. Consider the ultraproduct
\[
\Prod{\cF} (\F_p[x]).
\]
If $\cF$ is principal, then this is the polynomial ring in one variable over $\F_p$. It is generated as a module over $\F_p$ by the monomials $x^n$ where $n \in \N = \N^{\cF}$. For $\cF$ a non-principal ultrafilter, the resulting ring is very large and more difficult to describe (with generators and relations). For instance, it contains the equivalence class of
\[
\left ( \sum_{i=0}^{p}x^i\right )_{p \in \cP}
\]
in which the degree and number of monomials involved in each term both grow to infinity.
\end{ex}

This example represents a weakness of ultraproducts. They do not preserve gradings and send unbounded phenomena (such as sequence of polynomials with unbounded degree) to rather exotic objects. There is a solution to this problem, known as the protoproduct \cite[Chapter 9]{schoutensbook}, whose categorical analogue plays an important role throughout this paper. The protoproduct takes in a collection of filtered objects and produces a subset of the ultraproduct. The next examples display the behavior of the protoproduct for two filtrations on polynomial rings.

\begin{ex} \label{ex:protopoly}
We will use the notation of the previous example. Consider the collection $(\F_p[x], \F_p[x]^{\leq k})_{p \in \cP}$ of polynomial rings equipped with the degree filtration, so $\F_p[x]^{\leq k}$ is the subset of $\F_p[x]$ of polynomials of degree $\leq k$. The protoproduct is defined as a quotient of the ``bounded product" 
\[
\Prod{\cP}^{\flat} (\F_p[x], \F_p[x]^{\leq k}) = \Colim{k} \Prod{\cP} \F_p[x]^{\leq k}
\]
by the same equivalence relation as the ultraproduct. Since colimits and quotients commute, this quotient is the same as the colimit 
\[
\Prod{\cF}^{\flat} (\F_p[x], \F_p[x]^{\leq k}) = \Colim{k} \Prod{\cF} \F_p[x]^{\leq k} \cong (\Colim{k} \Prod{\cP} \F_p[x]^{\leq k})_{/\sim}.
\] 
The protoproduct along the degree filtration sends polynomial rings to polynomial rings:
\[
\Prod{\cF}^{\flat} (\F_p[x], \F_p[x]^{\leq k}) \cong \F_{\cF}[x].
\] 
\end{ex}

\begin{ex}
There is another natural filtration that we may put on $\F_p[x]$. Let $\F_p[x]^{\leq k\text{-mon}}$ be the subset of polynomials built out of less than or equal to $k$ monomials. For $\cF$ a non-principal ultrafilter, the protoproduct
\[
\Prod{\cF}^{\flat}(\F_p[x],\F_p[x]^{\leq k\text{-mon}})
\]
is the monoid-algebra over $\F_{\cF}$ on $\N^{\cF}$. It has an $\F_{\cF}$-basis given by $x^{[n_p]}$ where $[n_p] \in \N^{\cF}$. Thus this is an ``ultra" polynomial ring. It is not graded by the natural numbers but by $\N^{\cF}$.
\end{ex}

\begin{rem}
Let $I$ be a set viewed as a discrete space, and denote by $\beta I$ the set of ultrafilters on $I$, with a natural map $I \to \beta I$ given by sending an element $x \in I$ to the principal ultrafilter $\cF_x$ on $I$. For $A \subseteq I$, write $\hat{A}$ for the family of ultrafilters on $I$ containing $A$. The sets $\hat{A}$ for all $A \subseteq I$ form a basis of open subsets for the topology on $\beta I$, the Stone topology. This construction makes $\beta I$ into a compact Hausdorff space, and $I \to \beta I$ can be identified with the  Stone--\u{C}ech compactification of $I$. 

In these terms, the ultraproduct admits a geometric interpretation in the following sense~\cite{schoutensbook}: Consider a category $\cC$ closed under products and filtered colimits and let $(c_i)_{i \in I} \in \cC^I$ be a collection of objects in $\cC$ indexed by the set $I$. A sheaf on the discrete space $I$ with values in a category $\cC$ is given by a functor 
\[
\Prod{-}c_i\colon (A \subseteq I) \mapsto \Prod{i \in A}c_i.
\]
For a given ultrafilter $\cF$ on $I$, the two inclusions $\iota\colon I \to \beta I$ and $\{\cF\} \to \beta I$ induce geometric morphisms $(\iota_*,\iota^*)$ and $(\cF_*,\cF^*)$ between the corresponding categories of sheaves. The composite 
\[
\xymatrix{\Sh(I) \ar[r]^-{\iota_*} & \Sh(\beta I) \ar[r]^-{\cF^*} & \cC}
\]
can then be identified with the ultraproduct functor $\Prod{\cF}$. In other words, the ultraproduct $\Prod{\cF}c_i$ is equivalent to the stalk at $\cF$ of the sheaf $\iota_*\cE$, where $\cE \in \Sh(I)$ corresponds to the collection $(c_i)_{i \in I}$.
\end{rem}

\section{Ultraproducts}\label{sec:ultraproducts}

\subsection{Ultraproducts in $\infty$-categories}

In this section, we define the ultraproduct of a collection of objects in an $\infty$-category that admits filtered colimits and products. In particular, we study the special case of the $\infty$-category $\Cat$ of $\infty$-categories, which gives rise to the ultraproduct of $\infty$-categories. An independent account of some of the results in this section can be found in \cite[E.3.3.4]{sag}.

Given a collection of nonempty sets $X_i$ and an ultrafilter $\cF$ on $I$, there is a canonical isomorphism
\[
\Prod{\cF} X_i = \big ( \Prod{I} X_i \big )/\sim \lra{\cong} \Colim{U \in \cF} \Prod{i \in U} X_i
\]
induced by the projections, where the colimit is taken along reverse inclusions. This motivates the following definition:

\begin{defn}
Let $\cC$ be an $\infty$-category that admits products and filtered colimits and let $(c_i)_{i \in I}$ be a collection of objects in $\cC$. For an ultrafilter $\cF$ on $I$ we define the ultraproduct of $(c_i)_{i \in I}$ to be the object
\[
\Prod{\cF} c_i = \Colim{U \in \cF} \Prod{i \in U} c_i,
\]
where the colimit is taken along reverse inclusions.
\end{defn}

\begin{rem}
Let $U \in \cF$. Then we obtain an ultrafilter on the set $U$, $\cF_U$, by intersecting the elements of $\cF$ with $U$. Let $\cC$ be an $\infty$-category with products and filtered colimits and let $(c_i)_{i \in I}$ be a collection of objects in $\cC$. There is a canonical equivalence
\[
\Prod{\cF} c_i \simeq \Prod{\cF_U} c_i.
\]
Thus for any set $U \notin \cF$, we may ``throw out" the objects supported on $U$. When $\cF$ is non-principal, \cref{finiteprincipal} implies that we may throw out any finite number of objects in the ultraproduct. 
\end{rem}

\begin{ex}
If $\cF = \cF_j$ is a principal ultrafilter for some $j \in I$, then
\[
\Prod{\cF_j} c_i \simeq c_j.
\]
\end{ex}

\begin{rem} \label{rem:finitelims}
If $\cC$ is compactly generated then filtered colimits in $\cC$ are left exact. This follows from the fact that $\cC = \Ind(\cC^{\omega})$ and the $\Ind$-category consists of finite limit preserving presheaves on $\cC^{\omega}$ \cite[5.3.5.4]{htt}. Thus the statement reduces to the corresponding fact for the $\infty$-category of spaces, which follows from \cite[5.3.3.3]{htt}. Therefore, when $\cC$ is compactly generated, ultraproducts in $\cC$ commute with finite limits.
\end{rem}

By the definition of ultraproduct, there is a canonical map 
\[
[-]_{\cF}\colon \Prod{I}c_i \longrightarrow \Prod{\cF}c_i.
\]
When $\cF$ is clear from context, we will abbreviate this to $[-]$.

We now consider the case that $\cC$ is $\Cat$, which is bicomplete \cite[Sections 3.3.3, 3.3.4]{htt} and compactly generated \cite[Section 5.5]{htt}. Given a collection of $\infty$-categories $(\cC_i)_{i \in I}$ and objects $c_i \in \cC_i$, we will write $(c_i)_{i \in I}$ for the corresponding object in the product $\Prod{I}\cC_i$ and $[c_i]_{\cF}$ for the object $[(c_i)_{i \in I}]_{\cF}$ in the ultraproduct $\Prod{\cF} \cC_i$. When the indexing set $I$ is clear we will denote these objects by $(c_i)$ and $[c_i]$. If $c \in \Prod{\cF} \cC_i$, we will say that $c$ is represented by $(c_i)\in\Prod{I}\cC_i$ if $[c_i] \simeq c$.

Let $\Top$ be the $\infty$-category of $\infty$-groupoids. We will refer to the objects of $\Top$ as spaces. The inclusion functor from spaces to $\infty$-categories
\[
\Top \to  \Cat
\]
has both a right adjoint, which is the core functor $\cC \mapsto \cC^{\simeq}$, and a left adjoint, which is the groupoidification functor $\cC \mapsto \cC[\cC^{-1}]$. The notation $\cC[\cC^{-1}]$ is justified by considering the groupoidification as ``inverting all morphisms  in $\cC$''. 
One can also invert only some of the morphisms: Given a subcategory $W \subset \cC$, we can define $\cC[W^{-1}]$ to be the pushout of the diagram
\[
\xymatrix{
W \ar[r]\ar[d] & \cC \\
W[W^{-1}].
}
\]
It is easy to see that for every $\infty$-category $\cD$ we get that $\Fun(\cC[W^{-1}],\cD) = \cD^{\cC[W^{-1}]}$ is the full subcategory of $\Fun(\cC,\cD)$ consisting of functors that send a morphism in $W$ to an equivalence in $\cD$. 

\begin{lem}\label{l:local}
Let $\cC$ and $\cD$ be $\infty$-categories and let $W \subset \cC$. There is an equivalence of $\infty$-categories
\[
(\cC \times \cD)[(W \times \cD^{\simeq})^{-1}] \lra{\simeq} \cC[W^{-1}] \times \cD.
\]
\end{lem}

\begin{proof}
We shall first prove this for $W= \cC$. In this case we are required to prove the following diagram is a pushout diagram  
\[
\xymatrix{
W\times \cD^{\simeq} \ar[r] \ar[d]&  W \times \cD \ar[d]\\
W[W^{-1}] \times \cD^{\simeq}\ar[r] & W[W^{-1}] \times \cD.
}
\]
It is enough to show that for every $\infty$-category $\mathcal{T}$ the following diagram is a pullback diagram:
\[
\xymatrix{
\Fun(\cD,\mathcal{T}^{W[W^{-1}]})\ar[r]\ar[d] &  \Fun(\cD,\mathcal{T}^W) \ar[d]\\
\Fun( \cD^{\simeq} , \mathcal{T}^{W[W^{-1}]} )\ar[r]&  \Fun(\cD^{\simeq} ,\mathcal{T}^W).
}
\]
This follows from the fact that $\mathcal{T}^{W[W^{-1}]} \to \mathcal{T}^W$ is a fully faithful functor.
 
Now for a general $W\subset \cC$ consider the diagram 
\[
\xymatrix{
W\times \cD^{\simeq} \ar[r]\ar[d]&  W \times \cD \ar[d]\ar[r] & \cC \times \cD \ar[d]\\
W[W^{-1}] \times \cD^{\simeq}\ar[r] & W[W^{-1}] \times \cD\ar[r] & \cC[W^{-1}] \times \cD. 
}
\]
  
The right square is a pushout square since $- \times \cD$ preserves colimits and, since the left square is a pushout square, the outer square is also a pushout square.
\end{proof}

\begin{cor}
Let $\cC$ and $\cD$ be $\infty$-categories and let $W \subset \cC$ and $Z \subset \cD$ be subcategories that contain the core. There is an equivalence of $\infty$-categories
\[
(\cC \times \cD)[(W \times Z)^{-1}] \lra{\simeq} \cC[W^{-1}] \times \cD[Z^{-1}].
\]
\end{cor}
\begin{proof}
This follows from the universal property and applying \cref{l:local} twice.
\end{proof}

We will say that an $\infty$-category is contractible if its underlying simplicial set (or the $\infty$-groupoid $\cC[\cC^{-1}]$) is contractible. For instance, if $\cC$ has an initial or terminal object, then it is contractible. 

\begin{prop} \label{prop:ultrainvert}
Let $(\cC_i)_{i \in I}$ be a collection of  $\infty$-categories such that $\Prod{I}\cC_i$ is contractible. There is an equivalence
\[
\Prod{\cF} \cC_i \simeq (\Prod{I} \cC_i)[W_{\cF}^{-1}],
\]
where $W_{\cF}$ is the subcategory supported on the morphisms that are an equivalence on a set in the ultrafilter.
\end{prop}
\begin{proof}
Let 
\[
W_{U} = \Prod{i \in U} \cC_{i}^{\simeq} \times \Prod{i \notin U}\cC_i.
\]
Note that 
\[
W_{\cF} = \Colim{U \in \cF} W_U.
\]
For every $U \in \cF$, there is a pushout diagram
\[
\xymatrix{W_{U} \ar[r] \ar[d] & \Prod{I} \cC_i \ar[d] \\ W_{U}[W_{U}^{-1}] \ar[r] & (\Prod{I} \cC_i)[W_{U}^{-1}].}
\]
Since $W \mapsto W[W^{-1}]$ is a left adjoint, taking the colimit over $U \in \cF$ gives an equivalence
\begin{align*}
(\Prod{I} \cC_i)[W_{\cF}^{-1}] &\simeq \Colim{U\in \cF} \big ( \Prod{ I} \cC_i \big [W_{U}^{-1}]\big ] \big ).
\end{align*}
By \cref{l:local} we get 
\[
(\Prod{I} \cC_i)[W_{\cF}^{-1}]  \simeq \Colim{U\in \cF} \big ( \Prod{i \in U} \cC_i \times \Prod{i \notin U}\cC_i[(\Prod{i \notin U}\cC_i)^{-1}] \big ). 
\]
Now since $\Prod{I}\cC_i$ is contractible none of the $\cC_i$ are empty so we get that  $\Prod{i \notin U}\cC_i$ is  a retract of $\Prod{I}\cC_i$ and thus 
$\Prod{i \notin U}\cC_i$  is contractible. 
We get that 
\begin{align*}
(\Prod{I} \cC_i)[W_{\cF}^{-1}]  
&\simeq \Colim{U\in \cF} \big ( \Prod{i \in U} \cC_i \times * \big ) \\
&\simeq \Prod{\cF}\cC_i. \qedhere
\end{align*}
\end{proof}

\begin{rem} Note that the $\infty$-category $\Prod{I}\cC_i$ is contractible if it has a terminal (respectively initial) object. This happens if each $\cC_i$ has a terminal (respectively initial) object.
\end{rem}

Given a model category in which filtered colimits of weak equivalences are weak equivalences, homotopy filtered colimits can be computed $1$-categorically. Since products of weak equivalences between fibrant objects are always weak equivalences, this implies that ultraproducts of fibrant objects can be computed $1$-categorically. Model categories with the property that filtered colimits of weak equivalences are weak equivalences include the category of simplicial sets with the Quillen or Joyal model structure and the category of chain complexes of modules over a ring with the standard model structure.

\begin{lem} \label{lem:setwise}
Given quasicategories $(S_{i,\bullet})_{i \in I}$, let $\Prod{\cF} S_{i,\bullet}$ be the ultraproduct taken in the $1$-category of simplicial sets. This simplicial set is a quasicategory modeling the ultraproduct of the $\infty$-categories corresponding to the collection $(S_{i,\bullet})_{i \in I}$. Note that we have an isomorphism of sets
\[
(\Prod{\cF} S_{i,\bullet})_n \cong \Prod{\cF} S_{i,n}.
\]
\end{lem}

\begin{lem}\label{cor:distributivity}
Filtered colimits distribute over infinite products in the $\infty$-category $\Top$. That is, let $I$ be a set and for each $i \in I$, let $J_i$ be a filtered category and 
\[
F_i \colon J_i \lra{} \Top
\]
be a functor. Let 
\[
F \colon \Prod{I}J_i \lra{\Prod{I}F_i} \Top^{I} \lra{\prod} \Top
\]
be the composite. Then there is a canonical equivalence
\[
\Colim{\Prod{I}J_i}\, F \simeq \Prod{I}\Colim{J_i}\, F_i.
\]
\end{lem}
\begin{proof}
This is true in $\Set$ by \cite{alr}. Therefore, $1$-categorically, it is true in $\sSet$ and since the Quillen model structure on $\sSet$ satisfies the conditions of the paragraph above, it is true homotopically since we may apply fibrant replacement.
\end{proof}

\begin{lem}\label{lem:opultracat}
Let $(\cC_i)_{i \in I}$ be a collection of $\infty$-categories and let $\cF$ be an ultrafilter on $I$. There is an equivalence of $\infty$-categories
\[
(\Prod{\cF}\cC_i)^{\op} \simeq \Prod{\cF}(\cC_i^{\op}).
\]
\end{lem}

\subsection{Ultraproducts of $\infty$-categories}
We study categorical properties of ultraproducts of $\infty$-categories. We begin with a key proposition computing $\Hom$-spaces in ultraproducts of $\infty$-categories. The proof of this proposition depends on a technical lemma due to Rozenblyum.

\begin{prop} \label{lem:rozenblyum}
Let $(\cC_i)_{i \in I}$ be a collection of $\infty$-categories. For two objects $c,d \in \Prod{\cF}\cC_i$ represented by $(c_i)$ and $(d_i)$, there is a natural equivalence
\[
\Map_{\Prod{\cF}\cC_i}(c,d) \simeq \Prod{\cF}\Map(c_i,d_i).
\]
\end{prop}
\begin{proof}
For this proof we will work with quasicategories. Given a quasicategory $\cC$, we will write $\TwArr(\cC)$ for the twisted arrow category of \cite[5.2.1.1]{ha}. By \cite[5.2.1.3]{ha}, there is a right fibration
\[
\TwArr(\cC) \to \cC \times \cC^{\op},
\]
where the fiber over $(c,d)$ is the mapping space $\Map(c,d)$.

Now given a collection of quasicategories $(\cC_i)_{i \in I}$, there is a commutative square of quasicategories
\[
\xymatrix{\Prod{\cF}\TwArr(\cC_i) \ar[r]^-{\cong} \ar[d] & \TwArr(\Prod{\cF}\cC_i) \ar[d] \\ \Prod{\cF}(\cC_i \times \cC_{i}^{\op})  \ar[r]^-{\cong} & (\Prod{\cF}\cC_i \times \Prod{\cF}\cC_{i}^{\op}),}
\]
where the vertical arrows are right fibrations. The top horizontal arrow is an isomorphism by \cref{lem:setwise} and the fact that $n$-simplices in $\TwArr(\Prod{\cF}\cC_i)$ are $2n+1$-simplices in $\Prod{\cF}\cC_i$. Given $c,d \in \Prod{\cF}\cC_i$ represented by $(c_i)$ and $(d_i)$, there is a pullback square
\[
\xymatrix{\Map_{\Prod{\cF}\cC_i}(c,d) \ar[r] \ar[d] & \TwArr(\Prod{\cF}\cC_i)  \ar[d] \\ 
\ast \ar[r]^-{(c,d)} & \Prod{\cF}\cC_i \times \Prod{\cF}\cC_{i}^{\op}.}
\]
Further, since ultraproducts commute with pullbacks in simplicial sets, there is a pullback square 
\[
\xymatrix{\Prod{\cF}\Map(c_i,d_i) \ar[r] \ar[d] & \Prod{\cF}\TwArr(\cC_i) \ar[d] \\ \ast \ar[r]^-{[(c_i,d_i)]} & \Prod{\cF}(\cC_i \times \cC_{i}^{\op}).}
\]
It follows that there is an isomorphism of Kan complexes
\[
\Map_{\Prod{\cF}\cC_i}(c,d) \cong \Prod{\cF}\Map(c_i,d_i).
\]
\end{proof}

\cref{lem:rozenblyum} has the following consequences:

\begin{cor}\label{cor:homotopymaps}
For $c,d \in \Prod{\cF}\cC_i$ represented by $(c_i)$ and $(d_i)$, there is an isomorphism
\[
[c,d] \cong \Prod{\cF}[c_i, d_i],
\]
where the ultraproduct is computed in the category of sets.
\end{cor}
\begin{proof}
This is due to the fact that $\pi_0$ commutes with filtered colimits as  $S^0$ is compact.
\end{proof}

\begin{cor} \label{fullyfaithful}
The ultraproduct of fully faithful functors between $\infty$-categories is fully faithful.
\end{cor}

There is also a stable version of \Cref{lem:rozenblyum}. Let $\Sp$ be the $\infty$-category of spectra.

\begin{cor}\label{lem:stablerozenblyum}
With notation as in \Cref{lem:rozenblyum}, if the $\infty$-categories $\cC_i$ are stable, then so is the ultraproduct $\Prod{\cF}\cC_i$ and there is an equivalence of mapping spectra
\[
\Hom_{\Prod{\cF}\cC_i}(c,d) \simeq \Prod{\cF}\Hom(c_i,d_i),
\]
where the ultraproduct on the right side is computed in $\Sp$.
\end{cor}
\begin{proof}
Products of stable $\infty$-categories are stable by \cite[Theorem 1.1.4.3]{ha} and filtered colimits of stable $\infty$-categories are stable by \cite[Theorem 1.1.4.6]{ha}.
\end{proof}

\begin{lem} \label{prop:ultracolims}
Let $(\cC_i)_{i \in I}$ be a collection of $\infty$-categories. The ultraproduct $\Prod{\cF} \cC_i$ enjoys the following properties:
\begin{enumerate}
\item For a finite simplicial set $K$, there is a canonical equivalence of $\infty$-categories
\[
\Fun(K, \Prod{\cF} \cC_i) \simeq \Prod{\cF} \Fun(K, \cC_i).
\]
\item Let $K$ be a finite simplicial set, assume we are given for all $i \in I $, $\rho_i\colon K \to \cC_i$. There is a canonical equivalence
\[
 (\Prod{\cF}  \cC_i )_{/ \Prod{\cF} \rho_i} \simeq \Prod{\cF} \cC_{i/ \rho_i},
\] 
where $\Prod{\cF} \rho_i\colon K \rightarrow \Prod{\cF}  \cC_i$.
\item If each $\cC_i$ has finite (co)limits then $\Prod{\cF} \cC_i$ has finite (co)limits.
\item For any $U \in \cF$, the canonical map $\Prod{U} \cC_i \rightarrow \Prod{\cF} \cC_i$ preserves all finite (co)limits that exist.
\item If the functors $f_i\colon\cC_i \to \cD_i$ preserve finite (co)limits then the ultraproduct
\[
\Prod{\cF} f_i\colon \Prod{\cF} \cC_i \longrightarrow \Prod{\cF} \cD_i
\]
preserves all finite (co)limits that exist.
\item Let $f_i \colon \cC_i \rightarrow \cD_i$ be a collection of (co)Cartesian fibrations. Then $\Prod{\cF} f_i$ is a (co)Cartesian fibration.
\end{enumerate}
\end{lem}
\begin{proof}
Part (1) follows from the compactness of $K$ in $\sSet$. 
Recall that $\Cat$ is compactly generated so filtered colimits commute with pullbacks. As a composition of products and filtered colimits, ultraproducts commute with pullbacks. Now Part (2) follows from the pullback square:
\[
\xymatrix{
\cC_{/\rho} \ar[r]\ar[d] & \Fun(K^{\lhd},\cC) \ar[d] \\
\ast   \ar[r]_-{\rho}  & \Fun(K ,\cC). 
}
\]
For Parts (3), (4), and (5), in view of (2) and \Cref{lem:opultracat}, it is enough to show that $\Prod{\cF} \cC_i$ has an initial object, the map from the product to the ultraproduct preserves the initial object, and that the initial object is preserved by $\Prod{\cF} f_i$.
For each $i \in I$, let $\emptyset_i \in \cC_i$ be a choice of initial object. We will show that $[\emptyset_i]$ is initial in 
$\Prod{\cF}\cC_i$.
Indeed if $t_i \in \cC_i$, \cref{lem:rozenblyum} gives equivalences
\[
\Map_{\Prod{\cF}\cC_i} ([\emptyset_i]  , t) \simeq \Prod{\cF} \Map(\emptyset_i,t) \simeq \Prod{\cF}* \simeq *.
\]
To finish off Parts (4) and (5) note that the initial object is sent to the initial object under both maps.

The proof of Part (6) is similar to the proof of the previous parts and uses Part (1) and the fact that ultraproducts respect pullbacks.
\end{proof}

\begin{cor} \label{lem:ultraadjunction}
Given a collection of adjunctions $(f_i\colon \cC_i \leftrightarrows \cD_i\noloc g_i)_{i \in I}$ and an ultrafilter $\cF$ on $I$, there is an induced adjunction
\[
\xymatrix{\Prod{\cF}f_i\colon \Prod{\cF}\cC_i \ar@<0.5ex>[r] & \Prod{\cF}\cD_i\noloc\Prod{\cF}g_i \ar@<0.5ex>[l]}
\]
such that the following diagram commutes:
\[
\xymatrix{\Prod{I}\cC_i \ar@<0.5ex>[r]^{\Prod{}f_i} \ar[d]_{[-]} & \Prod{I}\cD_i \ar@<0.5ex>[l]^{\Prod{}g_i} \ar[d]^{[-]} \\ \Prod{\cF}\cC_i \ar@<0.5ex>[r]^{\Prod{\cF}f_i} & \Prod{\cF}\cD_i. \ar@<0.5ex>[l]^{\Prod{\cF}g_i}}
\]
\end{cor}
\begin{proof}
By \cite[Definition 5.2.2.1]{htt}, an adjunction between $\infty$-categories is a Cartesian and coCartesian fibration over $\Delta^1$. Since $\Prod{\cF} \Delta^1 \simeq \Delta^1$ by \cref{ex:finite} and \cref{lem:setwise}, the result follows from Part (6) of \cref{prop:ultracolims}.
\end{proof}

\begin{lem}\label{lem:ultrafiltered}
Suppose $(J_i)$ is a collection of filtered categories, then $\Prod{\cF}J_i$ is also filtered. 
\end{lem}
\begin{proof}
It is enough to show that for any finite simplicial set $K$ the map 
\[
\Fun_0(K^{\rhd},\Prod{\cF}J_i) \to \Fun_0(K,\Prod{\cF}J_i),
\]
given by restriction, is surjective. The subscript $0$ is meant to denote that we are taking the $0$-simplices of the simplicial set of functors. The claim now follows from the proof of Part (1) of \cref{prop:ultracolims}, which establishes an isomorphism of simplicial sets
\[
\Fun(K,\Prod{\cF}J_i) \cong \Prod{\cF}\Fun(K, J_i)
\]
for any finite simplicial set $K$ and the fact that the ultraproduct of surjective maps is surjective.
\end{proof}

In contrast to \Cref{prop:ultracolims}, the ultraproduct in $\Cat$ does not behave well with respect to infinite (co)limits. It does not send presentable $\infty$-categories to presentable $\infty$-categories. We work the example of the ultraproduct of the category of sets at a non-prinicipal ultrafilter in order to clarify these issues.

\begin{ex}\label{ex:ultraproductswithoutlimits}
Consider $\Prod{\cF} \Set$, the ultraproduct of the category of sets over a non-principal ultrafilter $\cF$. We will produce an infinite tower in this category with no limit.

Let $\hat{\N}=[\N]_{\cF}$. This object has the property that
\[
\Map_{\Prod{\cF}\Set}(*, \hat{\N}) \cong \Prod{\cF}\N
\]
as sets. Note that $\Prod{\cF}\N$ is linearly ordered by {\L}o{\'s}'s theorem. However, unlike $\N$, it has the property that an element may have an infinite number of elements less than it. In fact, every element of $\N \subset \Prod{\cF}\N$ is less than every element of $(\Prod{\cF}\N) \smallsetminus \N$. The successor function applied to each coordinate $\Prod{I}\N \lra{s} \Prod{I}\N$ induces a map $\hat{\N} \lra{s} \hat{\N}$. Consider the diagram
\[
\ldots \lra{s} \hat{\N} \lra{s} \hat{\N}.
\]
Assume that a limit exists and call it $X$, then $X$ has the property that
\[
\Map_{\Prod{\cF}\Set}(*,X) \cong \lim \Map_{\Prod{\cF}\Set}(*,\hat{\N}) \cong \lim \Prod{\cF}\N.
\]
There is an isomorphism
\[
\lim \Prod{\cF}\N \cong \big ( \Prod{\cF}\N \big ) \smallsetminus \N.
\]
To see this note that it is clear that $\N$ is not in the limit. Every other element can have $1$ subtracted from it to get another element in $\big ( \Prod{\cF}\N \big ) \smallsetminus \N$. That is, the map is an isomorphism on the subset $\big ( \Prod{\cF}\N \big ) \smallsetminus \N$.

Finally, the limit (assuming it exists) $X$ must be nonempty and thus must be the image of a sequence $(X_i) \in \Prod{I} \Set$ in which the sets $X_i$ can be taken to be nonempty. We have a canonical map 
\[
g\colon X \lra{} \hat{\N}
\]
to the first $\hat{\N}$ in the sequence. Since
\[
\Map_{\Prod{\cF}\Set}(X, \hat{\N}) \cong \Prod{\cF} \Map_{\Set}(X_i, \N),
\]
the map $g$ can be represented by a collection of maps $g_i\colon X_i \rightarrow \N$. Let $n_i$ be the smallest natural number in the image of $g_i$. The image of
\[
\Map_{\Prod{\cF}\Set}(*,X) \lra{} \big ( \Prod{\cF}\N \big ) \smallsetminus \N
\]
cannot hit an element smaller than $[n_i]$ in $\Prod{\cF}\N$ as the following commutes
\[
\xymatrix{\Prod{I} \Map_{\Set}(*,X_i) \ar[r] \ar[d] & \Prod{I} \N \ar[d] \\ \Map_{\Prod{\cF}\Set}(*,X) \ar[r] & \Prod{\cF}\N.}
\]
Thus, the map
\[
\Map_{\Prod{\cF}\Set}(*,X) \lra{} \big ( \Prod{\cF}\N \big ) \smallsetminus \N
\]
is not an isomorphism and $X$ cannot be the limit.
\end{ex}

\begin{rem}
In light of this example, the reader might wonder if the ultraproduct internal to the $\infty$-category $\mathrm{Pr}^L$ of presentable $\infty$-categories and colimit preserving functors is better behaved than $\Prod{\cF}$. This is not the case: Suppose $\cF$ is a non-principal ultrafilter on a set $I$ and $(\cC_i)_{i \in I}$ is a collection of presentable $\infty$-categories. We claim that $\Prod{\cF}^{\mathrm{pr}}\cC_i$ is contractible, where $\Prod{\cF}^{\mathrm{pr}}$ denotes the ultraproduct internal to $\mathrm{Pr}^L$. Denote by $0_{\mathrm{pr}} = \Delta^{0}$ the initial presentable $\infty$-category and set
\[
\cC_{i,U} = 
	\begin{cases} 
	\cC_i &  i \in U \\
	0_{\mathrm{pr}} & i \not \in U.
	\end{cases} 
\]
By \cite[Remark 4.3.11]{ambidexterity}, products and coproducts are canonically equivalent in $\mathrm{Pr}^L$. 
So we get 
\[
\Prod{\cF}^{\mathrm{pr}}\cC_i \simeq  \colim_{U\in \cF} {\coprod}_{i \in U} \cC_i \simeq   \colim_{U\in \cF} {\coprod}_{i \in I} \cC_{i,U} \simeq  {\coprod}_{i \in I} \colim_{U\in \cF} \cC_{i,U}.
\]
Now since $\cF$ is a non-principal, for any $i \in I$, there exists $U \in \cF$ with $i \notin U$, so we get 
\[
\Prod{\cF}^{\mathrm{pr}}\cC_i   \cong   {\coprod}_{i \in I} 0_{\mathrm{pr}}   \cong  0_{\mathrm{pr}}.
\]
\end{rem}

\begin{lem}\label{lem:lsymmmon}
Let $(\cC_i)_{i \in I}$ and $(\cD_i)_{i \in I}$ be collections of symmetric monoidal $\infty$-categories, let $(f_i \colon \cC_i \rightarrow \cD_i)_{i \in I}$ be a collection of symmetric monoidal functors, and let $\cF$ be an ultrafilter on $I$. The ultraproducts
\[
\Prod{\cF}\cC_i \text{ and } \Prod{\cF}\cD_i
\]
are symmetric monoidal and the induced functor
\[
\Prod{\cF}f_i \colon \Prod{\cF}\cC_i \rightarrow \Prod{\cF}\cD_i
\]
is symmetric monoidal. Also, the canonical map
\[
[-]_{\cF} \colon \Prod{I}\cC_i \longrightarrow \Prod{\cF}\cC_i
\]
is symmetric monoidal.
\end{lem}
\begin{proof}
The $\infty$-category of symmetric monoidal $\infty$-categories and symmetric monoidal functors is given as the category $\CAlg(\Cat)$ of commutative algebra objects in $\Cat$. The forgetful functor $\CAlg(\Cat) \to \Cat$ preserves products and filtered colimits by \cite[Proposition 3.2.2.1]{ha} and \cite[Corollary 3.2.3.2]{ha}. 
\end{proof}

\subsection{Compactly generated ultraproducts}\label{ssec:omegaultraproduct}
The problems with the ultraproduct of $\infty$-categories exposed in \cref{ex:ultraproductswithoutlimits} are due to the fact that the ultraproduct is being taken in $\Cat$. In this subsection we study the ultraproduct in $\Catomega$, the $\infty$-category of compactly generated $\infty$-categories and functors which preserve colimits and compact objects. We will call the ultraproduct in $\Catomega$ the \emph{compactly generated ultraproduct} and denote it by $\Prod{\cF}^{\omega}(-)$. 

Associating to a compactly generated $\infty$-category its subcategory of compact objects induces a functor
\[
\xymatrix{(-)^{\omega}\colon\Catomega \ar[r] & \Cat}
\]
that preserves limits and filtered colimits (see \cite[Lemma A.4]{heuts_goodwillie}). In fact, this functor induces an equivalence between $\Catomega$ and the $\infty$-category of small idempotent complete $\infty$-categories with finite colimits and finite colimit preserving functors \cite[Proposition A.1]{heuts_goodwillie}. The inverse equivalence is given by $\Ind(-)$. Note that the forgetful functor $\Catomega \to \Cat$ does not preserve filtered colimits or infinite products, so the compactly generated ultraproduct is not the ultraproduct in $\Cat$. Moreover, recall that we always assume the symmetric monoidal structure $\otimes$ on a compactly generated $\infty$-category $\cC$ to preserve compact objects, i.e., if $X, Y \in \cC^{\omega}$, then $X \otimes Y \in \cC^{\omega}$. 

\begin{prop}
Let $(\cC_i)_{i \in I}$ be a collection of compactly generated $\infty$-categories and let $\cF$ be an ultrafilter on $I$. There is an equivalence 
\[
\Prod{\cF}^{\omega} \cC_i \simeq \Ind \Prod{\cF} \cC_{i}^{\omega}
\]
between the compactly generated ultraproduct and the $\Ind$-category of the ultraproduct of the categories of compact objects.  
\end{prop}
\begin{proof}
The $\infty$-category $\Prod{\cF}^{\omega} \cC_i$ is compactly generated by definition. By \cite[Proposition A.1]{heuts_goodwillie}, it suffices to determine its subcategory of compact objects. Since $(-)^{\omega}$ preserves filtered colimits and products \cite[Lemma A.4]{heuts_goodwillie}, we see that
\[
(\Prod{\cF}^{\omega} \cC_i)^{\omega} \simeq \Prod{\cF}\cC_i^{\omega}
\]
and the claim follows. 
\end{proof}

For an $\infty$-category $\cC$, let 
\[
\Pre(\cC) = \Fun(\cC^{op},\Top)
\]
be the $\infty$-category of presheaves on $\cC$. Consider the natural inclusion $\iota\colon\Prod{\cF}\cC_{i}^{\omega} \to \Prod{\cF}\cC_i$. The Yoneda embedding
\[
\Prod{\cF}\cC_i \lra{y} \Pre(\Prod{\cF}\cC_i)
\]
may be restricted along $\iota$ to give a map
\[
y_{\iota} \colon \Prod{\cF}\cC_{i} \lra{} \Pre(\Prod{\cF}\cC_{i}^{\omega}).
\]
Since $\iota$ preserves finite colimits by Part (5) of \cref{prop:ultracolims}, $y_{\iota}$ factors through $\Ind(\Prod{\cF}C_{i}^{\omega}) \simeq \Prod{\cF}^{\omega}\cC_{i}$ to give map
\[
m \colon \Prod{\cF}\cC_{i} \lra{} \Prod{\cF}^{\omega}\cC_{i}.
\]
We may use this to build a ``localization map" from the product to the compactly generated ultraproduct. We will show that the composite
\[
\Prod{I}\cC_i \lra{[-]} \Prod{\cF}\cC_i \lra{m} \Prod{\cF}^{\omega}\cC_{i}.
\]
is well-behaved.

\begin{lem}\label{lem:catsinheritproperties}
Let $\cC_i \in \Catomega$ be a collection of compactly generated $\infty$-categories indexed by some set $I$. If 
	\begin{enumerate}
		\item $\cC_i$ is stable for all $i\in I$ or
		\item $\cC_i$ is symmetric monoidal for all $i\in I$,
	\end{enumerate}
then so is $\Prod{\cF}^{\omega}\cC_i$ for any ultrafilter $\cF$ on $I$. 
\end{lem}
\begin{proof}
For Part (1), \cref{lem:stablerozenblyum} implies that ultraproducts of stable $\infty$-categories are stable and it follows from \cite[Theorem 1.1.3.6]{ha} that $\cC$ stable implies that $\Ind(\cC)$ is stable.  Part (2) follows from \cref{lem:lsymmmon} and \cite[Theorem 4.8.1.13]{ha}.
\end{proof}

Before we can state the next proposition, we recall some basic facts about lax symmetric monoidal functors. Let $\cC$ and $\cD$ be symmetric monoidal $\infty$-categories. In the notation of \cite{ha}, a map
\[
\xymatrix{\cC^{\otimes} \ar[rr] \ar[dr] & & \cD^{\otimes} \ar[dl] \\ & \mathrm{N}(\mathrm{Fin}_*) &} 
\]
is lax symmetric monoidal if it sends inert maps to inert maps. Let   $\Fun^{\mathrm{lax}}(\cC^{\otimes}, \cD^{\otimes}) \subset \Fun_{\mathrm{N}(\mathrm{Fin}_*)}(\cC^{\otimes}, \cD^{\otimes}) $ be the full subcategory consisting of lax symmetric monoidal functors. Let $\mathrm{Fin}^{\mathrm{inj}}_*$ denote the subcategory of $\mathrm{Fin}_*$ spanned by all objects together with those
morphisms $f \colon \langle m \rangle \to \langle n \rangle$ such that $|f^{-1}(i)| \leq 1$ for $1 \leq i \leq n$. We say that $F \in \Fun^{\mathrm{lax}}(\cC^{\otimes}, \cD^{\otimes})$ is unital if $F$ sends coCartesian edges over $\mathrm{Fin}^{\mathrm{inj}}_*$ to coCartesian edges.

The following lemma is an easy application of \cite[Lemma 2.16]{AFT}:
\begin{lem} \label{lem:msymmmon1}
Let $\mathcal{B}$ be a small symmetric monoidal $\infty$-category that admits finite colimits   and such that the symmetric monoidal structure for $\mathcal{B}$  distributes over finite colimits. Let $\mathcal{M}$ be a (locally small) symmetric monoidal $\infty$-category that admits finite colimits. Let $i \colon \mathcal{B} \to \mathcal{M}$ be fully faithful symmetric monoidal left exact functor. Then there exists a unital lax symmetric monoidal functor $i_{\natural}^{\otimes}\colon \mathcal{M}\to \Ind( \mathcal{B})$ 
such that the underlying functor from $\mathcal{M} $ to $\Ind( \mathcal{B})$ is the restriction of the Yoneda embedding.
 \end{lem}
\begin{proof}
We are going to apply \cite[Lemma 2.16]{AFT} to the case $\mathcal{V} = \Ind(\mathcal{B})$. First note that \cite[Lemma 2.16]{AFT} gives a unital lax symmetric monoidal functor if Conditions (1)-(4) of the lemma are satisfied. Conditions (1) and (3) follow from \cite[Theorem 4.8.1.13]{ha}. 
Since $i$ is symmetric monoidal, we have $i(1_{\mathcal{B}}) \cong 1_\mathcal{M}$. Since $i$ is fully faithful, the functor 
\[
\mathcal{B}_{/1_{\mathcal{B}}}\to \mathcal{B}_{/1_{\mathcal{M}}}
\]
is an equivalence and thus final. This gives condition (4).
To prove condition (2), we shall show that, for $M \in \mathcal{M}$, $\mathcal{B}_{M}$ is filtered (and in particular sifted).
Indeed, let $\rho \colon K \to  \mathcal{B}_{M}$ be a finite diagram. Since $\mathcal{B}$ admits finite colimits and  $i$ preserves colimits, $\rho$ can be extended to a colimit diagram $\rho \colon K^{\rhd} \to  \mathcal{B}_{M}$.
\end{proof}

\begin{cor}\label{prop:msymmmon}
Let $(\cC_i)_{i \in I}$ be a collection of symmetric monoidal compactly generated $\infty$-categories and let $\cF$ be an ultrafilter on the set $I$. The map
\[
m\colon\Prod{\cF} \cC_i \longrightarrow \Prod{\cF}^{\omega}\cC_i 
\]
is a unital symmetric monoidal functor.
\end{cor}
\begin{proof}
Set $i \colon \mathcal{B} \to \mathcal{M}$ to be the inclusion
\[
\Prod{\cF} \cC_i^{\omega}  \to \Prod{\cF} \cC_i 
\]
in \cref{lem:msymmmon1} so that $m = i_{\natural}^{\otimes}$. This implies that $m$ is unital and lax symmetric monoidal.

Finally, to prove that $m$ is symmetric monoidal, we must verify condition (5) in \cite[Lemma 2.16]{AFT}. We must show that if $x_1$ and $x_2$ are objects in $\Prod{\cF} \cC_i$, then the functor
\[
\bigotimes \colon (\Prod{\cF} \cC_i^{\omega})_{/x_1} \times (\Prod{\cF} \cC_i^{\omega})_{/x_2} \to (\Prod{\cF} \cC_i^{\omega})_{/(x_1 \otimes x_2)}
\]
is cofinal. By Theorem 4.1.3.1 in \cite{htt}, it is enough to show that for every $x_0 \in (\Prod{\cF} \cC_i^{\omega})_{/(x_1 \otimes x_2)}$ the groupoidification of
\[
\cD_{x_0,x_1,x_2} = \left ((\Prod{\cF} \cC_i^{\omega})_{/x_1} \times (\Prod{\cF} \cC_i^{\omega})_{/x_2} \right) \times_{(\Prod{\cF} \cC_i^{\omega})_{/(x_1 \otimes x_2)}}   (\Prod{\cF} \cC_i^{\omega})_{x_0//(x_1 \otimes x_2)}
\]
is weakly contractible. For this, it is enough to show that $\cD_{x_0,x_1,x_2}$ is filtered. Since the property of being filtered is preserved by ultraproducts by \cref{lem:ultrafiltered}, \cref{prop:ultracolims} (2) and the fact that filtered colimits commute with pullbacks imply that we are reduced to the case where $\cF$ is principal. 

Put differently, let $\cC$ be a compactly generated symmetric monoidal $\infty$-category. Given a pair of objects $x_1,x_2 \in \cC$ and an object $x_0\in (\cC^{\omega})_{/(x_1 \otimes x_2)}$, we must show that the $\infty$-category 
\[
\cD_{x_0,x_1,x_2} =  \left ( (\cC^{\omega})_{/x_1} \times  (\cC^{\omega})_{/x_2} \right) \times_{(\cC^{\omega})_{/(x_1 \otimes x_2)}}    (\cC^{\omega})_{x_0//(x_1 \otimes x_2)}
\]
is filtered. We will do this by expressing $\cD_{x_0,x_1,x_2}$ as a filtered colimit of filtered $\infty$-categories.

To do this, we will first show that the functor
\[
F \colon \cC \to (\Cat)_{/\cC^{\omega}}
\]
\[
x \mapsto \cC^{\omega}_{/x} 
\]
commutes with filtered colimits. Indeed, given a filtered diagram $\rho \colon J \to \cC$, the canonical map
\[
\colim_{J} (F \circ \rho) \to F(\colim_{J} \rho)
\]
is a map between two right fibrations over $\cC^{\omega}$ and thus can be checked to be an equivalence fiberwise. Thus it suffices to check that this map is an equivalence at the fiber over $c \in \cC^{\omega}$. Since filtered colimits commute with pullbacks in $\Cat$, the map corresponding to the fiber over $c \in \cC^{\omega}$ is 
\[
\colim_{j \in J} \Map_{\cC} (c, x_j) \to \Map_{\cC} (c, \colim_{j \in J} x_j), 
\]
which is an equivalence by the compactness of $c$.
Further, since the forgetful functor 
\[
(\Cat)_{/\cC^{\omega}} \to \Cat
\]
commutes with filtered colimits, the functor sending 
\[
x \mapsto \cC^{\omega}_{/x}
\]
commutes with filtered colimits as a functor to $\Cat$.
 
Given $x_0$, the $\infty$-category $\cD_{x_0,x_1,x_2}$ is functorial in $x_1$ and $x_2$ and thus can be considered as a functor
\[
\cD_{x_0, -,-} \colon \cC \times \cC \to \Cat.
\]

Further, since filtered colimits commute with pullbacks in $\Cat$ and with tensor products in $\cC$, the functor $\cD_{x_0,-,-}$ commutes with filtered colimits. Since a filtered colimit of filtered $\infty$-categories is a filtered $\infty$-category, we can reduce to the case that $x_1$ and $x_2$ are in the essential image of the embedding $\cC^{\omega} \to \cC$. In this case the $\infty$-category $\cD_{x_0,x_1,x_2}$ has a terminal object 
\[
\left((x_1 \xrightarrow{\Id_{x_1 }} x_1  , x_2 \xrightarrow{\Id_{ x_2}}  x_2),x_0 \to x_1 \otimes x_2 \xrightarrow{\Id_{x_1 \otimes x_2}} x_1 \otimes x_2 \right),
\]
which concludes the proof.
\end{proof}

In certain cases, collections of adjunctions between compactly generated $\infty$-categories give rise to adjunctions of compactly generated ultraproducts.

\begin{lem} \label{lem:compactadjunction}
If $f_i\colon \cC_i \leftrightarrows \cD_i \noloc g_i$ is a collection of adjunctions between compactly generated $\infty$-categories indexed by $I$ such that the left adjoints preserve compact objects, then there exists an induced adjunction
\[
\xymatrix{\Prod{\cF}^{\omega}f_i\colon \Prod{\cF}^{\omega}\cC_i \ar@<0.5ex>[r] & \Prod{\cF}^{\omega}\cD_i\noloc g_{\cF} \ar@<0.5ex>[l]}
\]
for any ultrafilter $\cF$ on $I$ such that the following diagram commutes:
\[
\xymatrix{\Prod{\cF}\cC_i \ar@<0.5ex>[r]^{\Prod{\cF}f_i} \ar[d]_{m} & \Prod{\cF}\cD_i \ar@<0.5ex>[l]^{\Prod{\cF}g_i} \ar[d]^{m} \\ \Prod{\cF}^{\omega}\cC_i \ar@<0.5ex>[r]^{\Prod{\cF}^{\omega}f_i} & \Prod{\cF}^{\omega}\cD_i. \ar@<0.5ex>[l]^{g_{\cF}}}
\]
\end{lem}
\begin{proof}
By assumption and \cref{prop:ultracolims}
\[
\Prod{\cF}f_{i}^{\omega} \colon \Prod{\cF}\cC_{i}^{\omega} \lra{} \Prod{\cF}\cD_{i}^{\omega}
\]
preserves finite colimits, so $\Prod{\cF}^{\omega}f_i$ preserves all colimits. The existence of the right adjoint $g_{\cF}$ follows. 

Let $[T_i] \in \Prod{\cF}\cC_{i}^{\omega}$ and $[d_i] \in \Prod{\cF}\cD_i$. There are equivalences
\begin{align*}
\Map_{\Prod{\cF}^{\omega} \cC_i}([T_i], g_{\cF}(m([d_i]))) &\simeq \Map_{\Prod{\cF}^{\omega}\cD_i}((\Prod{\cF} f_i)([T_i]), m([d_i])) \\ 
&\simeq \Map_{\Prod{\cF}D_i}([f_i(T_i)],[d_i]) \\
&\simeq \Map_{\Prod{\cF}\cC_i}([T_i],[g_i(d_i)]) \\
&\simeq \Map_{\Prod{\cF}^{\omega}\cC_i}([T_i],m[g_i(d_i)]).
\end{align*}
The commutativity of the other square follows from the naturality of $m$.
\end{proof}

\begin{lem} \label{lem:symmetriccompactfunctors}
Let $(\cC_i)_{i \in I}$ and $(\cD_i)_{i \in I}$ be collections of symmetric monoidal compactly generated $\infty$-categories and let $(f_i)_{i \in I}$ be a collection of symmetric monoidal functors that preserve colimits and compact objects. The compactly generated ultraproduct of the collection $(f_i)_{i \in I}$
\[
\Prod{\cF}^{\omega} f_i \colon \Prod{\cF}^{\omega} \cC_i \rightarrow \Prod{\cF}^{\omega} \cD_i
\]
is a symmetric monoidal functor that preserves colimits and compact objects.
\end{lem}
\begin{proof}
Restricting $\Prod{\cF}^{\omega}f_i$ to compact objects gives the functor
\[
\Prod{\cF}f_{i}^{\omega} \colon \Prod{\cF} \cC_{i}^{\omega} \rightarrow \Prod{\cF}\cD_{i}^{\omega},
\]
which is symmetric monoidal by \cref{lem:lsymmmon} and preserves finite colimits by Part (5) of \cref{prop:ultracolims}. Applying $\Ind(-)$ to this symmetric monoidal functor yields $\Prod{\cF}^{\omega}f_i$, which is symmetric monoidal by \cite[Corollary 4.8.1.13]{ha}.
\end{proof}

Note that for $c \in \Prod{\cF}^{\omega}\cC_i$ a compact object, we have an equivalence $c \simeq [c_i]$ for some $(c_i) \in \Prod{I}\cC_{i}^{\omega}$.

\begin{lem}\label{lem:htpyultracompact}
Let $c, d \in \Prod{\cF}^{\omega}\cC_i$ be compact objects such that $c \simeq [c_i]$ and $d \simeq [d_i]$ with $c_i, d_i \in \cC_{i}^{\omega}$. Then there is an equivalence 
\[
\Map(c,d) \simeq \Prod{\cF}\Map(c_i, d_i),
\]
where the ultraproduct on the right is computed in the $\infty$-category of spaces. The same result holds for mapping spectra in case the categories are stable.
\end{lem}
\begin{proof}
Because $c$ and $d$ are compact, we may compute the mapping space in the $\infty$-category $(\Prod{\cF}^{\omega}\cC_i)^{\omega} \simeq \Prod{\cF}\cC_i^{\omega}$. The result then follows from \cref{lem:rozenblyum} and \Cref{lem:stablerozenblyum}.
\end{proof}

\begin{lem}\label{fullyfaithful2}
The compactly generated ultraproduct of fully faithful functors between compactly generated $\infty$-categories is fully faithful.
\end{lem}
\begin{proof}
The result follows from \cref{fullyfaithful} and the fact that $\Ind$ preserves fully faithful functors by \cite[Proposition 5.3.5.11(1)]{htt}.
\end{proof}

\subsection{Protoproducts of compactly generated $\infty$-categories}

In this subsection we construct a variation of the compactly generated ultraproduct that takes into account filtrations on compactly generated $\infty$-categories.

\begin{defn}\label{def:compactfiltration}
A compact filtration $(\cC,F_*)$ of a compactly generated $\infty$-category $\cC$ is a sequence of fully faithful subcategories (closed under equivalences and retracts)
\[
\xymatrix{F_0\cC \ar[r] & F_1\cC \ar[r] & F_2\cC \ar[r] & \ldots \ar[r] & \cC^{\omega}}
\]
over $\cC$ such that 
\begin{enumerate}
\item the initial object $\emptyset$ is in $F_0\cC$
\item $\colim F_k\cC \simeq \cC^{\omega}$
\item there exists a function $\alpha \colon \N \to \N$ such that, for a diagram of the form $c \leftarrow e \rightarrow d$ in $F_k\cC$, the pushout 
\[
c\Coprod{e}d \in \cC^{\omega}
\]
lies in $F_{\alpha(k)}\cC$.
\end{enumerate}
\end{defn}

If $(\cC, F_*)$ is a compact filtration, $\cC$ is stable, and for $c \in F_k\cC$
\[
\Sigma^{-1}c \in F_{\alpha(k)}\cC,
\]
then we call the compact filtration a stable compact filtration. If $(\cC, F_*)$ is a compact filtration, $\cC$ is symmetric monoidal, the tensor unit of $\cC$ is contained in $F_0\cC$, and for $c,d \in F_k \cC$
\[
c \otimes d \in F_{\alpha(k)}\cC,
\]
then we call the compact filtration a symmetric monoidal compact filtration.

\begin{rem} \label{rem:lur}
By \cite[Proposition 4.4.2.2]{htt}, if a function $\alpha$ exists for the pushout then this implies the existence of such a function for any finite diagram category.
\end{rem}

\begin{defn}
A collection of $\infty$-categories equipped with (stable) (symmetric monoidal) compact filtrations is a (stable) (symmetric monoidal) filtered collection if there is a single function $\alpha$ that satisfies the conditions of the definition for each of the compact filtrations.
\end{defn}

\begin{defn}
Let $\cF$ be an ultrafilter on $I$. We define the protoproduct of a filtered collection of compactly generated $\infty$-categories $(\cC_i, F_{i,*})_{i \in I}$ to be
\[
\xymatrix{\Prod{\cF}^{\flat} (\cC_i, F_{i,*}) = \Ind\colim_k\Prod{\cF}F_{i,k}\cC_i.}
\]
\end{defn}

\begin{lem}
There is an equivalence of $\infty$-categories
\[
(\Prod{\cF}^{\flat} (\cC_i, F_{i,*}))^{\omega} \simeq \colim_k\Prod{\cF}F_{i,k}\cC_i.
\]
\end{lem}
\begin{proof}
This follows from the fact that
\[
\colim_k\Prod{\cF}F_{i,k}\cC_i
\]
is idempotent complete and has finite colimits by \cref{rem:lur}.
\end{proof}

\begin{ex}\label{ex:trivialfiltration}
The compactly generated ultraproduct is the special case of the protoproduct for which $F_{i,j}\cC_i = \cC_{i}^{\omega}$ for all $j$.
\end{ex}

\begin{lem} \label{lem:fullyfaithfultocompact}
Let $(\cC_i,F_{i,*})$ be a filtered collection of compactly generated $\infty$-categories. Then there is a fully faithful functor
\[
\xymatrix{\iota\colon\Prod{\cF}^{\flat}(\cC_i,F_{i,*}) \ar@{^{(}->}[r] & \Prod{\cF}^{\omega} \cC_i}.
\]
\end{lem}
\begin{proof}
For the categories of compact objects this follows from \cref{fullyfaithful} and $\Ind$ preserves fully faithful functors by \cite[Proposition 5.3.5.11(1)]{htt}.
\end{proof}

\begin{lem}\label{lem:protoproductcompgen}
The protoproduct of a (stable) (symmetric monoidal) filtered collection of compactly generated $\infty$-categories $(\cC_i, F_{i,*})_{i \in I}$ is compactly generated (and stable) (and symmetric monoidal).
\end{lem}
\begin{proof}
The $\infty$-category
\[
\colim_k\Prod{\cF}F_{i,k}\cC_i
\]
is a subcategory of $\Prod{\cF}\cC_{i}^{\omega}$ that is closed under finite colimits and retracts by \cref{rem:lur}. 

By \cref{lem:fullyfaithfultocompact}, the natural inclusions induce the fully faithful functor
\[
\xymatrix{\Prod{\cF}^{\flat}(\cC_i, F_{i,*}) \ar[r]^-{\iota} & \Prod{\cF}^{\omega}\cC_i.}
\]
Since the target is symmetric monoidal by \Cref{lem:catsinheritproperties}, \cite[Proposition 2.2.1.1]{ha} implies that it suffices to show that $\Prod{\cF}^{\flat}(\cC_i, F_{i,*})$ is closed under the symmetric monoidal structure on $\Prod{\cF}^{\omega}\cC_i$. But this follows from the assumptions on the filtrations $F_{i,*}$.

From the characterization of stable $\infty$-categories in \cite[Corollary 1.4.2.27(2)]{ha}, stability follows from the fact that desuspension maps
\[
\colim_k\Prod{\cF}F_{i,k}\cC_i
\]
to itself and thus desuspension is an equivalence
\[
\Sigma^{-1}\colon \colim_k\Prod{\cF} F_{i,k} \cC_i \lra{\simeq} \colim_k\Prod{\cF} F_{i,k} \cC_i.
\]
\end{proof}

\begin{defn}
Let $(\cC_i, F_{i,*})_{i \in I}$ and $(\cD_i, G_{i,*})_{i \in I}$ be filtered collections of compactly generated $\infty$-categories. A collection of colimit preserving functors
\[
(f_i\colon \cC_i \rightarrow \cD_i)_{i \in I}
\]
is called a \emph{collection of filtration preserving functors} if there exists a single function $\beta\colon \N \rightarrow \N$ such that
\[
f_i(F_{i,k}\cC_i) \subseteq G_{i,\beta(k)}\cD_i
\]
for all $i$ and all $k$.
\end{defn}

A collection of filtration preserving functors $(f_i\colon \cC_i \to \cD_i)$ induces a functor
\[
\Prod{\cF}^{\flat}f_i \colon \Prod{\cF}^{\flat}(\cC_i,F_{i,*}) \to \Prod{\cF}^{\flat}(\cD_i,G_{i,*}).
\]

\begin{lem}
A collection of fully faithful filtration preserving functors $(f_i\colon \cC_i \hookrightarrow \cD_i)$ induces a fully faithful functor
\[
\Prod{\cF}^{\flat}f_i \colon \Prod{\cF}^{\flat}(\cC_i,F_{i,*}) \hookrightarrow \Prod{\cF}^{\flat}(\cD_i,G_{i,*}).
\]
\end{lem}
\begin{proof}
Combining \cref{fullyfaithful2} and \cref{lem:fullyfaithfultocompact}, we have a commutative diagram
\[
\xymatrix{\Prod{\cF}^{\flat}(\cD_i,G_{i,*}) \ar@{^{(}->}[r] & \Prod{\cF}^{\omega}\cD_i \\ \Prod{\cF}^{\flat}(\cC_i,F_{i,*}) \ar[u] \ar@{^{(}->}[r] & \Prod{\cF}^{\omega}\cC_i \ar@{^{(}->}[u]}
\]
in which the horizontal arrows and right vertical arrow are fully faithful. This implies that the left arrow is fully faithful.
\end{proof}

\begin{prop}\label{lem:protoadjunction}
Let $(\cC_i, F_{i,*}\cC_i)_{i \in I}$ and $(\cD_i, G_{i,*}\cD_i)_{i \in I}$ be filtered collections of compactly generated $\infty$-categories and assume that we have a collection of filtration preserving functors
\[
(f_i\colon\cC_i \lra{} \cD_i)_{i \in I}.
\] 
Then the protoproduct of these functors
\[
\Prod{\cF}^{\flat}f_i\colon \Prod{\cF}^{\flat}(\cC_i, F_{i,*}) \longrightarrow \Prod{\cF}^{\flat}(\cD_i, G_{i,*})
\]
preserves colimits and has a right adjoint $g_{G_{i,*}}^{F_{i,*}}$ that preserves filtered colimits.  
\end{prop}
\begin{proof}
It suffices to prove that the functor
\[
(\Prod{\cF}^{\flat}f_i)^{\omega}\colon (\Prod{\cF}^{\flat}(\cC_i,F_{i,*}))^{\omega} = \colim_k\Prod{\cF}F_{i,k}\cC_i \rightarrow \colim_k \Prod{\cF}G_{i,k}\cD_i
\]
preserves finite colimits. Since it clearly preserves the initial object, it is enough to show that it preserves pushouts. Let $K$ be the span diagram. Since $K$ is compact, there exists a $k$ and $U \in \cF$ such that there are factorizations
\[
\xymatrix{& \Prod{U} F_{i,k}\cC_i \ar[d] \\ & \Prod{\cF} F_{i,k}\cC_i \ar[d] \\ K \ar[r] \ar@{-->}[ur] \ar@{-->}[uur] & \colim_k \Prod{\cF} F_{i,k}\cC_i.}
\]
By the definition of a filtered collection, there is a factorization
\[
\xymatrix{K^{\rhd} \ar@{-->}[r] & \Prod{U} F_{i,\alpha(k)}\cC_i \ar[r] & \Prod{U}\cC_{i}^{\omega} \\ K \ar[r] \ar[u] & \Prod{U} F_{i,k}\cC_i \ar[u] &}
\]
and the map $K^{\rhd} \to \Prod{U}\cC_{i}^{\omega}$ is the colimit. The canonical map
\[
\Prod{U} F_{i,\alpha(k)}\cC_i \to \colim_k \Prod{\cF} F_{i,k}\cC_i
\]
preserves the finite colimits that exist by \Cref{prop:ultracolims}. Also, the map
\[
\Prod{U}f_i \colon \Prod{U} F_{i,\alpha(k)}\cC_i \to \Prod{U} G_{i,\beta (\alpha(k))}\cD_i
\]
preserves colimits that exist by assumption. Thus the composite
\[
K^{\rhd} \to \Prod{U} F_{i,\alpha(k)}\cC_i \to \Prod{U} G_{i,\beta(\alpha(k))}\cD_i \to \colim_k \Prod{\cF} G_{i,k}\cD_i
\]
is the pushout diagram and $\Prod{\cF}^{\flat}f_i$ sends pushouts to pushouts. By \cite[Proposition 5.5.7.2]{htt}, since $\Prod{\cF}^{\flat}f_i$ sends compact objects to compact objects, the right adjoint $g_{G_{i,*}}^{F_{i,*}}$ preserves filtered colimits.
\end{proof}

\begin{rem}
When the compact filtrations are clear from context, we will just write $g$ for the right adjoint to $\Prod{\cF}^{\flat}f_i$.
\end{rem}

\begin{cor}
With notation as in \Cref{lem:protoadjunction}, if the categories are stable then $g_{G_{i,*}}^{F_{i,*}}$ preserves all colimits. 
\end{cor}

\begin{ex}
Let $(\cC_i, F_{i,*})_{i \in I}$ be a filtered collection of compactly generated $\infty$-categories and let $(\cC_i)_{i \in I}$ be the same categories with the trivial filtration of \cref{ex:trivialfiltration}. In this case the identity maps are a collection of filtration preserving functors $(\cC_i, F_{i,*}) \rightarrow \cC_i$ inducing
\[
\xymatrix{\iota\colon\Prod{\cF}^{\flat}(\cC_i,F_{i,*}) \ar@{^{(}->}[r] & \Prod{\cF}^{\omega} \cC_i.}
\]
We will always refer to the right adjoint to this map as $n^{F_{i,*}}$ or just $n$ when the filtration is clear from context.
\end{ex}

\begin{cor} \label{cor:squares}
With the notation of \cref{lem:protoadjunction}, both the solid square and the dashed square commute
\[
\xymatrix{\Prod{\cF}^{\flat}(\cC_i,F_{i,*}) \ar@<0.5ex>[d]^{\Prod{\cF}^{\flat}f_i} \ar@<0.5ex>[r]^-{\iota} &  \Prod{\cF}^{\omega}\cC_i \ar@<0.5ex>@{-->}[l]^-{n} \ar@<0.5ex>[d]^{\Prod{\cF}^{\omega}f_i} \\
\Prod{\cF}^{\flat}(\cD_i,G_{i,*}) \ar@<0.5ex>@{-->}[u]^{g_{G_{i,*}}^{F_{i,*}}} \ar@<0.5ex>[r]^-{\iota} & \Prod{\cF}^{\omega}\cD_i \ar@<0.5ex>@{-->}[u]^{g} \ar@<0.5ex>@{-->}[l]^-{n} }
\]
and $g_{G_{i,*}}^{F_{i,*}}$ preserves filtered colimits. 
\end{cor}

\begin{proof}
Each of these adjunctions is a special case of \cref{lem:protoadjunction}. The horizontal adjunctions are a special case by \Cref{ex:trivialfiltration}. The solid diagram commutes by naturality and the commutativity of the dashed diagram follows.
\end{proof}

\begin{lem}\label{lem:nsymmmon}
Let $(\cC_i, F_{i,*})_{i \in I}$ be a filtered symmetric monoidal collection of compactly generated $\infty$-categories. The protoproduct
\[
\Prod{\cF}^{\flat}(\cC_i, F_{i,*})
\]
is a full symmetric monoidal subcategory of $\Prod{\cF}^{\omega} \cC_i$ with unital lax symmetric monoidal right adjoint $n$.
\end{lem}
\begin{proof}
By \cref{lem:fullyfaithfultocompact}, the natural inclusions induce the fully faithful functor
\[
\xymatrix{\Prod{\cF}^{\flat}(\cC_i, F_{i,*}) \ar[r]^-{\iota} & \Prod{\cF}^{\omega}\cC_i}.
\]
Since the target is symmetric monoidal by \Cref{lem:catsinheritproperties}, it suffices to show that $\Prod{\cF}^{\flat}(\cC_i, F_{i,*})$ is closed under the symmetric monoidal structure on $\Prod{\cF}^{\omega}\cC_i$, see \cite[Proposition 2.2.1.1]{ha}. This follows from the assumptions on the filtrations $F_{i,*}$.

Since $\iota\colon \Prod{\cF}^{\flat}(\cC_i,F_{i,*}) \to \Prod{\cF}^{\omega}\cC_i$ is symmetric monoidal, the right adjoint $n$ inherits a natural structure of a lax symmetric monoidal functor by \cite[Corollary 7.3.2.7]{ha}. Finally, $n$ preserves units because the tensor unit of $\cC_i$ is contained in $F_{i,0}\cC_i$ for all $i \in I$. 
\end{proof}

\begin{cor} \label{cor:symmonfun}
Let $(\cC_i, F_{i,*})_{i \in I}$ and $(\cD_i, G_{i,*})_{i \in I}$ be  filtered symmetric monoidal collections of compactly generated $\infty$-categories and let $(f_i)_{i \in I}$ be a filtration preserving collection of symmetric monoidal functors. The protoproduct
\[
\xymatrix{\Prod{\cF}^{\flat}(\cC_i, F_{i,*}) \ar[r]^{\Prod{\cF}^{\flat}f_i} & \Prod{\cF}^{\flat}(\cD_i, G_{i,*})}
\]
is a symmetric monoidal functor with unital lax symmetric monoidal right adjoint $g$.
\end{cor}
\begin{proof}
By \cref{cor:squares}, we have a commutative diagram
\[
\xymatrix{\Prod{\cF}^{\omega} \cC_i \ar[r]^{\Prod{\cF}^{\omega}f_i} & \Prod{\cF}^{\omega} \cD_i \\ \Prod{\cF}^{\flat}(\cC_i, F_{i,*}) \ar[r]_{\Prod{\cF}^{\flat}f_i} \ar[u]^{\iota} & \Prod{\cF}^{\flat}(\cD_i, G_{i,*}). \ar[u]_{\iota}}
\]
By \cref{lem:symmetriccompactfunctors}, the top arrow is symmetric monoidal. Since the diagram commutes and the vertical arrows are fully faithful and symmetric monoidal by \cref{lem:nsymmmon}, the bottom arrow must be symmetric monoidal.

The right adjoint $g$ is unital lax symmetric monoidal by the proof of the same property for the right adjoint $n$ in \cref{lem:nsymmmon}.
\end{proof}

\begin{cor}\label{cor:compositesymmmon}
Let $(\cC_i, F_{i,*})$ be a filtered symmetric monoidal collection of compactly generated $\infty$-categories. Then the composite
\[
\Prod{I} \cC_i \lra{[-]_{\cF}} \Prod{\cF}\cC_i \lra{m} \Prod{\cF}^{\omega}\cC_i \lra{n} \Prod{\cF}^{\flat}(\cC_i,F_{i,*})
\]
is lax symmetric monoidal and preserves the unit. In particular, the composite sends commutative monoids to commutative monoids.
\end{cor}
\begin{proof}
This follows from \Cref{lem:lsymmmon}, \Cref{prop:msymmmon}, and \Cref{lem:nsymmmon}. 
\end{proof}

\subsection{Filtrations on module categories}

In this subsection and the next, we study three different filtrations on the $\infty$-category of modules over a ring spectrum that will play an important role later on. These three filtrations are stable compact filtrations in the sense of the previous subsection. 

\begin{itemize}
	\item The cell filtration, in which a compact module is in filtration $k$ if and only if it can be built out of at most $k$ cells. 
	\item The cell-dimension filtration, where both the number of cells and the dimension of the cells are bounded.
	\item The Pic-filtration, in which a compact module is in filtration $k$ if and only if it can be built out of at most $k$ invertible modules. 
\end{itemize}

Informally speaking, the difference between the first and third filtration comes from the fact that invertible objects can, in general, have many cells. This is the case, for example, in the $\infty$-category of $E(n)$-local spectra at small primes. 

We will begin with some recollections regarding module categories. Let $R$ be an $\bE_{1}$-ring spectrum and let $\Mod_R$ be the stable $\infty$-category of modules over $R$. A cell is an object of the form $\Sigma^nR$ for some $n \in \Z$. For any $n \in \N \cup \{\infty\}$, if $R$ is an $\bE_n$-ring spectrum, then $\Mod_R$ is $\bE_{n-1}$-monoidal. If $n>1$, then the corresponding monoidal structure will be denoted by $\otimes = \otimes_R$. Moreover, $\Mod_R$ is compactly generated by the $\otimes$-unit $R$. In fact, this property characterizes module categories of ring spectra by the following derived version of Morita theory, see \cite[Theorem 3.3.3]{schwedeshipleymorita} and \cite[Theorem 7.1.2.1]{ha}.

\begin{thm}[Schwede--Shipley]\label{thm:moritatheory}
If $\cC$ is a compactly generated stable $\infty$-category with compact generator $P$, then there is an equivalence 
\[
\xymatrix{\Hom_{\cC}(P,-)\colon \cC \ar[r]^-{\simeq} & \Mod_{\End_{P}}}
\]
with inverse given by $- \otimes_{\End_P}P$. Moreover, if $\cC$ is symmetric monoidal with unit $P$ and with the property that $\otimes$ commutes with colimits in each variable, then $P$ is an $\bE_{\infty}$-algebra and this is an equivalence of symmetric monoidal $\infty$-categories.
\end{thm}

We can now construct three filtrations on the $\infty$-category of modules over a ring spectrum that we will use to build three different types of protoproducts. An object in $\Mod_R$ is said to be built out of at most $1$ cell if it is equivalent to a suspension of $R$ or zero. Inductively, we say that an object in $\Mod_R$ is built out of at most $k$ cells if it is equivalent to the cofiber of a map between an object built out of at most $k_1$ cells and an object built out of at most $k_2$ cells, where $k_1+k_2 \leq k$.

\begin{defn}\label{defn:cellfiltation}
Let $\Cell_R$ be the filtration on $\Mod_R$ in which $\Cell_{R,k}\Mod_R$ consists of retracts of objects that can be built out of at most $k$ cells. For a collection of module categories $(\Mod_{R_i})_{i \in I}$ equipped with this filtration and an ultrafilter $\cF$ on $I$, we will denote the protoproduct by
\[
\Prod{\cF}^{\flat}\Mod_{R_i} = \Prod{\cF}^{\flat}(\Mod_{R_i}, \Cell_{R_i}).
\] 
We will simply refer to this as the protoproduct of the module categories.
\end{defn}

There is a finer filtration on $\Mod_R$ given by bounding both the number and the dimension of cells. Every compact $R$-module is a retract of a finite cell $R$-module. The dimension of a finite cell $R$-module is the maximum of the absolute value of the dimension of the top cell and the absolute value of the dimension of the bottom cell. The $k$th filtration step is given by $R$-modules that are retracts of finite cell $R$-modules for which the maximum of the dimension and number of cells is bounded by $k$.  

\begin{defn}\label{defn:celldimfiltration}
Let $\DimCell_R$ be the filtration on $\Mod_R$ described above given by the maximum of the number of cells and the dimension. For a collection of module categories $(\Mod_{R_i})_{i \in I}$ equipped with this filtration and an ultrafilter $\cF$ on $I$, we will denote the protoproduct by
\[
\Prod{\cF}^{\doubleflat}\Mod_{R_i} = \Prod{\cF}^{\flat}(\Mod_{R_i}, \DimCell_{R_i}).
\] 
We will call this the bounded protoproduct of the module categories.
\end{defn}
 
If $R$ is an $\bE_{\infty}$-ring spectrum, then the $\infty$-category $\Mod_{R}$ is symmetric monoidal and we can define the Picard groupoid $\Pic(R)\subseteq \Mod_R$ of $R$ to be the $\infty$-groupoid of invertible objects in $\Mod_R$. In fact, $\Pic$ is a functor from the $\infty$-category of symmetric monoidal $\infty$-categories and symmetric monoidal functors to $\Top$. By \cite[Proposition 2.2.3]{ms_picard}, the functor $\Pic$ preserves filtered colimits and limits, therefore we have the following lemma:
\begin{lem} \label{lem:picultra}
The functor $\Pic$ preserves ultraproducts.
\end{lem}
Note that every suspension of $R$ is invertible. The invertible objects can be used to construct a third filtration on $\Mod_R$, which is coarser than $\Cell_R$.

\begin{defn}\label{defn:picfiltration}
The Pic-filtration $\mathrm{PicCell}_{R}$ on $\Mod_R$ is defined analogously to the cell filtration in \Cref{defn:cellfiltation}, but allowing arbitrary objects in $\Pic(R)$ as cells instead of suspensions of $R$. For a collection of module categories $(\Mod_{R_i})_{i \in I}$ equipped with this filtration and an ultrafilter $\cF$ on $I$, we will denote the protoproduct by
\[
\Prod{\cF}^{\Pic}\Mod_{R_i} = \Prod{\cF}^{\flat}(\Mod_{R_i},\mathrm{PicCell}_{R_i}). 
\] 
We will call this the $\Pic$-generated protoproduct of the module categories.
\end{defn}

\begin{rem}\label{rem:gencellfiltration}
More generally, we could construct a filtration in which the cells are taken from any submonoid of $\Pic(R)$ which is closed under suspensions. The filtrations $\Cell_R$ and $\mathrm{PicCell}_R$ are the minimal and maximal cases, respectively. 
\end{rem}

\begin{rem} \label{rem:picprotocats}
Assume that $(\cC_i)_{i \in I}$ is a collection of symmetric monoidal compactly generated $\infty$-categories with compact units. All three of the filtrations make sense in this more general setting. The obvious analogues of the cell filtration and cell-dimension filtration reduce to constructions in module categories as they only see the cellular objects. However, the Pic-filtration is interesting. Because the unit is compact, the invertible objects are compact. Thus we may define $\mathrm{PicCell}(\cC_i)$ to be the filtration with $k$th filtration step retracts of objects built out of at most $k$ invertible objects and define
\[
\Prod{\cF}^{\Pic}\cC_i = \Prod{\cF}^{\flat}(\cC_i,\mathrm{PicCell}(\cC_i)). 
\]
\end{rem}

\begin{lem}\label{lem:picprotogp}
Let $(\Mod_{R_i})_{i \in I}$ be a collection of module categories over $\mathbb{E}_\infty$-rings. Then
\[
\Pic(\Prod{\cF}^{\Pic}\Mod_{R_i}) \simeq \Prod{\cF} \Pic \Mod_{R_i}.
\]
\end{lem}
\begin{proof}
Since invertible objects of $\Mod_{R_i}$ are compact, we have $\Pic \Mod_{R_i} = \Pic \Mod_{R_i}^{\omega}$. It follows from the definition that 
\[
\mathrm{PicCell}_{R_i,0}\Mod_{R_i} \supseteq \Pic \Mod_{R_i}. 
\]
Thus the right hand side is contained in the left hand side.

To prove the other inclusion, first note that the unit of $\Prod{\cF}^{\Pic}\Mod_{R_i}$ is compact, so all invertible objects in this $\infty$-category are compact as well. Therefore, it suffices to consider 
\[
\Pic(\colim_k \Prod{\cF} \mathrm{PicCell}_{R_i,k}\Mod_{R_i}),
\]
where $\mathrm{PicCell}_{R_i,k}\Mod_{R_i}$ is the $k$-th filtration step in the Pic-filtration on $\Mod_{R_i}$. However, this is a subspace of $\Pic\Prod{\cF}\Mod_{R_i}^{\omega}$, which in turn is the same as $\Prod{\cF}\Pic \Mod_{R_i}$, as $\Pic$ preserves ultraproducts by \cref{lem:picultra}. 
\end{proof}

\subsection{Protoproducts of module categories}

We analyze the protoproducts associated to the filtrations of the previous section. Let $(R_i)_{i \in I}$ be a collection of $\mathbb{E}_\infty$-ring spectra. The identity maps on the $\infty$-categories $(\Mod_{R_i})_{i \in I}$ induce filtration preserving symmetric monoidal functors of stable symmetric monoidal filtered collections
\[
(\Mod_{R_i},\DimCell_{R_i}) \lra{} (\Mod_{R_i},\Cell_{R_i}) \lra{} (\Mod_{R_i},\mathrm{PicCell}_{R_i})
\]
inducing, by \cref{cor:symmonfun}, symmetric monoidal functors
\[
\xymatrixcolsep{5pc}
\xymatrix{\Prod{\cF}^{\doubleflat}\Mod_{R_i} \ar[r]^-{\Prod{\cF}^{\flat}\Id_{\Mod_{R_i}}} & \Prod{\cF}^{\flat}\Mod_{R_i} \ar[r]^-{\Prod{\cF}^{\flat}\Id_{\Mod_{R_i}}} & \Prod{\cF}^{\Pic}\Mod_{R_i}.}
\]
By \cref{lem:protoadjunction}, these functors are left adjoints. By \cref{cor:squares}, we have a commutative diagram in which the left adjoints commute with the left adjoints and the right adjoints commute with the right adjoints
\begin{equation} \label{eqn:adjunctions}
\xymatrix{& & \Prod{\cF}^{\omega}\Mod_{R_i} \ar@<0.5ex>[lld]^-{n} \ar@<0.5ex>[d]^-{n} \ar@<0.5ex>[drr]^-{n}  & & \\ \Prod{\cF}^{\doubleflat}\Mod_{R_i} \ar@<0.5ex>[urr]^-{\iota} \ar@<-0.5ex>[rr]_{\Prod{\cF}^{\flat}\Id_{\Mod_{R_i}}} & & \Prod{\cF}^{\flat}\Mod_{R_i} \ar@<-0.5ex>[rr]_{\Prod{\cF}^{\flat}\Id_{\Mod_{R_i}}} \ar@<0.5ex>[u]^-{\iota} \ar@<-0.5ex>[ll] & & \Prod{\cF}^{\Pic}\Mod_{R_i}. \ar@<-0.5ex>[ll] \ar@<0.5ex>[ull]^-{\iota}}
\end{equation}
The right adjoints are unital lax symmetric monoidal functors. Since the functors $\iota$ are fully faithful functors with right adjoints, the $\infty$-categories on the bottom row are all localizing subcategories of $\Prod{\cF}^{\omega} \Mod_{R_i}$. We will explicitly describe the right adjoints as colocalizations. 

In $\Prod{\cF}^{\omega}\Mod_{R_i}$ there are three natural notions of ``homotopy groups". Let $R = [R_i] \in \Prod{\cF}^{\omega}\Mod_{R_i}$ be the unit. For $[n_i] \in \Z^{\cF} = \Prod{\cF}\Z$, let
\[
\Sigma^{[n_i]}R = [\Sigma^{n_i}R_i] \in \Prod{\cF}^{\omega}\Mod_{R_i}.
\]

\begin{defn} Let $f \colon M \to N$ be a map in $\Prod{\cF}^{\omega}\Mod_{R_i}$. Then
\begin{itemize}
\item we say that $f$ is a $\pi_{[*]}$-equivalence if $[\Sigma^{[n_i]}R, M] \longrightarrow [\Sigma^{[n_i]}R,N]$ is an isomorphism for all $[n_i] \in \Z^{\cF}$. 

\item we say $f$ is a $\pi_{*}$-equivalence if $[\Sigma^{n}R, M] \longrightarrow [\Sigma^{n}R,N]$ is an isomorphism for all $n \in \Z \subset \Z^{\cF}$.

\item we say that $f$ is a $\pi_{\Pic}$-equivalence if $[L,M] \longrightarrow [L,N]$ is an isomorphism for all objects $L \in \Pic(\Prod{\cF}^{\omega}\Mod_{R_i})$.
\end{itemize}
\end{defn}

\begin{prop}\label{prop:protoproductgenerators}
The protoproduct $\Prod{\cF}^{\flat}\Mod_{R_i}$  is generated by $(\Sigma^{[n_i]}R)_{[n_i] \in \Prod{\cF}\Z}$ and is the colocalization of $\Prod{\cF}^{\omega}\Mod_{R_i}$ with respect to the $\pi_{[*]}$-equivalences. The bounded protoproduct $\Prod{\cF}^{\doubleflat}\Mod_{R_i}$ is generated by $(\Sigma^{n}R)_{n \in \Z}$ and is the colocalization of $\Prod{\cF}^{\omega}\Mod_{R_i}$ with respect to the $\pi_*$-equivalences. The $\Pic$-generated protoproduct $\Prod{\cF}^{\Pic}\Mod_{R_i}$ is generated by $L \in \Pic(\Prod{\cF}^{\omega}\Mod_{R_i})$ and is the colocalization of $\Prod{\cF}^{\omega}\Mod_{R_i}$ with respect to the $\pi_{\Pic}$-equivalences.
\end{prop}
\begin{proof}
The first part of the claim in each of these sentences implies the second part of the claim, because generators detect equivalences. We will show the claim for the protoproduct, the argument for the bounded protoproduct $\Prod{\cF}^{\doubleflat}\Mod_{R_i}$ and $\Pic$-generated protoproduct $\Prod{\cF}^{\Pic}\Mod_{R_i}$ being similar. 

By \Cref{lem:protoproductcompgen}, $\Prod{\cF}^{\flat}\Mod_{R_i}$ is compactly generated by its subcategory of compact objects, which by construction is $\colim_k\Prod{\cF}\Cell_{R_i,k}\Mod_{R_i}$. This $\infty$-category contains $\Sigma^{[n_i]}R$ for all $[n_i] \in \Z^{\cF}$, so it suffices to show that the thick subcategory $\Thick((\Sigma^{[n_i]}R)_{[n_i] \in \Prod{\cF}\Z})$ generated by these objects coincides with $\colim\Prod{\cF}\Cell_{R_i,k}\Mod_{R_i}$. To that end, consider an arbitrary non-trivial object $X \in \colim\Prod{\cF}\Cell_{R_i,k}\Mod_{R_i}$, so that there exists $k\ge 1$ with $X \in \Prod{\cF}\Cell_{R_i,k}\Mod_{R_i}$. By construction of the filtrations $\Cell_{R_i}$ this means that $X$ is a retract of a complex $Y$ built from at most $k$ cells in each coordinate $i \in I$. Therefore, it is possible to find a cell $\Sigma^{[n_i]}R$ and a cofiber squence
\[
\xymatrix{\Sigma^{[n_i]}R \ar[r] & Y' \ar[r] & Y,}
\]
such that the compact object $Y'$ is contained in 
\[
\Prod{\cF}\Cell_{R_i,k-1}\Mod_{R_i} \subset \Thick((\Sigma^{[n_i]}R)_{[n_i] \in \Prod{\cF}\Z}). 
\]
An induction on the number $k$ of cells then finishes the proof. 
\end{proof}

There is another way to describe the $\Pic$-generated protoproduct that is conceptually useful. We say that a symmetric monoidal compactly generated stable $\infty$-category $\cC$ is a $\Pic$-compactly generated $\infty$-category if 
\[
\Loc_{\cC}(\Pic(\cC)) = \cC.
\]
That is, the smallest localizing (stable) subcategory of $\cC$ containing $\Pic(\cC)$ is all of $\cC$.
Recall that we assume that the unit is compact in a symmetric monoidal compactly generated $\infty$-category, thus $\Pic(\cC) \subset \cC^{\omega}$. Let $\mathrm{Cat}_{\infty}^{\Pic}$ be the full subcategory of $\mathrm{Cat}_{\infty}^{\otimes,\omega,st}$ (the $\infty$-category of stable symmetric monoidal compactly generated $\infty$-categories) spanned by the $\Pic$-compactly generated $\infty$-categories. Given $\infty$-categories $\cC$ and $\cD$ in $\mathrm{Cat}_{\infty}^{\otimes,\omega,st}$ such that $\Loc_{\cC}(\Pic(\cC)) = \cC$, the canonical map $\Loc_{\cD}(\Pic(\cD)) \to \cD$ induces an equivalence
\[
\Map_{\mathrm{Cat}_{\infty}^{\otimes,\omega,st}}(\cC,\Loc_{\cD}(\Pic(\cD))) \lra{\simeq} \Map_{\mathrm{Cat}_{\infty}^{\otimes,\omega,st}}(\cC,\cD).
\]
Thus the functor $\Loc \Pic$ exhibits $\mathrm{Cat}_{\infty}^{\Pic}$ as a colocalization of $\mathrm{Cat}_{\infty}^{\otimes,\omega,st}$.

\begin{cor} \label{thisguy}
Let $(\cC_i)_{i \in I}$ be a collection of objects in $\mathrm{Cat}_{\infty}^{\Pic}$. The $\Pic$-generated protoproduct of \cref{rem:picprotocats},
\[
\Prod{\cF}^{\Pic}\cC_i,
\]
is the ultraproduct in the $\infty$-category of $\Pic$-compactly generated $\infty$-categories.
\end{cor}
\begin{proof}
The embedding 
\[
U \colon \mathrm{Cat}_{\infty}^{\Pic} \hookrightarrow \mathrm{Cat}_{\infty}^{\otimes,\omega,st}
\]
admits a right adjoint given by 
\[
\cC \mapsto \Loc_{\cC}(\Pic(\cC)) \simeq \Ind \Thick_{\cC} \Pic(\cC), 
\]
which preserves filtered colimits.

Thus for a collection of $\Pic$-compactly generated $\infty$-categories $\cC_i$, we have an equivalence
\[
\Prod{\cF}^{\Pic} \cC_i \simeq \Loc(\Pic(\Prod{\cF}^{\omega}U(\cC_i))).
\]
\end{proof}

\begin{rem}
\cref{thisguy} applies to the following situation: Let $(R_i)_{i \in I}$ be a collection of $\bE_{\infty}$-ring spectra and let $\cF$ be an ultrafilter on $I$. It follows that the $\Pic$-generated protoproduct
\[
\Prod{\cF}^{\Pic}\Mod_{R_i}
\]
is the ultraproduct in the $\infty$-category of $\Pic$-compactly generated $\infty$-categories.
\end{rem}

\begin{cor} \label{cor:evenrings}
Let $(R_i)_{i \in I}$ be a collection of ring spectra which are periodic of the same period. Then the adjunction of Diagram \ref{eqn:adjunctions} induces an equivalence
\[
\xymatrix{\Prod{\cF}^{\flat}\Mod_{R_i} \ar[r]^-{\simeq} & \Prod{\cF}^{\doubleflat}\Mod_{R_i}}
\]
for any ultrafilter $\cF$ on $I$. 
\end{cor}
\begin{proof}
As above, let $R = [R_i]$. Then this is a consequence of the fact that $\Sigma^{[n_i]}R \simeq \Sigma^k R$ for some $k$ between $0$ and the periodicity. This follows from the fact that $\Sigma^{n_i}R_i \simeq \Sigma^{k_i}R_i$ for $k_i$ between $0$ and the period. But now since the $k_i$'s are bounded, \cref{lem:finitepartition} implies that precisely one value can be supported on the ultrafilter.
\end{proof}

The next example highlights the difference between the protoproduct and the compactly generated ultraproduct. 

\begin{ex}
Suppose $\cF$ is a non-principal ultrafilter on $I = \N$. Let $R = R_i = H\Z$. We will construct two objects in $\Prod{\cF}^{\omega}\Mod_{R}$ and a map between them that is a $\pi_{[*]}$-equivalence but not an equivalence.

The first object is $Y = [\bigoplus_{l=0}^{i}\Sigma^lR]$, which is an object of $\Prod{\cF} \Mod_{R}^{\omega} \subset \Prod{\cF}^{\omega}\Mod_{R}$. 

The second object will be defined on a filtered diagram. Let $\mathfrak{A} = (A_i)_{i \in I}$ be a collection of subsets of $\N$ such that $A_i \subset \{1,\dots,i \}$ and $\max_i |A_i| < \infty$. We can define an order on all such collections by defining $\mathfrak{A} \leq \mathfrak{B}$ if 
$\forall i \in \mathbb{N}, \, A_i \subset B_i$. Denote the poset of all such collections by $\mathfrak{P}$. It follows from the definition that this is a directed poset.

For $\mathfrak{A} \in \mathfrak{P}$, define 
\[
X(\mathfrak{A})= [\bigoplus_{l \in A_i}\Sigma^lR]_{i \in I}.
\]
This construction extends to a functor
\[
X \colon \mathfrak{P} \to \Prod{\cF}\Mod^{\omega}_{R},
\]
and, since $\mathfrak{P}$ is directed, this gives us an object in $\Prod{\cF}^{\omega}\Mod_R$ by passing to the colimit along $X$.

There is a canonical map $f \colon X \rightarrow Y$ in $\Prod{\cF}^{\omega}\Mod_{R}$ induced by the canonical map of compact objects $X(\mathfrak{A}) \rightarrow Y$. The map $f$ is a $\pi_{[*]}$-equivalence, but not an equivalence. 

To see that it is a $\pi_{[*]}$-equivalence, we must show that
\[
[\Sigma^{[n_i]}R,X] \rightarrow [\Sigma^{[n_i]}R,Y]
\]
is an isomorphism for all $[n_i] \in \Z^{\cF}$. But
\[
[\Sigma^{[n_i]}R,X] \cong \Colim{\mathfrak{A} \in \mathfrak{P}}[\Sigma^{[n_i]}R,X(\mathfrak{A})]
\]
and this is isomorphic to $[\Sigma^{[n_i]}R,Y]$ because $\Sigma^{[n_i]}R$ has one cell.

To see that it is not an equivalence, note that $Y$ is compact, so an inverse $X \rightarrow Y$ to $f$ would need to factor through a finite stage. Thus $Y$ would be a retract of $X(\mathfrak{A})$ for some $\mathfrak{A} \in \mathfrak{P}$, but this is impossible as can be seen by applying $\pi_{[i]_{i \in \N}}$.
\end{ex}

Now we prove the main result of this subsection.

\begin{thm}\label{prop:protomodules}
Let $(R_i)_{i \in I}$ be a collection of $\bE_{1}$-ring spectra. Let $\cF$ be an ultrafilter on $I$. Then there is a canonical equivalence
\[
\Prod{\cF}^{\doubleflat} \Mod_{R_i} \simeq \Mod_{\Prod{\cF}R_i},
\]
where $\Prod{\cF}R_i$ denotes the ultraproduct of the spectra $(R_i)_{i \in I}$ in the $\infty$-category $\Sp$. In addition, if the ring spectra $R_i$ are $\bE_{\infty}$-rings, then the equivalence is an equivalence of symmetric monoidal $\infty$-categories.
\end{thm}
\begin{proof}
In order to identify the symmetric monoidal structures in the case that the ring spectra are $\bE_{\infty}$, we will first construct a symmetric monoidal comparison functor $\Phi\colon \Prod{\cF}^{\doubleflat}\Mod_{R_i} \to \Mod_{\Prod{\cF}R_i}$. To this end, note that forming ultraproducts induces a functor
\[
\xymatrix{\Prod{\cF}\colon \Prod{\cF} \Mod_{R_i}^{\omega} \ar[r] & \Sp,}
\]
which factors through $\Mod_{\Prod{\cF}R_i}$ because $\Prod{\cF}$ is lax symmetric monoidal. The universal extension of this functor to $\Prod{\cF}^{\omega}\Mod_{R_i}$ is unital lax symmetric monoidal, so upon restriction to $\Prod{\cF}^{\doubleflat}\Mod_{R_i}$ we obtain a unital lax symmetric monoidal functor
\[
\xymatrix{\Phi\colon \Prod{\cF}^{\doubleflat}\Mod_{R_i} \to \Mod_{\Prod{\cF}R_i}}
\]
between symmetric monoidal $\infty$-categories.  

We now claim that $\Phi$ is in fact symmetric monoidal. Indeed, fix an object $M \in \Prod{\cF}^{\doubleflat}\Mod_{R_i}$ and consider the full subcategory $\cC_M$ of $\Prod{\cF}^{\doubleflat}\Mod_{R_i}$ of all objects $N$ such that $\Phi(M\otimes N) \simeq \Phi(M) \otimes \Phi(N)$ via the given lax structure on $\Phi$. Since $\Phi$ is unital, the unit $\textbf{1}$ of  $\Prod{\cF}^{\doubleflat}\Mod_{R_i}$ belongs to $\cC_M$. Moreover, $\cC_M$ is closed under colimits, so $\cC_M = \Prod{\cF}^{\doubleflat}\Mod_{R_i}$ as $\Prod{\cF}^{\doubleflat}\Mod_{R_i}$ is compactly generated by $\textbf{1}$ by \Cref{prop:protoproductgenerators}. In other words, as $M$ runs through the objects of $\Prod{\cF}^{\doubleflat}\Mod_{R_i}$, we see that $\Phi$ is symmetric monoidal.

Invoking Morita theory (\cref{thm:moritatheory}), it therefore remains to identify the spectrum $\End(\textbf{1})$. Since the canonical functor
\[
\xymatrix{\Prod{\cF}^{\doubleflat} \Mod_{R_i} \ar[r] &  \Prod{\cF}^{\omega} \Mod_{R_i}}
\]
is symmetric monoidal and fully faithful by \Cref{lem:nsymmmon}, it sends $\textbf{1}$ to the unit $[R_i]$ in $\Prod{\cF}^{\omega} \Mod_{R_i}$, and we get an equivalence
\[
\End(\textbf{1}) \simeq \End_{\Prod{\cF}^{\omega} \Mod_{R_i}}([R_i]).
\]
This latter spectrum can be identified as 
\[
\End_{\Prod{\cF}^{\omega} \Mod_{R_i}}([R_i]) \simeq \End_{\Prod{\cF} \Mod_{R_i}^{\omega}}([R_i]) \simeq \Prod{\cF}\End_{\Mod_{R_i}}(R_i) \simeq \Prod{\cF} R_i
\]
by \Cref{lem:htpyultracompact} since $[R_i]$ is compact.
\end{proof}

\begin{rem}\label{rem:ultramodules}
Given a collection of $\bE_1$-ring spectra $R_i$ and $R_i$-module homomorphisms $f_i\colon M_i \to N_i$, then their ultraproduct 
\[
\xymatrix{\Prod{\cF}f_i\colon \Prod{\cF}M_i \ar[r] & \Prod{\cF}N_i}
\]
is canonically  an $\Prod{\cF}R_i$-module homomorphism. To see this, it suffices to observe that the ultraproduct is composed of an infinite product and a filtered colimit in $\Sp$, for which the claim is easily verified.
\end{rem}

We end this subsection with three examples and an application to chromatic homotopy theory.

\begin{ex}
Let $\cF$ be an ultrafilter on the set $I$. Consider the bounded protoproduct 
\[
\Prod{\cF}^{\doubleflat}\Mod_{H\Z}.
\]
If $\cF$ is principal on the set $\{i\}$, then this bounded protoproduct is equivalent to $\Mod_{H\Z}$. Now assume that $\cF$ is non-principal. In this case, \Cref{prop:protoproductgenerators} provides an equivalence
\[
\Prod{\cF}^{\doubleflat}\Mod_{H\Z} \simeq \Mod_{H\Z^{\cF}},
\]
where $\Z^{\cF}$ is the ring discussed in \Cref{ex:zf}. 
\end{ex}

\begin{ex}
Let $\mathrm{GrAb}$ be the $1$-category of graded abelian groups and let $\cF$ be an ultrafilter on the set of primes $\cP$. Since the Eilenberg--MacLane functor $H\colon \mathrm{GrAb} \to \Sp$ preserves products and filtered colimits, there is a natural equivalence
\[
\Prod{\cF}(HM_p) \simeq H(\Prod{\cF}M_p)
\]
for any collection of abelian groups $(M_p)_{p\in \cP}$. In particular, let $R_p = H\F_p[x]$ with $x$ of degree $2$. Because of the grading, the ultraproduct of the $R_p$'s recovers the Eilenberg--MacLane spectrum of a graded version of the protoproduct of \Cref{ex:protopoly}.
\end{ex}

\begin{ex}\label{cor:bprotoplocalspectra}
Let $\cF$ be a non-principal ultrafilter on $\cP$ and let $\Sp_{(p)}$ be the $\infty$-category of $p$-local spectra. Then there is a natural equivalence
\[
\Prod{\cF}^{\doubleflat}\Sp_{(p)} \simeq \Mod_{H\Z_{(\cF)}},
\]
where $\Z_{(\cF)}$ is the $\Q$-algebra $\Prod{\cF} \Z_{(p)}$ (similar to \Cref{ex:fadics}). It follows from \Cref{prop:protoproductgenerators} that 
\[
\Prod{\cF}^{\doubleflat}\Sp_{(p)} \simeq \Mod_{\Prod{\cF}S_{(p)}^0},
\]
and it remains to give a more explicit description of $\Prod{\cF}S_{(p)}^0$. Since the homotopy groups of $S_{(p)}^0$ are zero in degrees $[1,2p-4]$, we obtain an equivalence
\[
\Prod{\cF}S_{(p)}^0 \simeq \Prod{\cF}H\Z_{(p)} \simeq H\Z_{(\cF)}. 
\]
\end{ex}

Finally, we apply several of the ideas of this subsection to the case that we fix a prime $p$ and consider the protoproduct of the categories of modules of $K(n)$ at an ultrafilter on the natural numbers.

\begin{thm}\label{thm:heightinfinity}
Let $\cF$ be a non-principal ultrafilter on the natural numbers $\N$, let $PK(n)$ be $2$-periodic Morava $K$-theory of height $n$ at the prime $p$, and let $PH\F_p$ be $2$-periodic singular cohomology with coefficients in $\F_p$. There is a natural equivalence
\[
\xymatrix{\Prod{\cF}^{\flat}\Mod_{PK(n)} \ar[r]^-{\simeq} & \Mod_{PH\F_p}}
\]
of compactly generated stable $\infty$-categories. 
\end{thm}
\begin{proof}
Because all theories involved are 2-periodic, \cref{cor:evenrings} and \cref{prop:protomodules} give an equivalence
\[
\Prod{\cF}^{\flat}\Mod_{PK(n)} \simeq \Mod_{\Prod{\cF}PK(n)}.
\]
We will produce an equivalence of ring spectra
\[
\Prod{\cF}PK(n) \simeq PH\F_p.
\]
Note that there is a map of $\bE_1$-algebras
\[
K(n) \rightarrow PK(n)
\]
inducing a map of $\bE_1$-algebras
\[
\Prod{\cF}K(n) \rightarrow \Prod{\cF}PK(n).
\]
As the periodicity of $\pi_*K(n)$ strictly increases as the height $n$ increases, we see that
\[
\pi_*(\Prod{\cF}K(n))
\]
is concentrated in degree $0$ and $\pi_0(\Prod{\cF}K(n)) \cong \F_p$, thus
\[
\Prod{\cF}K(n) \simeq H\F_p.
\]
Recall that a ring spectrum $R \in \Sp$ is called a field object if $\pi_*R$ is a graded field in the algebraic sense. As a result of the nilpotence theorem, Hopkins and Smith \cite{nilp2} have classified the minimal field objects in $\Sp$: They are the Eilenberg--MacLane spectra $H\F_p$ as well as the Morava $K$-theories $K(n)$ for all $n$ and $p$. 

Since ultraproducts commute with homotopy groups, it follows that $\Prod{\cF}PK(n) \in \Sp$ is a field object. Since it is an $H\F_p$-algebra and $\pi_{2i} \Prod{\cF}PK(n) \cong \F_p$, there is an equivalence of ring spectra
\[
\Prod{\cF}PK(n) \simeq PH\F_p.
\]
\end{proof}

\begin{rem}
We would like to point out one curious consequence of this result: In the course of the proof we produced equivalences of ring spectra of the form $\Prod{\cF}PK(n) \simeq PH\F_p$ and $\Prod{\cF}K(n) \simeq H\F_p$. However, $PK(n)$ and $K(n)$ are not $\bE_2$ for any $n>0$ or $p$. We might therefore interpret this as saying that $PK(n)$ and $K(n)$ become more commutative as $n \to \infty$, without moving through the usual hierarchy of $\bE_m$-operads.
\end{rem}

\section{Formality} \label{sec:formality}

Let $\mathrm{GrAb}$ be the symmetric monoidal category of $\Z$-graded abelian groups with grading preserving maps. The functor
\[
\pi_* \colon \Sp \rightarrow \mathrm{GrAb}
\]
admits a right inverse
\[
H \colon \mathrm{GrAb} \to \Sp,
\]
which commutes with products, called the generalized Eilenberg--MacLane spectrum functor. Note that $H$ is lax symmetric monoidal. As a lax functor $H$ induces a functor from $\CAlg(\mathrm{GrAb})$ to $\CAlg(\Sp)$ which we will still denote by $H$. For any $\infty$-category $D$ an object in $\CAlg(\Sp)^{D}$ is called formal if it is in the image of 
\[
H^D\colon \CAlg(\mathrm{GrAb})^D \to \CAlg(\Sp)^D.
\]

Let $\E$ be a Morava $E$-theory at height $n$ and the prime $p$. Let $\E^{\otimes \bullet+1}$ be the Amitsur complex of the unit map $S^0 \rightarrow \E$, defined carefully in \cref{ss:reduce_formal}. The goal of this section is to prove that the cosimplicial spectrum
\[
\Prod{\cF} \big ( \E^{\otimes \bullet+1} \big ) \in  \CAlg(\Sp)^\Delta.
\]
is formal. This is the first step in showing that the distinction between spectral and algebraic data disappears at a non-principal ultrafilter. 

To explain the idea of the proof first consider the formality of  
\[
E_{n,\cF}= \Prod{\cF} \big ( \E \big ).
\]
The $\bE_{\infty}$-ring $\E$ admits an action of $C_{p-1}$ such that the induced action on $\pi_{2j} \E$ is of weight $j$. Decomposing $\E$ according to weights, we deduce that any non-trivial $k$-invariant of $\E$ can only appear in degrees divisible by $2(p-1)$. As the prime $p$ goes to infinity  the non-trivial $k$-invariants appear sparser and sparser and in the limit they do not appear at all. 

To apply this idea to
$\Prod{\cF}  \E^{\otimes \bullet+1}, $
two issues need to be addressed:
\begin{enumerate}
\item For $k \geq 2$, the $C_{p-1}$ action on $\E^{\otimes k}$  is not of a single weight for every homotopy group.
\item The formality needed in the cosimplicial case needs to include not only the formality at every cosimplicial degree but also all of the coherence data in the diagram as well as the algebra structure.
\end{enumerate}

It turns out that the first issue disappears after replacing the Amitsur complex of the unit map $S^0 \rightarrow \E$ by that of the map
$\hS \to \E$, where $\hS$ is the $p$-complete sphere. 

To aid the reader we provide a brief outline of this section. Sections \ref{ss:H} to \ref{ss:ratio_formal} give a reduction to the Amitsur complex of $\hS \to \E$. In \cref{ss:H} we establish basic general properties of the functor $H$. We use the arithmetic fracture square in \cref{ss:reduce_formal} to reduce the formality of $\Prod{\cF}  \E^{\otimes \bullet+1} $ to the formality of a diagram built out of the rationalization of $\E$ and $\Prod{\cF}  \E^{\otimes_{\hS} \bullet+1} $. The required results regarding $\Q \otimes \E$ are proved in \cref{ss:ratio_formal} using Andre--Quillen obstruction theory. We are now left with the need to prove the formality of $\Prod{\cF}  \E^{\otimes_{\hS} \bullet+1}$ while taking issue (2) into account. 

In \cref{ss:cp_1} we prove that the action of $C_{p-1}$ on  $\Prod{\cF}  \E^{\otimes_{\hS} \bullet+1} $ has the desired properties. In \cref{ss:monoid} we introduce operadic tools that allow us to explore the interplay between symmetric monoidal structures and weight decompositions. We employ these tools in \cref{ss:weight} to produce the  weight decomposition of $\Prod{\cF}  \E^{\otimes_{\hS} \bullet+1} $. 
Finally, in \cref{ss:formal_infty} we collect all of the results at each finite prime to obtain the formality at the non-principal ultrafilters.

\subsection{Properties of $H$}\label{ss:H}

In this subsection we establish basic properties of the functor $H \colon \mathrm{GrAb} \to \Sp$ that will help us deduce the formality of the cosimplicial spectrum
\[
\Prod{\cF} \big ( \E^{\otimes \bullet+1} \big ) \in  \CAlg(\Sp)^\Delta.
\]

The following proposition follows from the analogous fact for $\pi_*$.

\begin{prop} \label{prop:Hcommutes}
The functor $H$ commutes with products and filtered colimits and therefore with ultraproducts.
\end{prop}

\begin{defn}\label{defn:star}
Let $X$ be a spectrum. Define
\[
X_{\star} = H\pi_*(X).
\]
\end{defn}

Since $H$ and $\pi_*$ are lax symmetric monoidal functors, if $R$ is an $\bE_n$-ring spectrum, then $R_{\star}$ is an $\bE_n$-ring spectrum ($n = \infty$ is allowed). In fact, since $\mathrm{GrAb}$ is a $1$-category, if $n \geq 2$, then $R_{\star}$ automatically has the structure of an $\bE_{\infty}$-ring spectrum.

These ring spectra behave in an algebraic way, as can be seen in the following theorem of Schwede and Shipley (\cite{schwedeshipleymorita}, \cite[Section 7.1.2]{ha}). For a ring spectrum $R$, let $C(R)$ be the differential graded algebra in which the chain groups are precisely the homotopy groups of $R$ and the differentials are all the zero map.

\begin{thm} (Schwede--Shipley)
Let $\mathcal{D}(\mathrm{Ab})$ be the $\infty$-category of chain complexes of abelian groups and let $R$ be an $\bE_{\infty}$-ring spectrum. There is a symmetric monoidal equivalence of stable $\infty$-categories
\[
\Mod_{R_\star} \simeq \Mod_{C(R)}(\mathcal{D}(\mathrm{Ab})).
\]
\end{thm}

Thus we may think about $R_{\star}$ as a differential graded algebra.

\begin{prop}
Let $D$ be a small category and let 
\[
\xymatrix{P \ar[r] \ar[d] & B \ar[d] \\ A \ar[r] & C}
\]
be a pullback diagram in $\CAlg(\mathrm{GrAb}^{D})$. If for every object $d \in D$, the map $B(d) \oplus A(d) \rightarrow C(d)$ is surjective, then
\[
H^D(P) \rightarrow H^D(B) \times_{H^D(C)} H^D(A)
\]
is an equivalence in $\CAlg(\Sp^{D})$ and thus the pullback $H^D(B) \times_{H^D(C)} H^D(A)$ is formal.
\end{prop}
\begin{proof}
Consider the long exact sequence of homotopy groups for the pullback $H^D(B) \times_{H^D(C)} H^D(A)$. The surjectivity of the map
\[
B(d) \oplus A(d) \rightarrow C(d)
\]
implies that the homotopy groups of the pullback are exactly the pullback of the homotopy groups. Thus the canonical map $H^D(P) \rightarrow H^D(B) \times_{H^D(C)} H^D(A)$ is an equivalence.

\end{proof}

\begin{cor}\label{c:pullback_formal}
Let $D$ be a small category and let
\[
\xymatrix{ P\ar[d]\ar[r] & B \ar[d] \\ A \ar[r] & C}
\]
be a pullback diagram in $\CAlg(\Sp)^{D}$.
Assume that $(A \to C)$ and $(B \to C)$ are formal as objects in $\CAlg(\Sp)^{D \times \Delta^1}$ and that for every $d\in D$ and $i\in \Z$ the map 
\[
\pi_i(A(d)) \oplus \pi_i(B(d)) \to \pi_i(C(d))
\]
is surjective, 
then $P$ is formal as an object in $\CAlg(\Sp)^{D}$.
\end{cor}

\begin{lem}\label{l:formal_q}
Let $D$ be a small category and let $R \in \CAlg(\Sp)^D$ be formal. Let 
\[
\big ( R \to \Q \otimes R \big ) \in \CAlg(\Sp)^{D\times \Delta^1}
\] 
be the canonical map to the rationalization. Then $R \to \Q \otimes R$ is formal as an object in $\CAlg(\Sp)^{D\times \Delta^1} $.
 
\end{lem}
\begin{proof}

By assumption 
$R \simeq H^{D}\pi_{*}^{D}(R)$.
We also have an object 
\[
(\pi_{*}^{D}(R) \to \pi_{*}^{D}(R) \otimes \Q) \in \CAlg(\mathrm{GrAb})^{D\times \Delta^1}.  
\]
By the laxness of $H^{D}$ there is a commutative square
\[
\xymatrix{
H^D\pi_{*}^{D}(R) \ar[r] \ar[d]& H^D\pi_{*}^{D}(R) \otimes H^D\Q  \ar[d] \\
H^D\pi_{*}^{D}(R) \ar[r] & H^D(\pi_{*}^{D}(R) \otimes \Q),  
}
\]
where the right vertical map is an equivalence by the flatness of $\Q$ over $\pi_{*}(S^0)$ and $H^D\Q$ is just $H^D$ applied to the constant diagram. Thus the fact that the bottom horizontal map is formal implies that the top horizontal map is formal.
\end{proof}

Let $\Sp_{\Q}[0,0]$ be the full symmetric monoidal subcategory of $\Sp_{\Q}$ on objects whose nontrivial homotopy groups are concentrated in degree $0$. Also, let $\mathrm{GrAb}_{\Q}$ be the category of graded $\Q$-vector spaces and let ${\Q}-\Mod$ be the symmetric monoidal category of rational vector spaces viewed as a full subcategory of $\mathrm{GrAb}$ on degree $0$ objects.

\begin{lem}\label{l:rational-heart}
The restriction of the lax symmetric monoidal functor 
\[
H^D \colon \mathrm{GrAb}^D \rightarrow \Sp^D
\]
to
\[
H^D_{\Q} \colon \mathrm{GrAb}_{\Q}^D \rightarrow \Sp_\Q^D
\]
is symmetric monoidal. Restricting further to ${\Q}-\Mod^D$ induces a symmetric monoidal equivalence
\[
H^D|_{{\Q}-\Mod^{D}} \colon {\Q}-\Mod^{D} \lra{\simeq} \Sp_{\Q}[0,0]^{D}.
\]
\end{lem}
\begin{proof}
Both statements reduce to the case $D = *$.
The first statement follows from the fact that every $\Q$-module is flat by using the K{\"u}nneth spectral sequence. The equivalence  
\[
{\Q}-\Mod \lra{\simeq} \Sp_{\Q}[0,0]
\]
is classical.
\end{proof}

\subsection{The reduction to the rational and $p$-complete cases}\label{ss:reduce_formal}

Let $\hS$ be the $p$-complete sphere spectrum and fix a non-principal ultrafilter $\cF$ on $I = \cP$. In this subsection we use the arithmetic fracture square to reduce the verification of the formality of the cosimplicial spectrum 
\[
\Prod{\cF}\E^{\otimes \bullet+1}
\]
to the formality of the cosimplicial spectrum 
\[
\Prod{\cF} \E^{\otimes_{\hS} {\bullet +1}} 
\]
and the formality of the canonical map of cosimplicial spectra
\[
\Prod{\cF}\Q \otimes \E^{\otimes {\bullet +1}} \to  \Prod{\cF}\Q \otimes \E^{\otimes_{\hS} {\bullet +1}}.
\]

\begin{rem}
The ultraproducts in this section are being taken in a variety of different categories. But all of these categories are either diagram categories or categories of commutative algebra objects in diagram categories. Since ultraproducts commute with the forgetful functor and restriction of diagrams, there should be no ambiguity regarding what is meant by the symbol $\Prod{\cF}$.
\end{rem}

Given a spectrum $M$ and a map of spectra $f \colon A \to B$, we say that $f$ is an $M$-equivalence if $f \otimes M$ is an equivalence of spectra. Note that if $f$ and $g$ are $M$-equivalences, then $f \otimes g$ is an $M$-equivalence.

\begin{lem} \label{completetensor0}
Let $M$ be a spectrum and let $f \colon R\to S$ be a map of $\mathbb{E}_1$-spectra which is an $M$-equivalence. Let $A$ and $B$ be right and left $S$-modules. The map 
\[
A \otimes_R B  \to A \otimes_{S} B  
\]
is an $M$-equivalence.
\end{lem}
\begin{proof}
Recalling the definition of the relative tensor product as a geometric realization (bar construction) and using the fact that $\otimes$ commutes with colimits in each variable, it is enough to show that for every $k \geq 0$ the map
\[
A \otimes R^{\otimes k} \otimes B \to A \otimes {S}^{\otimes k} \otimes B 
\]
is an $M$-equivalence. This follows from the fact that the tensor product of $M$-equivalences is an $M$-equivalence.
\end{proof}

\begin{cor}\label{completetensor}
For every $k \geq 1$, the map
\[
\E^{\otimes k} \to \E^{\otimes_{\hat{S}}k }
\]
is an equivalence after $p$-completion 
\end{cor}
\begin{proof}
Recall that $p$-completion is localization with respect to the Moore spectrum $M(p)$. The map $S^0 \to \hat{S}$ is an $M(p)$-equivalence so the result follows from the previous lemma.
\end{proof}

Let $\cC$ be a presentably symmetric monoidal $\infty$-category and let $A \rightarrow B$ be a map in $\CAlg(\cC)$. In this situation, $B$ may be considered as an object in $\CAlg_{A}(\cC)$. Evaluating at the object $[0] \in \Delta$
gives a functor 
\[
\CAlg_{A}(\cC)^{\Delta} \to \CAlg_{A}(\cC)
\]
This functor admits a left adjoint 
\[
L \colon \CAlg_{A}(\cC) \to \CAlg_{A}(\cC)^{\Delta}
\]
such that $L(B)([n]) \simeq B^{\otimes_A( n+1)}$. We call $L(B)$ the Amitsur complex of $B$ over $A$, the cosimplicial object
\[
\xymatrix{B \ar@<0.5ex>[r] \ar@<-0.5ex>[r] & B \otimes_A B \ar[l] \ar@<-1ex>[r] \ar[r]  \ar@<1ex>[r] & B \otimes_A B \otimes_A B \ar@<-0.5ex>[l] \ar@<0.5ex>[l]  \ar@<-2ex>[r] \ar@<-1ex>[r] \ar[r] \ar@<1ex>[r] \ar@<2ex>[r] & \cdots. \ar@<-1.5ex>[l] \ar@<-0.5ex>[l] \ar@<0.5ex>[l] \ar@<1.5ex>[l]}
\]
By abuse of notation we give the same name to the image of $L(B)$ under the forgetful functor
\[
\CAlg_{A}(\cC)^{\Delta} \to \CAlg(\cC)^{\Delta}.
\] 
We denote this image by $B^{\otimes_A \bullet+1}$.

\begin{lem}\label{p:fracture}
There is a pullback square in $\CAlg(\Sp)^{\Delta}$
\[\xymatrix{
\E^{\otimes {\bullet +1}} \ar[r] \ar[d] & \E^{\otimes_{\hS} {\bullet +1}} \ar[d] \\ \Q \otimes \E^{\otimes {\bullet +1}} \ar[r] & \Q \otimes \E^{\otimes_{\hS} {\bullet +1}}.
}\]
\end{lem}
\begin{proof}
Consider the diagram of cosimplicial spectra 
\[\xymatrix{
\E^{\otimes {\bullet +1}} \ar[r] \ar[d] & \E^{\otimes_{\hS} {\bullet +1}} \ar[d]  \ar[r]& \left ( \E^{\otimes {\bullet +1}}\right)_p^{\wedge} \ar[d]\\ 
\Q \otimes \E^{\otimes {\bullet +1}} \ar[r] & \Q \otimes \E^{\otimes_{\hS} {\bullet +1}} \ar[r]& \Q \otimes \left( \E^{\otimes {\bullet +1}} \right)_p^{\wedge}.
}\]
At cosimplicial degree $k$, the outer square is the arithmetic fracture square associated to $E^{\otimes k +1}$, thus it is a pullback square. We show that the right-hand square is a pullback square. By \cref{completetensor} at cosimplicial degree $k$ it is the fracture square attached to $E^{\otimes_{\hat{S}} k +1}$. The claim now follows from the pasting lemma.
\end{proof}

\begin{cor}\label{cor:Squere}
There is a pullback square in $\CAlg(\Sp)^{\Delta}$
\[\xymatrix{
\Prod{\cF}\E^{\otimes {\bullet +1}} \ar[r] \ar[d] & \Prod{\cF}\E^{\otimes_{\hS} {\bullet +1}} \ar[d] \\ \Prod{\cF}\Q \otimes \E^{\otimes {\bullet +1}} \ar[r] & \Prod{\cF}\Q \otimes \E^{\otimes_{\hS} {\bullet +1}}.
}\]
\end{cor}
\begin{proof}
This follows from \cref{p:fracture} plus the fact that ultraproducts commute with finite limits in $\Sp$.
\end{proof}

The next two theorems follow from the main results of \cref{ss:ratio_formal} and \cref{ss:formal_infty}.

\begin{thm}\label{l:formalh}
The canonical diagram in $\CAlg(\Sp)^{\Delta \times \Delta^1}$
\[
\Prod{\cF}\Q \otimes \E^{\otimes {\bullet +1}} \to  \Prod{\cF}\Q \otimes \E^{\otimes_{\hS} {\bullet +1}} \]
is formal.
\end{thm}
\begin{proof}
We prove in \cref{cor:formalp} that this is true at every finite prime $p$. Therefore, we are done as \cref{prop:Hcommutes} implies that the functor $H$ commutes with ultraproducts.
\end{proof}

\begin{thm}\label{l:formalv1}
The cosimplicial $\bE_{\infty}$-ring
\[
\Prod{\cF} \E^{\otimes_{\hS} {\bullet +1}} 
\]
is formal.
\end{thm}
\begin{proof}
This is \cref{prop:cosimplicialformalobj}.
\end{proof}

\begin{lem}\label{l:formalv}
The canonical diagram in $\CAlg(\Sp)^{\Delta \times \Delta^1}$
\[
\Prod{\cF} \E^{\otimes_{\hS} {\bullet +1}} \rightarrow \Prod{\cF} \Q \otimes \E^{\otimes_{\hS} {\bullet +1}}
\]
is formal.
\end{lem}
\begin{proof}
This follows from \cref{l:formalv1} by \cref{l:formal_q}.
\end{proof}

\begin{lem}\label{l:tor}
For every  $k\geq 1$, $\pi_*(\E^{\otimes k})$ is torsion-free.
\end{lem}
\begin{proof}
It follows inductively from \cite[Proposition 2.16]{hoveystrickland} that $\pi_*(\E^{\otimes k})$ is flat over $\pi_*(\E)$. On the other hand, $\pi_*(\E)$ is flat as a $\Z_p$-module. Thus $\pi_*(\E^{\otimes k})$ is flat as a $\Z_p$-module and therefore torsion-free.
\end{proof}

\begin{lem}\label{l:surj_alg}
Let $f \colon A \to B$ be a map of abelian groups that induces a surjection on the quotient $A/nA \to B/nB$ for every $n \in \mathbb{N}$. Then the map 
\[
B \oplus \Q \otimes A \to \Q \otimes B 
\]
is surjective.
\end{lem}
\begin{proof}
Let $\frac{1}{n} \otimes b \in \Q \otimes B$.
Since $ A/nA \to  B/nB$ is surjective, we have some $a\in A$, $b_0 \in B$ such that $b = f(a) +nb_0$.
It follows that $\frac{1}{n} \otimes b = \frac{1}{n} \otimes f(a) + 1\otimes b_0. $
\end{proof}

\begin{lem}\label{l:surj_sp}
Let $A \to B$ be a map of $p$-local spectra that is an equivalence after $p$-completion and assume that $\pi_i(A)$ is torsion-free for all $i \in \Z$. Then the map
\[
\pi_i(B \oplus \Q \otimes A) \to \pi_i(\Q \otimes B) .
\]
is surjective.
\end{lem}
\begin{proof}
Let $S^0/p^m$ be the Moore spectrum.
The map $A \otimes S^0/p^m \to B \otimes S^0/p^m$ is an equivalence for $m=1$ by the definition of $p$-completion and for higher 
$m$ by induction using cofiber sequences.  
 Consider the following diagram: 
\[
\xymatrix{
A \ar[r]^{\times p^m}\ar[d]& A\ar[d] \ar[r] & A \otimes S^0/p^m \ar[d]^{\simeq} \\
B \ar[r]^{\times p^m} & B  \ar[r]& B \otimes S^0/p^m.
}
\]
In this diagram the two rows are cofiber sequences.
Taking  the corresponding map of long exact sequences in homotopy groups  and keeping in mind $\pi_i(A)$ is torsion-free, we get the following diagram in which both rows are exact:
\[
\xymatrix{
0\ar[r]& \pi_i(A)/p^m\pi_i(A)  \ar[r]^{\cong} \ar[d] & \pi_i(A \otimes S^0/p^m) \ar[r] \ar[d]^{\cong}& 0 \ar[d] & \empty   \\
0\ar[r]& \pi_i(B)/p^m\pi_i(B)  \ar[r]          & \pi_i(B \otimes S^0/p^m) \ar[r]    & \pi_i(B)[p^m]   \ar[r]& 0.   \\
}
\]
We thus conclude that $\pi_i(B)$ is torsion-free and that $ \pi_i(A)/p^m\pi_i(A) \to  \pi_i(B)/p^m\pi_i(B)$ is an isomorphism.
Since $A$ and $B$ are $p$-local, we get that $ \pi_i(A)/n\pi_i(A) \to  \pi_i(B)/n\pi_i(B)$ is an isomorphism for every $n \in \N$.
\cref{l:surj_alg} gives the conclusion.
\end{proof}

\begin{prop}\label{p:surj}
For each $k$, $i$, and $n$, the canonical map
\[
\pi_i( \Prod{\cF}\E^{\otimes_{\hS} k} \oplus   \Prod{\cF}\Q \otimes \E^{\otimes k} ) \to \pi_i(\Prod{\cF}\Q \otimes \E^{\otimes_{\hS} k } )
\]
is surjective.
\end{prop}
\begin{proof}
Since ultraproducts commute with homotopy groups and preserve surjections it is enough to show that for each $k$, $i$, $p$, and $n$, the canonical map
\[
\pi_i( \E^{\otimes_{\hS} k} \oplus  \Q \otimes \E^{\otimes k} ) \to \pi_i(\Q \otimes \E^{\otimes_{\hS} k } )
\]
is surjective. This follows from \cref{completetensor}, \cref{l:tor}, and \cref{l:surj_sp}.
\end{proof}

The main result of this section is the following theorem.

\begin{thm} \label{thm:sec4thm}
The object $\Prod{\cF}\E^{\otimes \bullet +1}$ in $\CAlg(\Sp)^{\Delta}$ is formal. That is, there is an equivalence of cosimplicial $\bE_{\infty}$-rings
\[
\Prod{\cF}\E^{\otimes \bullet +1} \simeq \Prod{\cF}(\E^{\otimes \bullet +1})_{\star}.
\] 
\end{thm}

\begin{proof}
Apply \cref{c:pullback_formal} to the pullback diagram in \cref{cor:Squere}. The pullback diagram satisfies the condition of \cref{c:pullback_formal} by \cref{l:formalh}, \cref{l:formalv}, and \cref{p:surj}.
\end{proof}

\subsection{Rational formality}\label{ss:ratio_formal} 

The purpose of this subsection is to prove \cref{l:formalh}. The proof is an application of obstruction theory to commutative differential graded algebras over a characteristic $0$ field $k$.

\begin{lem}\label{l:Lur}
Let $E_{\Q} = \Q \otimes \E$ be the rationalisation of Morava $E$-theory. Then there exists an $\bE_{\infty}$-ring map $f_p \colon H\pi_0(E_{\Q}) \to E_{\Q}$ inducing an isomorphism on $\pi_0$.
\end{lem}
\begin{proof}
The existence of the map $f_p$ follows from the construction of $\E$ that appears in \cite{LurE}. Specifically, given a perfect field $k$ of characteristic $p$ and $G_0$ a height $n$ formal group law over $k$, let $\E$ be the Morava $E$-theory associated to $G_0$. In \cite[Theorem~3.0.11 and Remark~3.0.14]{LurE} Lurie constructs an $\bE_{\infty}$-ring $R^{un}_{G_0}$ with $\pi_0(R^{un}_{G_0})$ canonically identified with the Lubin--Tate ring $\pi_0(\E)$. Then, given the same data, in Section 6 of \cite{LurE} Lurie  constructs another $\bE_{\infty}$-ring $R^{or}_{G_0}$ together with an $\bE_{\infty}$-ring map 
\[
R^{un}_{G_0} \to R^{or}_{G_0},
\]
see in particular \cite[Construction 6.0.1 and Remark 6.0.2]{LurE}. We can now identify the required map $f_p$ as
\[
H\pi_0(E_{\Q}) \cong \Q \otimes R^{un}_{G_0} \to \Q \otimes R^{or}_{G_0} \cong E_{\Q}
\] 
due to  the following facts:
\begin{enumerate}
	\item The map $R^{un}_{G_0} \to R^{or}_{G_0}$ induces an isomorphism on $\pi_0$, by \cite[Theorem 6.0.3]{LurE}.
	\item The $\bE_{\infty}$-ring $R^{or}_{G_0}$ can be identified with $\E$, by \cite[Theorem 5.1.5 and Remark 6.4.8]{LurE}.
	\item The $\bE_{\infty}$-ring $\Q \otimes  R^{un}_{G_0}$ is concentrated in degree $0$, by \cite[Theorem 6.3.1]{LurE}.\qedhere
\end{enumerate}
\end{proof}

\begin{rem}
For a fixed height $n$, we only need the existence of the map $f_p$ for large enough $p$. Using this observation, it is possible to replace the argument of \Cref{l:Lur} by an obstruction theoretic one. Indeed, in \Cref{ss:cp_1} we construct a $C_{p-1}$ action on $\E$ such that the $\bE_{\infty}$-ring map
\[
\E^{hC_{p-1}} \to \E
\]
induces an isomorphism on homotopy groups in degrees dividing $2(p-1)$ and such that $\pi_i(\E^{hC_{p-1}})$ is zero in all other degrees, see \Cref{r:sparse}. Let $B$ be the connective cover of $(\E^{hC_{p-1}})_{\Q}$. Then it is enough to show that there is an 
$\bE_{\infty}$-ring map
\[
H\pi_0(B) \to B
\]
inducing an isomorphism on $\pi_0$. This problem is amenable to standard techniques in obstruction theory; for an $\infty$-categorical treatment, see \cite[Section 7.4.1]{ha}. Specifically, applying \cite[Remark 7.4.1.29]{ha} to $\Mod_{H\Q_p}$ instead of $\Sp$, using the $\bE_{\infty}$-operad, and setting $A = H\pi_0(B)$, we get that the obstructions to the existence ($\epsilon =1$) and uniqueness ($\epsilon =0$) of the required map lie in 
\[
O_d^{\epsilon} = \Ext^{d+\epsilon}_{\pi_0(B)}(\mathbb{L}_{\pi_0(B)/\Q_p}, \pi_d(B)).
\]
Since $\pi_0(B)$ is regular, $\mathbb{L}_{\pi_0(B)/\Q_p}$ is concentrated in degree zero (see, for example, \cite{iyengar} Proposition 5.9 and Theorem 9.5) and thus as $\pi_0(B)$ is regular of Krull dimension $n-1$ the groups $O_d^{\epsilon}$ are all zero if $p > \frac{n+1}{2}$ and $\epsilon \in \{0,1\}$. 
\end{rem}

\begin{lem}\label{l:k-form}
Let $R$ be a commutative algebra in $\Mod_{Hk}$ and assume that there exists a map $H\pi_0(R) \to R$ which  induces an isomorphism on $\pi_0$  and that $\pi_*(R) \cong \pi_0(R)[\beta^{\pm 1}]$ for some $\beta \in \pi_2(R)$. Then $R$ is formal as an $Hk$-algebra.
\end{lem}
\begin{proof}
We have map  $H\pi_0(R) \to R$ that is an isomorphism on $\pi_0$. Now $\beta^{\pm} \in \pi_{\pm 2}(R)$ 
corresponds to maps $Hk^{\pm 2} \to R$.  By the free-forgetful adjunction between $\Mod_{Hk}$ and $\Mod_{H\pi_0(R)}$, we get maps of $H\pi_0(R)$-modules
\[
\alpha^{\pm} \colon \Sigma^{\pm 2} H\pi_0(R) \to R.
\]
By the free-forgetful adjunction between algebras and modules, we get two $H\pi_0(R)$-algebra maps
\[
\gamma^{\pm} \colon H\pi_0(R)[\beta^{\pm}] \to R.
\]
The map $\gamma^{+}$ (resp. $\gamma^{-}$) is an isomorphism on positive (resp. negative) homotopy groups.

The following diagram is a pushout diagram (as can be seen by applying $\pi_*$):
\[
\xymatrix{
H\pi_0(R)[\beta \otimes \beta^{-1}]\ar[r] \ar[d]& H\pi_{0}(R) \ar[d] \\
H\pi_0(R)[\beta]  \otimes_{H\pi_0(R)} H\pi_0(R)[\beta^{-1}] \ar[r] & R,
}
\]
where $H\pi_0(R)[\beta \otimes \beta^{-1}]$ is $H$ applied to the polynomial algebra over $\pi_0(R)$ on a (formal) generator $\beta \otimes \beta^{-1}$ in degree zero. The left vertical map sends $\beta \otimes \beta^{-1}$ to the element with the same name. Since $H\pi_0(R)[\beta]  \otimes_{H\pi_0(R)} H\pi_0(R)[\beta^{-1}] $ is flat over  $H\pi_0(R)[\beta \otimes \beta^{-1}]$ and 
$H\pi_0(R)[\beta] $ and $H\pi_0(R)[\beta^{-1}] $ are formal, we deduce the formality of $R$.
\end{proof}

Let $E = \E$ and let $E_{\Q} = \Q \otimes \E$.

\begin{cor} \label{cor:formalcomposite}
The diagram 
\[
H\Q \to H\Q_p \to E_{\Q}
\]
is formal as an object in $\CAlg(\Sp)^{\Delta^2}$. It is the image of $\Q \to \Q_p \to \pi_*(E_\Q)$ under $H^{\Delta^2}$. 
\end{cor}
\begin{proof}
Since $\pi_0(E_\Q)$ is a formally smooth $\Q_p$-algebra, \cref{l:k-form} and \cref{l:Lur} implies that the map $H\Q_p \to E_\Q $ is formal. The result now follows as the canonical map $H\Q \to H\Q_p$ is formal.
\end{proof}

Given maps of $\bE_{\infty}$-algebras $A \to B \to C$, taking the Amitsur complexes of $C$ over $A$ and over $B$ we get a map of cosimplicial $\bE_\infty$-rings 
\[
C^{ \otimes_A \bullet+1 } \to C^{\otimes_B \bullet+1 }.
\]
Thus the maps of $\bE_\infty$-rings 
\[
H\Q \to H\Q_p \to E_{\Q}
\]
gives rise to an object in the $\infty$-category of maps of cosimplicial $\bE_{\infty}$-rings
\[
(E_{\Q}^{\otimes_{H\Q} \bullet+1} \rightarrow E_{\Q}^{\otimes_{H\Q_p} \bullet+1}) \in \CAlg(\Sp)^{\Delta \times \Delta^1}.
\]
\begin{cor}\label{cor:formalp}
The object 
\[
(E_{\Q}^{\otimes_{H\Q} \bullet+1} \rightarrow E_{\Q}^{\otimes_{H\Q_p} \bullet+1}) \in \CAlg(\Sp)^{\Delta \times \Delta^1}
\]
is formal.
\end{cor}
\begin{proof}
By \cref{cor:formalcomposite}, there are equivalences
\[
(H\pi_*(E_\Q))^{\otimes_{H\Q} k} \simeq E_{\Q}^{\otimes_{H\Q} k} \text{ and } (H\pi_*(E_\Q))^{\otimes_{H\Q_p} k} \simeq E_{\Q}^{\otimes_{H\Q_p} k}
\]
for all $k$.

Now the laxness of $H^{\Delta}$ gives us a natural map in $\CAlg(\Sp)^{\Delta \times \Delta^1}$ 
\[
\xymatrix{
H^{\Delta}\left( \pi_*(E_{\Q})^{\otimes_{\Q} \bullet+1} \right) \ar[r]\ar[d]& H^{\Delta}\left( \pi_*(E_{\Q})^{\otimes_{\Q_p} \bullet+1} \right)\ar[d] \\
E_{\Q}^{\otimes_{H\Q} \bullet+1} \ar[r] & E_{\Q}^{\otimes_{H\Q_p} \bullet+1}.
}
\]
The vertical arrows give an equivalence in $\CAlg(\Sp)^{\Delta \times \Delta^1}$ by the flatness of $\pi_*(E_\Q)$ over $\Q_p$ and $\Q$.
\end{proof}

\subsection{The $C_{p-1}$-action on $E$-theory}\label{ss:cp_1}
The Morava stabilizer group acts on $\E$ through $\bE_{\infty}$-ring maps. This gives rise to a $C_{p-1}$-action by restriction. In this subsection, we study the effect of this action on the coefficients of $\E$ and its tensor powers.

It follows from \cite{ghproblems}, that the stabilizer group $\mathbb{S}$ (which depends on $n$ and $p$) acts on $E = \E$ through $\bE_{\infty}$-ring maps. Recall that $\mathbb{S} = \Aut(\G)$, where $\G$ is a height $n$ formal group law over $k$, a perfect field of characteristic $p$. The natural action of $\Z$ on $\G$ by left multiplication extends to an action of the $p$-adic integers $\Z_p$. Thus the units in $\Z_p$ act by automorphisms. This implies that $\mathbb{S}$ contains a distinguished subgroup $C_{p-1} \subset \Z_{p}^{\times} \subset \mathbb{S}$. This inclusion is a distinguished element in $\hom(C_{p-1},\Z_{p}^{\times})$ that we use to fix the isomorphism
\[
\Z/(p-1) \cong \hom(C_{p-1},\Z_{p}^{\times})
\]
sending $1$ to this element.

\begin{prop} \label{splitit}
The action of $i \in C_{p-1} \subset \Z_{p}^{\times}$ on $E^* \cong E^0[u,u^{-1}]$ is given by sending
\[
u \mapsto iu.
\]
\end{prop}
\begin{proof}
This follows from \cite[Proposition 3.3, Theorem 4.4]{dh_lt} and is also described in \cite[proof of 5.4.9]{rognes_galois}. See also \cite[Appendix]{henn_finiteresolutions}.
\end{proof}

\begin{rem}\label{r:sparse}
Since $p-1$ is prime to $p$, the homotopy groups of the homotopy fixed points $E^{hC_{p-1}}$ may be computed by taking the fixed points for the action of $C_{p-1}$ on $\pi_* E$. There is an isomorphism
\[
\pi_*(E^{hC_{p-1}}) \cong E^0[u^{\pm(p-1)}].
\]
\end{rem}

Since $C_{p-1}$ acts on $E$ and the map $S^0 \rightarrow E$ is $C_{p-1}$-equivariant (with the trivial action of $C_{p-1}$ on $S^0$), there is an action of $C_{p-1}$ on the Amitsur complex of the map $S^0 \to E$. On each degree of the Amitsur complex $C_{p-1}$ acts diagonally. The action of $C_{p-1}$ on $\pi_0(E^{\otimes k})$ is not trivial for $k>1$, thus the formula of \cref{splitit} does not extend to $E^{\otimes k}$ when $k>1$. However, this can be corrected by working over the $p$-complete sphere spectrum $\hS = \pS$ (working with the Amitsur complex of $\hS \rightarrow E$). Note that the action of $C_{p-1}$ on $E^{\otimes k}$ described above induces an action on $E^{\otimes_{\hS} k}$. Our initial goal is to prove that the action of $C_{p-1}$ on $\pi_0(E^{\otimes_{\hS} k})$ is trivial. This will follow from the next two lemmas. 

\begin{lem} \label{pcomp}
Let $A$ be a spectrum such that $\pi_*(A)$ is torsion-free. There is a canonical isomorphism
\[
\pi_*(A^{\wedge}_{p}) \lra{\cong} (\pi_*(A))^{\wedge}_{p}.
\]
\end{lem}
\begin{proof}
The proof is similar to the proof of \cref{l:surj_sp}. By \cite{Bousfieldlocalization}, there is an equivalence
\[
A^{\wedge}_{p} \simeq \lim(A \otimes S^0/p^k),
\]
where $S^0/p^k$ is the mod $p$ Moore spectrum. It suffices to show that 
\[
\pi_*(A \otimes S^0/p^k) \cong (\pi_*A)/p^k.
\]
These homotopy groups sit in a short exact sequence
\[
\pi_n(A)/p^k \rightarrow \pi_n(A \otimes S^0/p^k) \rightarrow \pi_{n-1}(A)[p^k].
\]
Since $\pi_*(A)$ is torsion-free, we have an isomorphism
\[
\pi_n(A)/p^k \lra{\cong} \pi_n(A \otimes S^0/p^k).
\]
\end{proof}

\begin{lem} \label{ptriv}
The action of $C_{p-1}$ on $\pi_0((E^{\otimes k})^{\wedge}_{p})$ is trivial.
\end{lem}
\begin{proof}
Fix a coordinate on $\G_{E}$, the formal group associated to $E$, and let $s,t \colon \pi_0E \rightarrow \pi_0(E \otimes E)$ be the two canonical maps. Recall from \cite{coctalos} that $\pi_0(E^{\otimes 2})$ carries the universal isomorphism of formal group laws
\[
s^* \G_{E} \cong t^* \G_{E}.
\]
That is, $\pi_0(E \otimes E)$ corepresents the functor on commutative rings sending a ring $R$ to the set of isomorphisms between the two formal group laws over $R$ determined by the two induced maps from $\pi_0E$ to $R$.

Since $E$ is $p$-complete, $a \in \Z_p$ determines an endomorphism $[a]$ of $\G_E$ ($\G_E$ is a formal $\Z_p$-module). The action of $a \in C_{p-1} \subset \Z_p$ on $\pi_0(E \otimes E)$ is given by conjugating  the universal isomorphism by the pullback of $[a]$ along $s$ and $t$.

Since $\pi_*(E^{\otimes k})$ is flat over $\pi_* E$ and $\pi_* E$ is torsion-free, it follows that $\pi_*(E^{\otimes k})$ is torsion-free. \cref{pcomp} implies that
\[
\pi_0((E^{\otimes k})^{\wedge}_{p}) \cong \pi_0(E^{\otimes k})^{\wedge}_{p}.
\] 

Thus, when we restrict the functor determined by $\pi_0(E \otimes E)$ to $p$-complete rings, the resulting functor is corepresented by $\pi_0((E \otimes E)^{\wedge}_{p})$. Since $\pi_0((E \otimes E)^{\wedge}_{p})$ is $p$-complete, it carries the universal isomorphism of formal $\Z_p$-modules $s^*\G_E \cong t^*\G_E$. Since this is an isomorphism of formal $\Z_p$-modules, the conjugation action is trivial.

We generalize this to the $k$th tensor power $\pi_0(E^{\otimes k})$. We have a $C_{p-1}$-equivariant equivalence
\[
(E \otimes E)^{\otimes_E k-1} \simeq E^{\otimes k},
\]
which, by flatness, induces an isomorphism
\[
\pi_0(E \otimes E)^{\otimes_{\pi_0 E} k-1} \cong \pi_0(E^{\otimes k}).
\]
Thus $\pi_0(E^{\otimes k})$ carries the universal $(k-1)$-tuple of composable isomorphisms between the $k$-formal group laws determined by the $k$ canonical maps from $\pi_0E$ to $\pi_0(E^{\otimes k})$. Since $C_{p-1}$ acts trivially on $\pi_0E$, the action of $C_{p-1}$ on $\pi_0(E^{\otimes k})$ is given by conjugating the string of $k-1$ composable isomorphisms. It follows from the $k=2$ case that this action is trivial over the $p$-completion. \end{proof}

\begin{prop} \label{trivaction}
The action of $C_{p-1}$ on $\pi_0(E^{\otimes_{\hS} k})$ is trivial.
\end{prop}
\begin{proof}
By \cref{completetensor},  the arithmetic square for $E^{\otimes_{\hS} k}$ takes the form
\[
\xymatrix{E^{\otimes_{\hS} k} \ar[r] \ar[d] & (E^{\otimes k})^{\wedge}_{p} \ar[d] \\ (\Q \otimes E)^{\otimes_{\Q \otimes \hS} k} \ar[r] & \Q \otimes (E^{\otimes k})^{\wedge}_{p}.}
\]
By even periodicity, we have an injection
\[
\pi_0(E^{\otimes_{\hS} k}) \hookrightarrow \pi_0((\Q \otimes E)^{\otimes_{\Q \otimes \hS} k}) \oplus \pi_0(E^{\otimes k})^{\wedge}_{p}.
\]
It suffices to show that the $C_{p-1}$-action on the codomain is trivial. The action on the right hand side is trivial by \cref{ptriv}. Since $\Q \otimes \hS \simeq H\Q_p$, $\Q \otimes E$ is flat over $\Q \otimes \hS$. Thus we have an isomorphism
\[
\pi_*((\Q \otimes E)^{\otimes_{\Q \otimes \hS} k}) \cong (\pi_*(\Q \otimes E))^{\otimes_{\Q_p} k}.
\]
Since $C_{p-1}$ acts on $\pi_*(\Q \otimes E)$ according to \cref{splitit}, direct computation shows that the action of $C_{p-1}$ on $\pi_0((\Q \otimes E)^{\otimes_{\Q \otimes \hS} k})$ is trivial for degree reasons.
\end{proof}

\begin{cor}\label{cor:action}
The action of $i \in C_{p-1}$ on $\pi_{2l}(E^{\otimes_{\hS} k})$ is given by multiplication by $i^l$.
\end{cor}
\begin{proof}
In the $\infty$-category of naive $C_{p-1}$-equivariant $E$-module spectra, there is a canonical equivalence
\[
E^{\otimes_{\hS} k} \otimes_{E} \Sigma^{-2l}E \lra{\simeq} \Sigma^{-2l}E^{\otimes_{\hS} k}.
\]
This gives rise to a $C_{p-1}$-equivariant isomorphism
\[
\pi_0(E^{\otimes_{\hS} k}) \otimes_{\pi_0E} \pi_{2l}E \lra{\cong} \pi_{2l}(E^{\otimes_{\hS} k}),
\]
where $C_{p-1}$ acts trivially on $\pi_0$ by \cref{trivaction} and by $i^l$ on $\pi_{2l}$ by \cref{splitit}.
\end{proof}

\begin{rem} \label{rem:gap}
There are isomorphisms 
\begin{equation*}
\pi_l((E^{\otimes_{\hS} k})^{hC_{p-1}}) \cong 
\begin{cases}
	\pi_l(E^{\otimes_{\hS} k}) & \text{if } 2(p-1)|l,\\
	0 & \text{otherwise}
\end{cases}
\end{equation*} 
and the map of $\bE_{\infty}$-ring spectra $(E^{\otimes_{\hS} k})^{hC_{p-1}} \rightarrow E^{\otimes_{\hS} k}$ is an isomorphism after applying $\pi_*$ when $* = 2(p-1)l$.
\end{rem}

\subsection{Symmetric monoidal categories from abelian groups}\label{ss:monoid}

The $C_{p-1}$-action constructed in the previous subsection shows that the nontrivial $k$-invariants of $E^{\otimes_{\hS} l}$ grow sparser as $p$ grows larger. This is the essential fact that we use to show formality. However, since we are interested in showing the formality of $E^{\otimes_{\hS} \bullet+1}$ as a cosimplicial $\bE_{\infty}$-ring, we need to analyze the $C_{p-1}$-action in a way that respects both the cosimplicial and $\bE_{\infty}$-ring structure simultaneously. The tools to do this are developed in the next three subsections. In this subsection, we set up the general machinery to encode both the weight decomposition and the $\bE_{\infty}$-ring structure on an object in a symmetric monoidal $\infty$-category.

Let $A$ be an abelian group. It may be considered as a symmetric monoidal $\infty$-category $A^{\otimes}$ whose underlying category is the discrete category which is $A$ as a set. This point of view is functorial, a homomorphism of abelian groups $A \to B$ gives rise to a symmetric monoidal functor of symmetric monoidal $\infty$-categories $A^{\otimes} \to B^{\otimes}$. Recall that the $\infty$-category of symmetric monoidal $\infty$-categories and lax symmetric monoidal functors is a full subcategory of the $\infty$-category of $\infty$-operads. Thus $A^{\otimes}$ may be viewed as an $\infty$-operad. For example if $0$ is the group with one element we have $0^{\otimes} = \mathbb{E}_{\infty}$.

Recall from \cite[Section 2.1]{ha} that an $\infty$-operad $\cO^{\otimes}$ may be thought of as a multicategory. It is possible to add an additional object $*$ to $\cO$, the underlying $\infty$-category of $\cO^{\otimes}$, which is ``multifinal". That is all multimapping spaces with target $*$ are contractible. We carefully define this contruction in the case $\cO^{\otimes} = A^{\otimes}$:

\begin{prop}
Let $A$ be an abelian group. There is a symmetric monoidal $\infty$-category $A^{\otimes | \rhd}$ under $\bE_{\infty}$ built out of $A^{\otimes}$ by adding a multifinal object. It is equipped with an inclusion
\[
A^{\otimes} \hookrightarrow A^{\otimes | \rhd}
\]
and the construction is natural in maps of abelian groups.
\end{prop}
\begin{proof}
For an abelian groups $A$, the $\infty$-operad $A^{\otimes | \rhd}$ can be described very explicitly.
Indeed, given a pointed set $X \in \Set_*$, let $X^{\circ}$ be the set obtained by removing the special point (note that this is not functorial).
An object in $A^{\otimes | \rhd}$ is a finite pointed set $X \in \Fins$ together with a map $a \colon X^{\circ}  \to  A \coprod \{*\}\in  \Set$.
While a morphism in  $A^{\otimes | \rhd}$ from $a \colon X^{\circ} \to   A \coprod \{*\}$ to $b \colon Y^{\circ} \to  A \coprod \{*\}$ is a map  $f \colon X \to Y$ in $\Fins$ such that 
for all $y\in Y$ such that $b(y) \not \in \{*\}$ we have that $ a(f^{-1}(y)) \subset A$ and $\sum a(f^{-1}(y)) = b(y)$. This gives $A^{\otimes | \rhd}$ the structure of a symmetric monoidal $\infty$-category with respect to the obvious map to $\mathrm{Fin}_{*}$. 

The inclusion $A^{\otimes} \hookrightarrow A^{\otimes | \rhd}$ is induced by the map postcomposing $a \colon X^{\circ}  \to  A$ with the canonical inclusion $A \hookrightarrow A \coprod \{*\}$.
\end{proof}

Let $A^{\rhd}$ be the underlying $\infty$-category of $A^{\otimes | \rhd}$.

Now let $\mathcal{C}$ be a presentably symmetric monoidal $\infty$-category and let $A$ be an abelian group. Restriction along the inclusion $A^{\otimes} \hookrightarrow A^{\otimes | \rhd}$ produces a map 
\[
\theta_A \colon \Alg_{A^{\otimes | \rhd}}(\mathcal{C}) \to \Alg_{A^{\otimes}}(\mathcal{C}).
\]
\begin{lem} \label{lem:operads}
Let  $F \in \Alg_{A^{\otimes}}(\mathcal{C})$, let $F'\in  \Alg_{A^{\otimes | \rhd}}(\mathcal{C})$, and let $f \colon F \to \theta_A(F')$ be a map. Then $f$ exhibits $F'$ as a free $A^{\otimes | \rhd}$-algebra generated by $F$ if and only if 
\begin{enumerate}
\item For every $a \in A$ the map $f(a) \colon F(a) \to \theta_A(F')(a)$ is an equivalence in $\mathcal{C}$.
\item The underlying functor 
\[
\underline{F'} \colon A^{\rhd} \to \cC
\]
is a colimit diagram.
\end{enumerate}
\end{lem}
\begin{proof}
We have attempted to use the notation of \cite[Definition 3.1.3.1]{ha}. In \cite[Definition 3.1.3.1]{ha}, for an object $x \in A^{\rhd}$, Lurie defines
\[
\left(A^{\otimes}_{\mathrm{act}}\right)_{/x} = A^{\otimes}\times_{A^{\otimes | \rhd}} \left(A^{\otimes | \rhd}_{\mathrm{act}}\right)_{/x}.
\]
Note that $\left(A^{\otimes}_{\mathrm{act}}\right)_{/a}$ is in fact the overcategory $\left(A^{\otimes}_{\mathrm{act}}\right)_{/a}$ for $a$ an object in $A \subset A^{\otimes}_{\mathrm{act}}$. We define 
\[
c_{a} \colon \Delta_0 \to \left(A^{\otimes}_{\mathrm{act}}\right)_{/a} 
\]
to be the functor that chooses the cone point. Also let
\[c_{*} \colon A \to \left(A^{\otimes}_{\mathrm{act}}\right)_{/*} 
\]
be the functor induced by the two obvious functors
\[
A \to A^{\otimes}
\]
and
\[
A \to A^{\rhd} \to   A^{\rhd}_{/*} \to \left(A^{\otimes | \rhd}_{\mathrm{act}}\right)_{/*}.
\]
Both $c_a$ and $c_*$ are cofinal. This is clear for $c_a$. The cofinality of $c_{*}$ is a straightforward computation since 
\[
\left(A^{\otimes}_{\mathrm{act}}\right)_{/*} \simeq A^{\otimes}_{\mathrm{act}} \simeq \Coprod{a \in A} \left(A^{\otimes}_{\mathrm{act}}\right)_{/a}. 
\]

By Definition 3.1.3.1 of \cite{ha}, $f$ exhibits $F'$ as a free $A^{\otimes | \rhd}$-algebra generated by $F$ if and only if for every object $x \in A^{\rhd}$, the map  \[
\bar{\alpha}_{x} \colon \big ( \left(A^{\otimes}_{\mathrm{act}}\right)_{/x} \big )^{\rhd} \to \mathcal{C}^{\otimes} 
\] is an operadic colimit diagram.
Since $\mathcal{C}$ is a presentably symmetric monoidal $\infty$-category, \cite[Proposition 3.1.1.15]{ha} and \cite[Proposition 3.1.1.16]{ha} prove that this is equivalent to the following condition: For every object $x \in A^{\rhd}$ the map  \[
\bar{\beta}_{x} \colon \big (\left(A^{\otimes}_{\mathrm{act}}\right)_{/x} \big )^{\rhd} \to \mathcal{C}
\] is a colimit diagram. Here $\bar{\beta}_{x}$ is a certain functor built out of $\bar{\alpha}_{x}$. 

By the cofinality of $c_{a}$ and $c_{*}$, $\bar{\beta}_{x}$ is a colimit diagram for every $x \in A^{\rhd}$ if and only if $\bar{\beta}_{a} \circ c_{a}^{\rhd}$ and $\bar{\beta}_{*} \circ c_{*}^{\rhd}$ are colimit diagrams. These two conditions are conditions (1) and (2) in the statement of the lemma.
\end{proof}

Now since  $\mathcal{C}$ is presentably symmetric monoidal, \cite[Corollary 3.1.3.5]{ha} implies that $\theta_A$ has a left adjoint (the free algebra functor) that we shall denote by 
\[
L_{A}\colon \Alg_{A^{\otimes}}(\mathcal{C}) \to \Alg_{A^{\otimes | \rhd}}(\mathcal{C}).
\]

\begin{prop} \label{prop:operads}
In the above situation we have:
\begin{enumerate}
\item The unit of the adjunction $ L_{A} \dashv \theta_A$, $\eta\colon\Id \Rightarrow \theta_A\circ L_A$, is an equivalence.
\item Let $F \in \Alg_{A^{\otimes}}(\mathcal{C})$ and let 
\[
\underline{L_A(F)} \colon A^{\rhd} \to \mathcal{C} 
\]
be the underlying functor of $\infty$-categories. Then $\underline{L_A(F)}$ is a colimit diagram. 
\item Let $F' \in  \Alg_{A^{\otimes | \rhd}}(\mathcal{C})$ be such that the underlying functor 
\[
\underline{F'} \colon A^{\rhd} \to \mathcal{C} 
\]
is a colimit diagram. Then the counit $\epsilon \colon L_A \circ \theta_A \Rightarrow \Id$ applied to $F'$ is an equivalence.
\end{enumerate}
\end{prop}
\begin{proof}
For $F \in \Alg_{A^{\otimes}}(\mathcal{C})$, the map 
\[ \eta_F \colon F \to  \theta_A(L_A(F))
\]
exhibits $L_A(F)$ as the free $A^{\otimes | \rhd}$-algebra generated by $F$.
We thus get (1) and (2) from parts (1) and (2) of \cref{lem:operads} respectively.

For part (3), let $F' \in  \Alg_{A^{\otimes | \rhd}}(\mathcal{C})$ be such that the underlying functor 
\[
\underline{F'} \colon A^{\rhd} \to \mathcal{C} 
\]
is a colimit diagram. Consider the identity map $f \colon \theta_A(F') \to \theta_A(F')$, by \cref{lem:operads} $f$ exhibits $F'$ as the free $A^{\otimes | \rhd}$-algebra generated by $\theta_A(F')$. Thus the map $\epsilon_{F'} \colon L_A (\theta_A(F')) \to F'$ is an equivalence.
\end{proof}

We now turn to understanding symmetric monoidal functors with domain $(\Z/n)^{\otimes}$. This result will be used in the next two subsections. 

\begin{prop}\label{p:lax-homotopy}
Let $\cC$ be symmetric monoidal $1$-category and let $a$ be a tensor invertible object in $\cC$ such that $a^{\otimes n} \cong 1_{\cC}$. Assume that the symmetry map $\beta_{a,a} = \Id_{a\otimes a}$, then there is a symmetric monoidal functor
\[
(\Z/n)^{\otimes} \to \cC
\]
sending $1$ to $a$. 
\end{prop}
\begin{proof}
Without loss of generality, we may assume that $\cC$ is skeletal. Thus we may assume that $a^{\otimes n} = 1_{\cC}$. Let $m>0$ be the smallest positive integer such that $a^{\otimes m} = 1_{\cC}$. It follows that $m|n$. Consider the discrete subcategory of $\cC$ supported on $a^{\otimes i}$ for $i \geq 0$. By our assumption, the symmetry maps of powers of $a$ are identity maps. Thus, this discrete category is a symmetric monoidal subcategory of $\cC$ isomorphic to $(\Z/m)^{\otimes}$. We compose with the symmetric monoidal functor $(\Z/n)^{\otimes} \to (\Z/m)^{\otimes}$ to get the desired symmetric monoidal functor.
\end{proof}

\subsection{Functorial weight decompositions}\label{ss:weight}

In this subsection we will use the $C_{p-1}$-action described in \cref{ss:cp_1} to decompose $E^{\otimes_{\hS} k}$ as a finite sum of $(E^{\otimes_{\hS} k})^{hC_{p-1}}$-module spectra. We will develop the theory of weights in the context of naive $C_{p-1}$-equivariant modules over $\hS$ to produce a functorial decomposition. In the next subsection, this naturality will be used to decompose the entire Amitsur complex of the map $\hS \to E$. 

Let $\Mod_{\hS}$ be the $\infty$-category of modules over $\hS$ and let 
\[
(\Mod_{\hS})^{BC_{p-1}} = \Fun(BC_{p-1},\Mod_{\hS})
\] 
be the $\infty$-category of naive $C_{p-1}$-equivariant modules over $\hS$. Since $\Mod_{\hS}$ is a symmetric monoidal $\infty$-category, the functor category is as well. We will make use of the forgetful functor
\[
U \colon (\Mod_{\hS})^{BC_{p-1}} \rightarrow \Mod_{\hS}
\]
given by restriction along $Be \rightarrow BC_{p-1}$.

\begin{lem}
There is an action of $\Z_{p}^{\times}$ on $\hS$ inducing the obvious action on $\pi_* \hS$.
\end{lem}
\begin{proof}
The canonical map
\[
B\Z \rightarrow B\Z_p
\]
is an $H\F_p$-homology equivalence. This induces a map
\[
S^0 \simeq \Sigma^{-1}\Sigma^{\infty} B\Z \rightarrow \Sigma^{-1}(\Sigma^{\infty} B\Z_p)^{\wedge}_{p}
\]
that factors to give a map
\[
\hat{S} \rightarrow \Sigma^{-1}(\Sigma^{\infty} B\Z_p)^{\wedge}_{p}.
\]
This is a map between $p$-complete connective spectra and it is an $H\F_p$-homology equivalence, thus it is an equivalence. There is an obvious action of $\Z_{p}^{\times}$ on $B\Z_p$ and this induces the action on $\hat{S}$ that we desire.
\end{proof}

Let $j \in \Z/(p-1) = \hom(C_{p-1},\Z_{p}^{\times})$. Let $\hS(j)$ be the $p$-complete sphere equipped with the action of $C_{p-1}$ induced by $j$. We are viewing this as a naive $C_{p-1}$-equivariant module. Let $\hS(0) \simeq \hS$ be the $C_{p-1}$-spectrum determined by $j = 0$, the zero homomorphism. 

Now let $B$ be a naive $C_{p-1}$-equivariant module over $\hS$. We define
\[
B_{j} = (B \otimes_{\hS} \hS(-j))^{hC_{p-1}}.
\]
Note that the underlying non-equivariant modules of $B \otimes_{\hS} \hS(-j)$ and $B$ agree:
\[
U(B \otimes_{\hS} \hS(-j)) \simeq U(B). 
\]

\begin{lem} \label{lem:Natslemma}
Let $\Z_p(-j)$ be $\Z_p$ acted on by $C_{p-1}$ through $-j$ and let $B \in (\Mod_{\hS})^{BC_{p-1}}$. Then there is an isomorphism of $C_{p-1}$-modules
\[
\pi_*(B \otimes_{\hS} \hS(-j)) \cong \pi_*(B) \otimes_{\Z_p} \Z_p(-j)
\] 
and an isomorphism of abelian groups
\[
\pi_*(B_j) \cong (\pi_*(B) \otimes_{\Z_p} \Z_p(-j))^{C_{p-1}}.
\]
\end{lem}
\begin{proof}
There is an isomorphism of $C_{p-1}$-modules
\[
\pi_0(\hS(-j)) \cong \Z_p(-j)
\]
inducing a map of $C_{p-1}$-modules
\[
\pi_*(B) \otimes_{\Z_p} \Z_p(-j) \to \pi_*(B \otimes_{\hS} \hS(-j)) 
\]
and this is an isomorphism of abelian groups and thus an isomorphism of $C_{p-1}$-modules. The second isomorphism follows from the fact that $p-1$ is coprime to $p$ so that the homotopy fixed points can be computed as the fixed points.
\end{proof}

\begin{ex} \label{cor:simplesplitting}
By \cref{cor:action} and \cref{lem:Natslemma}, there are isomorphisms 
\begin{equation*}
\pi_l((E^{\otimes_{\hS} k})_j) \cong 
\begin{cases}
	\pi_lE^{\otimes_{\hS} k} & \text{if } l \equiv 2j \text{ mod } 2(p-1),\\
	0 & \text{otherwise}.
\end{cases}
\end{equation*}
Thus, on the level of homotopy groups, we have an isomorphism
\[
\bigoplus_{j \in \Z/(p-1)} \pi_*((E^{\otimes_{\hS} k})_j) \lra{\cong} \pi_*(E^{\otimes_{\hS} k}).
\]
\end{ex}

We are going to promote this isomorphism to an equivalence of $(E^{\otimes_{\hS} k})_0$-modules
\[
\Oplus{j \in \Z/(p-1)} (E^{\otimes_{\hS} k})_{j} \lra{\simeq} E^{\otimes_{\hS} k}.
\]
In fact, we would like these equivalences to assemble into a decomposition of the entire Amitsur complex $E^{\otimes_{\hS} \bullet + 1}$. The next propositions deal with all of the coherence involved in doing this. 

Recall the construction of the $\infty$-operad $(\Z/(p-1))^{\otimes}$ from \cref{ss:monoid}.

\begin{prop} \label{goodoldW}
There is an object $W \in \Alg_{(\Z/(p-1))^{\otimes}}(\Mod_{\hS}^{BC_{p-1}})$ such that
\[
W(j) = \hS(-j).
\]
\end{prop}
\begin{proof}
We will construct a symmetric monoidal functor 
\[
W \colon (\Z/(p-1))^{\otimes} \to \Mod_{\hS}^{BC_{p-1}}
\]
such that $W(j) = \hS(-j)$. In particular, $W$ can be considered as a lax symmetric monoidal functor and thus an object in $\Alg_{(\Z/(p-1))^{\otimes}}(\Mod_{\hS}^{BC_{p-1}})$. This will be done in two steps. 

First we produce a symmetric monoidal functor
\[
(\Z/(p-1))^{\otimes} \to \Ho(\Mod_{\hS}^{BC_{p-1}}) 
\]
to the homotopy category with the correct property. We apply \cref{p:lax-homotopy}, so we need to check two conditions. 

First, we show that $\hS(-1)^{\otimes j} \cong \hS(-j)$ (note that $\hS(p-1) = \hS(0)$ by definition) in $\Ho(\Mod_{\hS}^{BC_{p-1}})$. We have an equivalence
\[
\Hom_{  \Mod_{\hS}^{BC_{p-1}}}(\hS(-1)^{\otimes j} , \hS(-j)) \simeq \Hom_{\Mod_{\hS}}(U(\hS(-1))^{\otimes j},U(\hS(-j)))^{C_{p-1}} \simeq \hS^{hC_{p-1}},
\]
where the action in the middle is the conjugation action. The first equivalence has been proven in the $\infty$-categorical setting in \cite[Proposition 2.3]{glasman}. The conjugation action on the homotopy groups of $\hS$ on the right hand side is trivial since, on the level of homotopy groups, the weights are the same. Since $p-1$ is prime to $p$, we get $\hS^{hC_{p-1}} = \hS$. Thus we have an equivariant map lifting the identity.

Second, we would like to prove that 
\[
\beta_{\hS(-1),\hS(-1)} = \Id_{\hS(-1)\otimes \hS(-1)} \in \Map_{\Ho(\Mod_{\hS}^{BC_{p-1}})}(\hS(-1)\otimes \hS(-1), \hS(-1)\otimes \hS(-1)).
\]
We have a symmetric monoidal forgetful functor
\[
U\colon \Ho(\Mod_{\hS}^{BC_{p-1}}) \to \Ho(\Mod_{\hS}).
\]
Similarly to previous paragraph, $U$ induces an isomorphism  
\[ 
\Map_{\Ho(\Mod_{\hS}^{BC_{p-1}})}(\hS(-1)\otimes \hS(-1), \hS(-1)\otimes \hS(-1)) \lra{\cong}  \Map_{\Ho(\Mod_{\hS})}(\hS \otimes \hS , \hS \otimes \hS)
\]
sending $\beta_{\hS(-1),\hS(-1)}$  to $\beta_{\hS,\hS}$, but $\beta_{\hS,\hS} = \Id_{\hS \otimes \hS}$
as $\hS$ is the unit. Now \cref{p:lax-homotopy} furnishes us with the functor. 

Finally, we show that the symmetric monoidal functor
\[
(\Z/(p-1))^{\otimes} \to \Ho(\Mod_{\hS}^{BC_{p-1}}) 
\]
lifts to a functor 
\[(\Z/(p-1))^{\otimes} \to \Mod_{\hS}^{BC_{p-1}}.
\]
By passing to Picard spectra, we are looking for a lift in the diagram:
\[
\xymatrix{
\empty & \pic(\Mod_{\hS}^{BC_{p-1}}) \ar[d]\\
H\Z/(p-1) \ar[r] \ar@{-->}[ru] & \pic(\Mod_{\hS}^{BC_{p-1}})[0,1], 
}
\]
where $\pic(\Mod_{\hS}^{BC_{p-1}})[0,1]$ is the first Postnikov truncation of $\pic(\Mod_{\hS}^{BC_{p-1}})$.
The obstruction to this lift is in the abelian group 
\[
\big [H\Z/(p-1) , \Sigma  \pic(\Mod_{\hS}^{BC_{p-1}})[2,\infty] \big ].
\]
But 
\[
\pi_i \big ( \Sigma \pic(\Mod_{\hS}^{BC_{p-1}})[2,\infty] \big ) = 0 \text{ for } i<3
\] 
and 
\[
\pi_i \big ( \Sigma \pic(\Mod_{\hS}^{BC_{p-1}})[2,\infty] \big ) = \pi_{i-2}(\hS) \text{ for } i \geq 3
\]
so all of the homotopy  groups of $\Sigma  \pic(\Mod_{\hS}^{BC_{p-1}})[2,\infty] $ are finite of order prime to $(p-1)$.
\end{proof}

Recall from \cite[Example 3.2.4.4]{ha} that if $\cC^{\otimes}$ is a symmetric monoidal $\infty$-category and $\cO^{\otimes}$ is an $\infty$-operad, then the $\infty$-category $\Alg_{\cO^{\otimes}}(\cC)$ inherits a symmetric monoidal structure given by pointwise tensor product.

\begin{prop}\label{p:weights}
Let $D$ be a small $\infty $-category and let 
\[
B^{\bullet} \in \CAlg\left(\left(\Mod_{\hS}^{BC_{p-1}}\right)^D\right) \simeq \CAlg(\Mod_{\hS}^{BC_{p-1}})^D. 
\]
There exists 
\[
X^{\bullet}  \in \Alg_{\Z/(p-1)^{\otimes | \rhd}}( \Mod_{\hS}^D),
\]
depending on $B^{\bullet}$, such that: 
\begin{enumerate}
\item For all $j \in \Z/(p-1)$ we have an equivalence  
\[
X^{\bullet}(j) \simeq (B^{\bullet})_{j} \in \Mod_{\hS}^D. 
\]
\item Let $X^{\bullet}(*) \in \CAlg( \Mod_{\hS}^D) $ be the restriction to the cone point. Then we have an equivalence 
\[
X^{\bullet}(*)  \simeq U(B^{\bullet}) \in \CAlg( \Mod_{\hS}^D).
\]
\end{enumerate}
\end{prop}
\begin{proof}
Recall that in \cref{goodoldW} we constructed an object
\[
W \in \Alg_{(\Z/(p-1))^{\otimes}}(\Mod_{\hS}^{BC_{p-1}}).
\]
In \cref{prop:operads}, for $A$ a finite abelian group, we studied a functor 
\[
L_{A}\colon \Alg_{A^{\otimes}}(\mathcal{C}) \to \Alg_{A^{\otimes | \rhd}}(\mathcal{C}).
\]
Thus we have a lax symmetric monoidal functor
\[
L_{\Z/(p-1)}(W) \colon \Z/(p-1)^{\otimes | \rhd} \rightarrow \Mod_{\hS}^{BC_{p-1}}.
\] 
Postcomposition with the restriction along the map $ D\to *$ gives a map 
\[
L^{\bullet} \colon \Z/(p-1)^{\otimes | \rhd} \rightarrow \left(\Mod_{\hS}^{BC_{p-1}}\right)^{D}.
\]
 
Now given 
\[
B^{\bullet} \in \CAlg \big (\big (\Mod_{\hS}^{BC_{p-1}} \big )^D \big ),
\] 
pulling back with respect to the (unique) map $\Z/(p-1)^{\otimes | \rhd}  \to \mathbb{E}_{\infty}$ 
produces an object in $\Alg_{\Z/(p-1)^{\otimes | \rhd}}\big ( \big (\Mod_{\hS}^{BC_{p-1}} \big )^D \big)$ that we will abusively also denote by $B^{\bullet}$. 
Now we may tensor $B^{\bullet}$ with $L^{\bullet}$ in $\Alg_{\Z/(p-1)^{\otimes | \rhd}}\big ( \big (\Mod_{\hS}^{BC_{p-1}} \big )^D \big)$. This provides us with an object 
\[
B^{\bullet} \otimes L^{\bullet} \in \Alg_{\Z/(p-1)^{\otimes | \rhd}}\big ( \big (\Mod_{\hS}^{BC_{p-1}} \big )^D \big).
\]
Note that the homotopy fixed point functor 
\[
(-)^{hC_{p-1}}\colon \Mod_{\hS}^{BC_{p-1}}  \to \Mod_{\hS}
\] is lax symmetric monoidal by \cite[Corollary 7.3.2.7]{ha} as it is right adjoint to the constant functor $\Mod_{\hS} \rightarrow \Mod_{\hS}^{BC_{p-1}}$.
We set 
\[
X^{\bullet} =  (B^{\bullet} \otimes L^{\bullet})^{hC_{p-1}} \in    \Alg_{\Z/(p-1)^{\otimes | \rhd}}( \Mod_{\hS}^D).
\]
All of the properties of $X^{\bullet}$ can be checked objectwise in $D$. Property $(1)$ follows from the definition and Property $(2)$ follows from \cref{prop:operads} and can be checked on homotopy groups.
\end{proof}

Our goal is to apply \cref{p:weights} to the Amitsur complex of $\hS \rightarrow E$. This allows us to apply the theory of weights to the entire cosimplicial $\bE_{\infty}$-ring spectrum 
\[
E^{\otimes_{\hS} \bullet+1}.
\] 
Let $D = \Delta$, let 
\[
B^{\bullet} = E^{\otimes_{\hS} \bullet + 1} \in  \CAlg \big ( \big (\Mod_{\hS}^{BC_{p-1}} \big )^\Delta \big ),
\]
and let 
\[
X_{p}^{\bullet} \in \Alg_{\Z/(p-1)^{\otimes | \rhd}}( \Mod_{\hS}^\Delta)
\]
be the object constructed in \cref{p:weights}. Note that this depends on the prime. A part of this dependence may be removed by considering the restriction along the canonical map
\[
\Z^{\otimes | \rhd} \lra{q} \Z/(p-1)^{\otimes | \rhd}.
\]
Let $Y_{p}^{\bullet} = X_{p}^{\bullet} \circ q$. The next result follows from \cref{p:weights} and  \cref{cor:simplesplitting}.

\begin{prop} \label{prop:Yp}
The object
\[
Y_{p}^{\bullet} \in \Alg_{\Z^{\otimes | \rhd}}( \Mod_{\hS}^\Delta)
\]  
satisfies the following properties:
\begin{enumerate}
\item For all $j ,l \in \Z $ and $[n] \in \Delta$ the map   
\[
\pi_l(\underline{Y_{p}}^{[n]}(j)) \to \pi_l(\underline{Y_{p}}^{[n]}(*))
\] 
is an isomorphism  on $\pi_l$ if $l = 2j \text{ mod }  2(p-1)$ and it is the inclusion of zero otherwise, where
\[
\underline{Y_{p}}^{\bullet} \colon \Z^{\rhd} \rightarrow \Mod_{\hS}^\Delta
\]
is the underlying functor of $\infty$-categories.
\item The restriction to the cone point   $Y_{p}^{\bullet}(*) \in \CAlg( \Mod_{\hS}^{\Delta}) $ is equivalent to the Amitsur complex. That is, we have an equivalence 
\[
Y_{p}^{\bullet}(*)  \simeq E^{\otimes_{\hS} \bullet +1} \in \CAlg( \Mod_{\hS}^\Delta).
\]
\end{enumerate}
\end{prop}

\subsection{Formality of the ultraproduct}\label{ss:formal_infty}

The goal of this subsection is to prove \cref{l:formalv1}. 

The symmetric monoidal  structure on  $\mathrm{GrAb}_\Q$ induces a symmetric monoidal structure on 
 $\Alg_{\Z^{\otimes}}(\mathrm{GrAb}_\Q)$ with the unit $e$ the constant functor with value the unit. 

\begin{prop}
There is an invertible object $L^2 \in \Alg_{\Z^{\otimes}}(\mathrm{GrAb}_\Q)$ given by 
\[
L^2(i) \cong \Q[2i],
\]
where $\Q[2i]$ is the object with $\Q$ in degree $2i$ and $0$ everywhere else.
\end{prop}
\begin{proof}
To construct $L^{2}$, we apply \cref{p:lax-homotopy}. Since 
\[
\beta_{\Q[i],\Q[i]} = (-1)^{i^2}\Id_{\Q[i] \otimes \Q[i]},
\] 
it follows that 
\[
\beta_{\Q[2],\Q[2]} = \Id_{\Q[2] \otimes \Q[2]}.
\]
Since the image of $L^{2}$ consists of invertible objects, $L^2$ is invertible. 
\end{proof}
We denote the inverse to $L^2$ by $L^{-2}$ so that 
\[
L^{-2}(i) \cong  \Q[-2i].
\]

\begin{prop} \label{prop:evenalg}
Let $F \in \Alg_{\Z^{\otimes}}(\Sp_{\Q}^{D})$ be such that $\pi_*(F(i)(d))$ is concentrated in degree $2i$. Then $F$ is in the image of 
\[
H^D \colon \Alg_{\Z^{\otimes}}(\mathrm{GrAb}_{\Q}^{D}) \rightarrow \Alg_{\Z^{\otimes}}(\Sp_{\Q}^{D}). 
\]
\end{prop}
\begin{proof}
By pulling back along $D \rightarrow \ast$, we will view $L^2$ and $L^{-2}$ as objects in $\Alg_{\Z^{\otimes}}(\mathrm{GrAb}_{\Q}^{D})$. Let 
\[
G = F \otimes H^D(L^{-2}).
\] 
All of the objects in the image of $G$ are concentrated in degree $0$. Thus $G$ lands in $\Alg_{\Z^{\otimes}} ( \Sp_{\Q}[0,0]^{D})$ and \cref{l:rational-heart} implies that $G$ is in the image $H^{D}$. Thus $G \simeq H^D\pi_{*}^{D}(G)$, where $\pi_{*}^{D}(G)$ is in $\Alg_{\Z^{\otimes}}(\mathrm{GrAb}_{\Q}^{D})$. 

We now use the first statement in \cref{l:rational-heart} to conclude that 
\begin{align*}
H^D(\pi_{*}^{D}(G) \otimes L^{2})  &\simeq H^D(\pi_{*}^{D}(G)) \otimes H^D(L^{2})\\ &\simeq G \otimes H^D(L^{2}) \\
&= F \otimes H^D(L^{-2}) \otimes   H^D(L^{2}) \\ 
&\simeq F \otimes H^D(L^{-2i} \otimes L^{2}) \\
&\simeq F \otimes H^D(e) \\
&\simeq F.
\end{align*}
\end{proof}

\begin{prop}\label{p:ratio-formal}
Let  $F \in \Alg_{\Z^{\otimes|\rhd}}(\Sp_{\Q}^{D})$ and assume that for all $i \in \Z$ and $d \in D$ the graded abelian group
\[
\pi_*(F(i)(d))
\]
is concentrated in degree $2i$
and that 
\[
\underline{F} \colon \Z^{\rhd} \to \Sp_{\Q}^{D}
\]
is a colimit diagram. Then $F(*) \in \CAlg(\Sp_{\Q}^D)$, the restriction to the terminal object,  is formal.
\end{prop}
\begin{proof}
Consider the following collection of $\infty$-categories and functors:
\[
\xymatrix{\Alg_{\Z^{\otimes}}(\mathrm{GrAb}_{\Q}^{D}) \ar[r]^{H^D_{\Z}} \ar[d]_{L^{\mathrm{GrAb}}_{\Z}} & \Alg_{\Z^{\otimes}}(\Sp_{\Q}^{D}) \ar[d]^{L^{\Sp}_{\Z}} \\ \Alg_{\Z^{\otimes | \rhd}}(\mathrm{GrAb}_{\Q}^{D}) \ar[r]_{H^D_{\Z^{\rhd}}} & \Alg_{\Z^{\otimes | \rhd}}(\Sp_{\Q}^{D}).}
\]
The horizontal arrows are induced by postcomposition with 
\[
H^D \colon \mathrm{GrAb}_{\Q}^{D} \rightarrow \Sp_{\Q}^{D}.
\]
The vertical arrows are the left adjoints to the natural restriction maps $\theta_{\Z}^{\mathrm{GrAb}} $ and $\theta_{\Z}^{\Sp}$. First we prove that this diagram commutes.  It is clear that $H_{\Z}^D \circ \theta_{\Z}^{\mathrm{GrAb}}  = \theta_{\Z}^{\Sp} \circ H_{\Z^{\rhd}}^D $ since the $\theta$'s are induced by precomposition.
Using the $L \dashv \theta$ adjunctions studied in \cref{prop:operads}, we get a natural transformation 
\[
\beta \colon L^{\Sp}_{\Z}\circ H_{\Z}^D   \Rightarrow H_{\Z^{\rhd}}^D \circ L^{\mathrm{GrAb}}_{\Z} .
\]
We need to show that $\beta$ is an equivalence, but this can checked on the underlying functors.
Now by parts (1) and (2)  of \cref{prop:operads} the underlying functor produced by either $L$ is a coproduct diagram and  $H^D$ respects coproducts.

Recall that $\theta_{\Z} F \in  \Alg_{\Z^{\otimes}}(\Sp_{\Q}^{D})$ is the restriction of $F$. Let 
\[
G = \pi_*(\theta_{\Z}(F)) \in  \Alg_{\Z^{\otimes}}(\mathrm{GrAb}_{\Q}^{D}).
\] 
\cref{prop:evenalg} implies that $\theta_{\Z}F \simeq H_{\Z}^D(G)$. 
Now by part (3) of \cref{prop:operads}, 
\[
F \simeq L^{\Sp}_{\Z}(\theta_{\Z}F),
\]
so we have equivalences 
\[
F \simeq L^{\Sp}_{\Z}(\theta_{\Z}F) \simeq  L^{\Sp}_{\Z}(H_{\Z}^D(G)) \simeq  H_{\Z^{\rhd}}^D (L^{\mathrm{GrAb}}_{\Z}(G)).
\]
Thus $F$ is in the image  of $H^{D}_{\Z^{\rhd}}$ and so $F(*)$ is in the image of $H^D$.
\end{proof}

\begin{lem}\label{lem:ultra-weight}
The homotopy groups of the spectrum
\[
\Prod{\cF}\big ((\E^{\otimes_{\hS} k})_j\big )
\]
are concentrated in degree $2j$ and the canonical map
\[
\Prod{\cF}\big ((\E^{\otimes_{\hS} k})_j\big ) \rightarrow \Prod{\cF} (\E^{\otimes_{\hS} k})
\]
is an isomorphism in degree $2j$. In particular, the canonical map
\[
\Oplus{j \in \Z}\Prod{\cF}\big ((\E^{\otimes_{\hS} k})_j\big ) \rightarrow \Prod{\cF} (\E^{\otimes_{\hS} k})
\]
is an equivalence of spectra.
\end{lem}
\begin{proof}
Both statement follows immediately from the fact that $\pi_*(-)$ commutes with ultraproducts and \cref{cor:simplesplitting}. 
\end{proof}

Recall the definition of 
\[
Y_{p}^{\bullet} \in \Alg_{\Z^{\otimes | \rhd}}(\Mod_{\hS}^\Delta)
\]
from \cref{prop:Yp}. In particular, $Y_{p}^{\bullet}(\ast)$ is the Amitsur complex of the map $\hS \rightarrow \E$. The goal of the section is to study the formality of the ultraproduct of these Amitsur complexes at a non-principal ultrafilter. 

Now we construct the ultraproduct. Let $\cF$ be a non-principal ultrafilter on the set of primes. Consider the product
\[
\Prod{p \in \cP} Y_{p}^{\bullet} \colon \Z^{\otimes | \rhd} \lra{} \big ( \Prod{p \in \cP}  \Mod_{\hS_p} \big )^{\Delta}.
\]
Postcomposing with the canonical map
\[
\Prod{p \in \cP}  \Mod_{\hS_p} \xrightarrow{n \circ m \circ [-]} \Prod{\cF}^{\doubleflat} \Mod_{\hS_p}
\]
produces a lax functor
\[
\Z^{\otimes | \rhd} \lra{} \big ( \Prod{\cF}^{\doubleflat} \Mod_{\hS_p} \big )^{\Delta}.
\]
But \cref{prop:protomodules} implies that 
\[
\Prod{\cF}^{\doubleflat} \Mod_{\hS_p} \simeq \Mod_{\Prod{\cF} \hS_p}
\]
and, since $\Prod{\cF} \hS_p \simeq H\Z_{\cF}$, we have a lax functor
\[
\Z^{\otimes | \rhd} \lra{} \Mod_{H\Z_{\cF}}^{\Delta}.
\]
Since $\Z_{\cF}$ is a $\Q$-algebra, we may forget to $\Mod_{H\Q} \simeq \Sp_{\Q}$. This gives us a lax functor
\[
Y_{\cF}^{\bullet} \colon \Z^{\otimes | \rhd} \lra{} \Sp_{\Q}^{\Delta}.
\]
This object has been constructed so that 
\[
Y_{\cF}^{\bullet}(\ast) \simeq \Prod{\cF} \E^{\otimes_{\hS} \bullet + 1}.
\]

\begin{thm}\label{prop:cosimplicialformalobj}
There is an equivalence of cosimplicial $\bE_{\infty}$-rings 
\[
\Prod{\cF}\E^{\otimes_{\hS} \bullet +1} \simeq \Prod{\cF}(\E^{\otimes_{\hS} \bullet +1})_{\star}.
\]
\end{thm}
\begin{proof}
To prove that this cosimplicial $\bE_{\infty}$-ring $\Prod{\cF}\E^{\otimes_{\hS} \bullet +1} \simeq Y_{\cF}^{\bullet}(\ast)$ is formal, it suffices to show that $Y_{\cF}^{\bullet}$ satisfies the conditions of \cref{p:ratio-formal}. But this follows immediately from \cref{prop:Yp} and \cref{lem:ultra-weight}. 
\end{proof}

\section{Descent}

The goal of this section is to prove \Cref{thm:mainthmen}. The $\infty$-category $\Frnp$ appearing on the right hand side is constructed in \cref{ssec:algebraicmodel}. We write $\E$ for Morava $E$-theory at height $n$ and the prime $p$. 

\begin{thm}\label{thm:mainthmen}
For any non-principal ultrafilter $\cF$ on $\cP$, there is an equivalence of symmetric monoidal compactly generated rational stable $\infty$-categories
\[
\Prod{\cF}^{\Pic}\Sp_{n,p} \simeq \Prod{\cF}^{\Pic}\Frnp.
\]
\end{thm}

\subsection{Abstract descent}

We begin by recalling some facts regarding $\infty$-categorical descent theory. We then explore the relation between $\infty$-categorical descent and the $\Pic$-generated protoproduct. These operations do not commute in general. This subsection culminates in \cref{maincor}, which provides conditions under which the $\Pic$-generated protoproduct commutes with descent. This corollary will be applied on both the topological side and the algebraic side in order to prove \cref{thm:mainthmen}.

Suppose that $(\cC,\otimes,\mathbf{1})$ is a symmetric monoidal compactly generated stable $\infty$-category. If $A \in \cC$ is a commutative algebra object in $\cC$, we will write $A^{\otimes\bullet +1} \in \cC^{\Delta}$ for the Amitsur complex of $A$ defined in \cref{ss:reduce_formal}, i.e., the cosimplicial diagram
\[
\xymatrix{A \ar@<0.5ex>[r] \ar@<-0.5ex>[r] & A^{\otimes 2} \ar[l] \ar@<-1ex>[r] \ar[r]  \ar@<1ex>[r] & A^{\otimes 3} \ar@<-0.5ex>[l] \ar@<0.5ex>[l]   \ar@<-1ex>[r] \ar[r] \ar@<1ex>[r] \ar@<2ex>[r] & \cdots. \ar@<-1.5ex>[l] \ar@<-0.5ex>[l] \ar@<0.5ex>[l] }
\]
Similarly, $\Mod_{A^{\otimes\bullet +1}}(\cC)$ denotes the associated cosimplicial diagram of $\infty$-categories of modules over $A^{\otimes\bullet +1}$ in $\cC$:
\[
\xymatrix{\Mod_{A}(\cC) \ar@<0.5ex>[r] \ar@<-0.5ex>[r] & \Mod_{A^{\otimes 2}}(\cC) \ar[l] \ar@<-1ex>[r] \ar[r]  \ar@<1ex>[r] & \Mod_{A^{\otimes 3}}(\cC) \ar@<-0.5ex>[l] \ar@<0.5ex>[l]   \ar@<-1ex>[r] \ar[r] \ar@<1ex>[r] \ar@<2ex>[r] & \cdots. \ar@<-1.5ex>[l] \ar@<-0.5ex>[l] \ar@<0.5ex>[l] }
\]

In \cite[Proposition 3.22]{mathew_galois}, Mathew provides a condition under which $\cC$ can be recovered from this cosimplicial diagram: if the tower $\{\Tot^m(A^{\otimes\bullet +1})\}_{m\ge 0}$ of partial totalizations associated to the cosimplicial diagram $A^{\otimes\bullet +1}$ is pro-constant with limit $\mathbf{1}$, then the natural functor
\[
\xymatrix{\cC \ar[r]^-{\simeq} & \lim(\Mod_{A^{\otimes\bullet +1}}(\cC))}
\]
is an equivalence of symmetric monoidal $\infty$-categories. The limit here is calculated in $\Cat$. To prove a version of this result for protoproducts, we are inspired by Mathew's proof strategy.

Bousfield~\cite{bousfield_homolss} provides a convenient criterion for checking the assumption on the $\Tot$-tower, an $\infty$-categorical formulation of which can be found in \cite{mathew_galois}. To state it, we have to recall some auxiliary notation, first introduced in~\cite{hopkinsvanishinglines}. 

\begin{defn}
Let $\cC^{\N}$ be the $\infty$-category of towers of objects in $\cC$. A tower $\{Y_m\} \in \cC^{\N}$ is said to be strongly pro-constant if it satisfies the following condition:
\begin{itemize}
	\item[] Let $X = \lim_m Y_m$ and let $\{X\} \in \cC^{\N}$ denote the constant tower on $X$. Then the cofiber of the natural map $\{X\} \to \{Y_m\}$ is nilpotent, i.e., there exists an $r \ge 1$ such that any $r$-fold composite in this tower is $0$. 
\end{itemize}
The smallest integer $r\ge 1$ with this property will be called the nilpotence degree of $\{Y_m\}$.
\end{defn}

Note that, in particular, strongly pro-constant towers are pro-constant. 

\begin{defn}
A commutative algebra $A$ in $\cC$ has fast-degree $r$ if the $\Tot$-tower associated to the Amitsur complex is strongly pro-constant of nilpotence degree $r$. If there is no such natural number, then we will say the fast-degree of $A$ is $\infty$.
\end{defn}

\begin{defn}
If $A$ is a commutative algebra in $\cC$ such that $-\otimes A$ is conservative and the fast-degree of $A$ is less than $\infty$, then we will call $A$ descendable.
\end{defn}

We will make use of the following result, which follows from \cite[Corollary 4.4]{mathew_galois}. 
\begin{prop} [Mathew] \label{rem:desc}
If $A \in \cC$ is descendable, then, for any object $Y \in \cC$,
\[
Y \simeq \Tot(A^{\otimes \bullet +1} \otimes Y).
\]
\end{prop}

In general, if $X$ and $Y$ are objects in $\cC$ and $A$ is a commutative algebra in $\cC$, the cosimplicial spectrum
\[
\Hom(X, A^{\otimes \bullet +1} \otimes Y)
\]
gives rise to a spectral sequence computing $\Hom(X,Y)$. We call this spectral sequence the $A$-based Adams spectral sequence.

\begin{defn}
A commutative algebra $A$ in $\cC$ has vanishing-degree $r$ if for all $X, Y$ objects in $\cC$, the corresponding $A$-based Adams spectral sequence collapses at the $r$th page with a horizontal vanishing line of intercept $r$. If there is no such natural number, then we will say the vanishing-degree is $\infty$.
\end{defn}

In \cite[Proposition 3.12]{mathew_thick}, Mathew proves that the fast-degree is less than $\infty$ if and only if the vanishing-degree is less than $\infty$. A careful reading of his proof gives the following lemma:

\begin{lem} \label{bounds}
Let $A$ be a commutative algebra in $\cC$, let $v$ be its vanishing-degree and let $f$ be its fast-degree. Then
\[
v \leq f+1 \text{ and } f \leq 2v.
\]
\end{lem}

\begin{lem} \label{degrees}
Let $F \colon \cC \rightarrow \cD$ be a symmetric monoidal exact functor between presentably symmetric monoidal stable $\infty$-categories. Let $A$ be a commutative algebra in $\cC$ of fast-degree $r$. Then $F(A)$ is a commutative algebra in $\cD$ of fast-degree $r$.
\end{lem}
\begin{proof}
First note that $F$ sends nilpotent towers to nilpotent towers. The functor $F$ commutes with tensor powers, cofiber sequences, $\Tot^{n}$, and sends zero-maps to zero-maps. If $A$ is a commutative algebra in $\cC$ of fast degree $r$, it thus follows that the cofiber of the natural map 
\[
\{F(\Tot(A^{\otimes \bullet+1}))\} \to \{F(\Tot^n(A^{\otimes \bullet+1}))\} \simeq \{\Tot^n(F(A^{\otimes \bullet+1}))\}
\]
is nilpotent. Therefore, in order to show that $F(A)$ has fast-degree $r$, it suffices to prove that
\[
F(\Tot(A^{\otimes \bullet+1})) \simeq \Tot(F(A)^{\otimes \bullet+1}).
\]
Applying $F$ to the cofiber sequence of towers
\[
\{\Tot(A^{\otimes \bullet+1})\} \rightarrow \{\Tot^n(A^{\otimes \bullet+1})\} \rightarrow \{C_i\}
\]
gives the cofiber sequence of towers
\[
\{F(\Tot(A^{\otimes \bullet+1}))\} \rightarrow \{ \Tot^n(F(A)^{\otimes \bullet+1}) \} \rightarrow \{F(C_i)\}.
\]
Since the tower $\{F(C_i)\}$ is nilpotent, taking the inverse limit of the tower gives the cofiber sequence
\[
F(\Tot(A^{\otimes \bullet+1})) \rightarrow \Tot(F(A)^{\otimes \bullet+1}) \rightarrow 0,
\]
which proves the claim.
\end{proof}

\begin{prop} \label{mitchell}
Let $\cF$ be an ultrafilter on $I$ and let $(Y_{i}^{\bullet})_{i \in I}$ be a collection of cosimplicial spectra. Denote the Bousfield--Kan spectral sequence associated to $Y_{i}^{\bullet}$ by $E_{r}^{s,t}(Y_{i}^{\bullet})$, where $s$ is the cohomological grading. Assume that, for each $k \in \Z$,  there exists an $r_k$ and $s_k$ such that for each $i$ in some set $V$ in the ultrafilter $\cF$, $E_{r_k}^{s,t}(Y_{i}^{\bullet})=0$ for all $t-s=k$ and $s>s_k$ (ie. for each $k \in \Z$, the filtration in the $k$th column is finite for some set in the ultrafilter). This implies that the natural map
\[
\Prod{\cF}\Tot(Y_{i}^{\bullet}) \lra{} \Tot(\Prod{\cF}Y_{i}^{\bullet})
\]
is an equivalence.
\end{prop}
\begin{proof}
By intersecting each set in the ultrafilter $\cF$ with $V$, we may assume that $V = I$. We may rewrite the natural map above as
\[
\Colim{U \in \cF} \lim_{\bullet \in \Delta} \Prod{i \in U} Y_{i}^{\bullet} \lra{} \lim_{\bullet \in \Delta} \Colim{U \in \cF} \Prod{i \in U} Y_{i}^{\bullet}.
\]
We will apply Proposition 3.3 of \cite{mitchellthomasondescent} to show that it is an equivalence by verifying that the map satisfies the conditions used in the proof. In the notation of \cite{mitchellthomasondescent}, we will let 
\[
X_U = \lim_{\bullet \in \Delta} \Prod{i \in U} Y_{i}^{\bullet}
\]
for $U \in \cF$, and
\[
X = \lim_{\bullet \in \Delta} \Colim{U \in \cF} \Prod{i \in U} Y_{i}^{\bullet}.
\]
The spectral sequence we will use is the Bousfield--Kan spectral sequence. 

The first condition required for the proof is that there are no nontrivial elements in $\pi_* \Colim{U} X_U$ with infinite filtration. To satisfy this condition, in view of Remark 3.6 in \cite{mitchellthomasondescent}, it is enough to have that for each $k \in \Z$ and $U \in \cF$, the filtration in the $k$th column is finite  at some fixed page for the Bousfield--Kan spectral sequence associated to $\Prod{i \in U} Y_{i}^{\bullet}$. This is our assumption, so this first condition is satisfied. 

To satisfy the second condition we must show that there is an isomorphism
\[
\Colim{U \in \cF} H_q \pi_s \Prod{i \in U} Y_{i}^{\bullet} \lra{\cong} H_q \pi_s \Colim{U \in \cF} \Prod{i \in U} Y_{i}^{\bullet},
\]  
for all $q$ and $s$, but this follows from the fact that homology commutes with filtered colimits and that the sphere is compact. Thus the natural map is an equivalence. 
\end{proof}

Throughout the remainder of this section, the totalization of a cosimplicial diagram of compactly generated $\infty$-categories $\cE^{\bullet}$, denoted $\Tot(\cE^{\bullet})$, will always refer to the limit taken in the $\infty$-category of compactly generated $\infty$-categories.

\begin{prop}\label{prop:ultraenlocalcomparisonfunctor}
Let $(\cC_i)_{i \in I}$ be a collection of symmetric monoidal compactly generated stable $\infty$-categories with compact unit and let $(A_i)_{i \in I}$ be a collection of descendable objects such that there exists $r>0$ such that for all but finitely many $i \in I$, $A_i$ has fast-degree less than or equal to $r$. Then the canonical symmetric monoidal functor
\[
\xymatrix{\Prod{\cF}^{\Pic}\cC_i \ar[r]^-F & \Tot(\Prod{\cF}^{\Pic}\Mod_{A_{i}^{\otimes\bullet +1}}(\cC_i))}
\]
is fully faithful, where the totalization takes place in the $\infty$-category of compactly generated $\infty$-categories.
\end{prop}
\begin{proof}
By restricting to a smaller set in the ultrafilter, we may assume that conditions of the statement hold for all $i \in I$.

It suffices to check this on the mapping spectrum $\Hom([X_i],[Y_i])$ between compact objects $[X_i]$ and $[Y_i]$. We must show that the map
\[
\Hom_{\Prod{\cF}^{\Pic}\cC_i}([X_i],[Y_i]) \rightarrow \Hom_{\Tot(\Prod{\cF}^{\Pic}\Mod_{A_{i}^{\otimes\bullet +1}}(\cC_i))}(F[X_i],F[Y_i])
\]
is an equivalence. We will identify this with the natural map
\[
\Prod{\cF}\Tot(\Hom_{\cC_i}(X_i,A_{i}^{\otimes \bullet + 1} \otimes Y_i)) \lra{} \Tot(\Prod{\cF}\Hom_{\cC_i}(X_i,A_{i}^{\otimes \bullet+1} \otimes Y_i)),
\]
which is given by the composition
\begin{align*}
\Prod{\cF}\Tot(\Hom_{\cC_i}(X_i,A_{i}^{\otimes \bullet + 1} \otimes Y_i)) \simeq & \Prod{\cF}\Hom_{\cC_i}(X_i,Y_i) \\ 
\simeq & \Hom_{\Prod{\cF}^{\Pic}\cC_i}([X_i],[Y_i])\\ 
\rightarrow &\Hom_{\Tot(\Prod{\cF}^{\Pic}\Mod_{A_{i}^{\otimes\bullet +1}}(\cC_i))}(F[X_i],F[Y_i]) \\ 
\simeq & \Hom_{\Tot(\Prod{\cF}^{\Pic}\Mod_{A_{i}^{\otimes\bullet +1}}(\cC_i))}([A_{i}^{\otimes \bullet + 1} \otimes X_i],[A_{i}^{\otimes \bullet + 1} \otimes Y_i])\\
\simeq & \Tot \Hom_{\Prod{\cF}^{\Pic}\Mod_{A_{i}^{\otimes\bullet +1}}(\cC_i)}([A_{i}^{\otimes \bullet + 1} \otimes X_i],[A_{i}^{\otimes \bullet + 1} \otimes Y_i])\\
\simeq & \Tot \Prod{\cF}\Hom_{\Mod_{A_{i}^{\otimes\bullet +1}}(\cC_i)}(A_{i}^{\otimes \bullet + 1} \otimes X_i,A_{i}^{\otimes \bullet + 1} \otimes Y_i)\\
\simeq & \Tot \Prod{\cF} \Hom_{\cC_i}(X_i,A_{i}^{\otimes \bullet + 1} \otimes Y_i).
\end{align*}
The first equivalence follows from \cref{rem:desc}. The second equivalence follows by the compactness of $[X_i]$ and $[Y_i]$ and \Cref{lem:rozenblyum}. The third equivalence is the definition of $F$. The fourth equivalence follows from general facts regarding mapping spectra in limits of $\infty$-categories. The fifth equivalence follows from the compactness of $[A_{i}^{\otimes k} \otimes X_i]$ and $[A_{i}^{\otimes k} \otimes Y_i]$ and \Cref{lem:rozenblyum}. The last equivalence follows from the free-forgetful adjunction.

We apply \cref{mitchell} to the collection of cosimplicial spectra
\[
(\Hom_{\cC_i}(X_i,A_{i}^{\otimes \bullet + 1} \otimes Y_i))_{i \in I}.
\]
By our assumption on the collection $(A_i)_{i \in I}$ and \cref{bounds}, $(A_i)_{i \in I}$ has bounded vanishing-degree. \cref{mitchell} implies the map is an equivalence.
\end{proof}

\begin{rem}
The previous result holds for the protoproduct, bounded protoproduct, and $\Pic$-generated protoproduct.
\end{rem}

Recall that the Picard spectrum of a symmetric monoidal $\infty$-category $\cC$, $\pic(\cC)$, is the spectrum associated to the symmetric monoidal $\infty$-groupoid of invertible objects. As a functor to connective spectra, $\pic$ preserves all limits and filtered colimits of symmetric monoidal $\infty$-categories by \cite[Proposition 2.2.3]{ms_picard}.

The Picard space functor is denoted by $\Pic$; it is related to the spectrum-valued functor $\pic$ by a canonical equivalence $\Pic \simeq \Omega^{\infty}\pic$. Given a symmetric monoidal $\infty$-category $\cC$, we will write $\Loc \Pic(\cC)$ for the localizing subcategory generated by $\Pic(\cC)$.

Let $\iota_0 \colon \Sp_{\geq 0} \leftrightarrows \Sp \noloc \tau_{\geq 0}$ be the inclusion/truncation adjunction between the $\infty$-categories of spectra and connective spectra. When we want to consider the Picard spectrum functor as taking values in the category of spectra, we will write $\iota_0 \pic$.

The goal of the next proposition is to identify the essential image of the functor $F$ constructed in \Cref{prop:ultraenlocalcomparisonfunctor} as the localizing subcategory on $\Pic(\cC)$. 

\begin{lem} \label{piccommutes}
Let $(\cC_i)_{i \in I}$ be a collection of $\Pic$-compactly generated symmetric monoidal $\infty$-categories. Then there is a canonical equivalence of connective spectra
\[
\pic(\Prod{\cF}^{\Pic}\cC_i) \simeq \Prod{\cF} \pic \cC_i.
\]
\end{lem}
\begin{proof}
This follows from the proof of \cref{lem:picprotogp}.
\end{proof}

\begin{prop} \label{picequiv}
Let $(\cC_i)_{i \in I}$ and $(A_i)_{i \in I}$ be as in \cref{prop:ultraenlocalcomparisonfunctor} and also assume that the $A_i$-based Adams spectral sequence for $\End(1_{\cC_i}) = \Hom(1_{\cC_i},1_{\cC_i})$ collapses at the $E_2$-page on some set in the ultrafilter $\cF$. The canonical functor
\[
\xymatrix{\Prod{\cF}^{\Pic}\cC_i \ar[r]^-F & \Tot(\Prod{\cF}^{\Pic}\Mod_{A_{i}^{\otimes\bullet +1}}(\cC_i))}
\]
induces an equivalence of spaces
\[
\xymatrix{\Pic(\Prod{\cF}^{\Pic}\cC_i) \ar[r]^-{\simeq} & \Pic(\Tot(\Prod{\cF}^{\Pic}\Mod_{A_{i}^{\otimes\bullet +1}}(\cC_i))).} 
\]
\end{prop}
\begin{proof}
 Applying $\pic$ to the symmetric monoidal functor $F$ gives a map of connective spectra
\[
\xymatrix{\pic(\Prod{\cF}^{\Pic}\cC_i) \ar[r] & \pic(\Tot(\Prod{\cF}^{\Pic}\Mod_{A_{i}^{\otimes\bullet +1}}(\cC_i))).} 
\] 
As a functor to connective spectra, $\pic$ commutes with totalizations of symmetric monoidal $\infty$-categories by \cite[Proposition 2.2.3]{ms_picard}. Under the assumption that all of the $\cC_i$ are compactly generated with compact unit, $\pic$ commutes with totalizations of compactly generated symmetric monoidal $\infty$-categories with compact unit, because invertible objects in these $\infty$-categories are compact. Now \cref{piccommutes} implies that the map above is equivalent to the map
\[
\Prod{\cF}\pic(\cC_i) \simeq \tau_{\geq 0}\Prod{\cF} \Tot (\iota_0 \pic(\Mod_{A_{i}^{\otimes\bullet +1}}(\cC_i)))  \rightarrow \tau_{\geq 0}\Tot(\Prod{\cF}\iota_0\pic(\Mod_{A_{i}^{\otimes\bullet +1}}(\cC_i))), 
\]
where the totalizations are taking place in the $\infty$-category of spectra. Thus it suffices to prove that the map
\[
\Prod{\cF} \Tot (\iota_0 \pic(\Mod_{A_{i}^{\otimes\bullet +1}}(\cC_i)))  \rightarrow \Tot(\Prod{\cF}\iota_0\pic(\Mod_{A_{i}^{\otimes\bullet +1}}(\cC_i)))
\]
is an equivalence of spectra for which we would like to apply \cref{mitchell}.

For this we would like to show that the Bousfield--Kan spectral sequence associated to the cosimplicial spectrum 
\[
\iota_0 \pic(\Mod_{A_{i}^{\otimes\bullet +1}}(\cC_i))
\]
satisfies the conditions of \cref{mitchell}. We will compare this spectral sequence to the spectral sequence associated to the cosimplicial spectrum
\[
\Hom(1_{\cC_i},A_{i}^{\otimes\bullet +1}).
\]
We will use Adams grading, so in particular a column means the collection of groups $E_{2}^{s,t}$, where $t-s$ is fixed. For each $k \in \Z$, the two spectral sequences agree up to a shift at all but finitely many places in the $k$th column independent of $i \in U$. This follows from the fact that the homotopy groups of the two cosimplicial spectra that we are considering agree up to a shift after taking the $2$-connected cover. Thus for each column, the discrepency between the two spectral sequences appears in at most two entries. Finally, by assumption, the spectral sequence associated to the cosimplicial spectrum
\[
\Hom(1_{\cC_i},A_{i}^{\otimes\bullet +1})
\]
satisfies the condition of \cref{mitchell}, and thus the spectral sequence associated to the cosimplicial spectrum 
\[
\iota_0 \pic(\Mod_{A_{i}^{\otimes\bullet +1}}(\cC_i))
\]
satisfies the condition of \cref{mitchell} as well.
\end{proof}

\begin{rem}
It is possible to prove \cref{picequiv} under the weaker assumption that the spectral sequence collapses at the $E_r$-page on some set in the ultrafilter. However, the proof requires a more careful comparison between the $A_i$-based spectral sequence of $\pic(\cC_i)$ and $\End(1_{\cC_i})$. A careful comparison between the differentials in these two spectral sequences appears in \cite{ms_picard}. However, the proof of this stronger version of \cref{picequiv} would require a lengthy and ultimately unnecessary digression.
\end{rem}

\begin{thm}
Let $(\cC_i)_{i \in I}$ and $(A_i)_{i \in I}$ be as in the statement of \cref{picequiv}. There is a canonical equivalence of symmetric monoidal stable $\infty$-categories 
\[
\Prod{\cF}^{\Pic} \cC_i \simeq \Loc \Pic(\Tot(\Prod{\cF}^{\Pic}\Mod_{A_{i}^{\otimes\bullet +1}}(\cC_i))).
\]
\end{thm}
\begin{proof}
There is a canonical colimit preserving functor
\[
\Prod{\cF}^{\Pic} \cC_i \rightarrow \Tot(\Prod{\cF}^{\Pic}\Mod_{A_{i}^{\otimes\bullet +1}}(\cC_i)),
\]
which is fully faithful by \cref{prop:ultraenlocalcomparisonfunctor}. 

By \cref{picequiv}, the map is an equivalence after applying $\Pic(-)$. \cref{thisguy} implies that
\[
\Prod{\cF}^{\Pic} \cC_i \simeq \Loc \Pic(\Prod{\cF}^{\Pic} \cC_i).
\]
Therefore, we get the desired equivalence.
\end{proof}

\begin{lem} \label{lem:abstractcosimp}
Let $\cC^{\bullet}$ and $\cD^{\bullet}$ be cosimplicial $\infty$-categories and let $f^{\bullet} \colon \cC^{\bullet} \rightarrow \cD^{\bullet}$ be a map that is degree-wise fully faithful. If $f^0 \colon \cC^0 \rightarrow \cD^0$ is an equivalence, then 
\[
\Tot f^{\bullet} \colon \Tot \cC^{\bullet} \rightarrow \Tot \cD^{\bullet}
\]
is an equivalence.
\end{lem}
\begin{proof}
We will begin by proving the result in the case that $\cC^{\bullet}$ and $\cD^{\bullet}$ are cosimplicial $\infty$-groupoids. Let $F^{\bullet}$ be a fiber of the map $f^{\bullet}$. Since $f^{[k]}$ is fully faithful, the $\infty$-groupoid $F^{[k]}$ is either contractible or empty. Since $f^0$ is an equivalence $F^{[0]}$ is contractible. This immediately implies that $F^{[k]}$ must be non-empty and contractible for all $k \geq 0$. Since totalization commutes with taking the fiber and since the totalization of a cosimplicial diagram of contractible spaces is contractible, we have that every fiber of $\Tot f^{\bullet}$ is contractible. This implies that $\Tot f^{\bullet}$ is an equivalence.

For the general case, it is enough to show that for every $\infty$-category $T$ the map
\[
\Tot (\Map_{\Cat}(T,\cC^{\bullet})) \to \Tot (\Map_{\Cat}(T,\cD^{\bullet}))
\]
is an equivalence. Since 
\[
\Map_{\Cat}(\cC,\cD) \simeq \Fun(\cC,\cD)^{\simeq}
\]
and since $\Fun(T,\bullet)$ preserves fully faithful maps and categorical equivalences, the lemma is now reduced to the $\infty$-groupoid case.
\end{proof}

Recall that limits in the $\infty$-category of presentable $\infty$-categories and in the $\infty$-category of presentably symmetric monoidal $\infty$-categories may be computed in $\Cat$. \cref{lem:abstractcosimp} makes it clear that a further constraint on the collection $(A_i)_{i \in I}$ leads to a close relationship between the $\Pic$-generated protoproduct and the protoproduct.

\begin{cor} \label{picandcell}
Let $(\cC_i)_{i \in I}$ and $(A_i)_{i \in I}$ be as in the statement of \cref{picequiv} and also assume that the canonical map
\[
\Prod{\cF}^{\flat}\Mod_{A_i}(\cC_i) \lra{\simeq} \Prod{\cF}^{\Pic}\Mod_{A_i}(\cC_i)
\]
is an equivalence. Then there is an equivalence of symmetric monoidal $\infty$-categories
\[
\Tot(\Prod{\cF}^{\Pic}\Mod_{A_{i}^{\otimes\bullet +1}}(\cC_i)) \simeq \Tot(\Prod{\cF}^{\flat}\Mod_{A_{i}^{\otimes\bullet +1}}(\cC_i))
\]
and thus an equivalence of symmetric monoidal $\infty$-categories
\[
\Prod{\cF}^{\Pic} \cC_i \simeq \Loc \Pic(\Tot(\Prod{\cF}^{\flat}\Mod_{A_{i}^{\otimes\bullet +1}}(\cC_i))).
\]
\end{cor}

We say that a symmetric monoidal compactly generated $\infty$-category $\cC$ is monogenic if $\mathbf{1}$ is a compact generator of $\cC$. 

\begin{cor} \label{maincor}
Let $(\cC_i)_{i \in I}$ and $(A_i)_{i \in I}$ be as in the statements of \cref{picandcell} and assume that $\cC_i$ is monogenic for all $i \in I$. Then there is a canonical equivalence of symmetric monoidal $\infty$-categories
\[
\Prod{\cF}^{\Pic} \cC_i \lra{\simeq} \Loc \Pic \Tot(\Prod{\cF}^{\flat}\Mod_{\Hom(1_{\cC_i},A_{i}^{\otimes\bullet +1})}).
\]
\end{cor}
\begin{proof}
Let $i \in I$ and $k\ge 0$. The monogenicity assumption on $\cC_i$ implies that the symmetric monoidal $\infty$-category $\Mod_{A_{i}^{\otimes k +1}}(\cC_i)$ is compactly generated by its unit $A_{i}^{\otimes k +1}$, whose endomorphism ring spectrum is readily computed as 
\[
\End_{\Mod_{A_{i}^{\otimes k +1}}(\cC_i)}(A_{i}^{\otimes k +1}) \simeq \Hom_{\cC_i}(1_{\cC_i},A_{i}^{\otimes k +1}).
\]
It thus follows from Schwede--Shipley's version of Morita theory~\Cref{thm:moritatheory} that there is a symmetric monoidal equivalence 
\[
\Mod_{A_{i}^{\otimes \bullet +1}}(\cC_i) \simeq \Mod_{\Hom(1_{\cC_i},A_{i}^{\otimes\bullet +1})},
\]
where the right hand side refers to modules in the $\infty$-category of spectra. For varying $k$, these equivalences are compatible with the cosimplicial structure maps, thus inducing a symmetric monoidal equivalence
\[
\Mod_{A_{i}^{\otimes\bullet +1}}(\cC_i) \simeq \Mod_{\Hom(1_{\cC_i},A_{i}^{\otimes\bullet +1})}.
\]
Combining this with \Cref{picandcell}, we obtain the desired equivalence. 
\end{proof}

\subsection{Descent for the $E$-local categories}

The goal of this subsection is to show that $(\cC_p)_{p \in \cP} = (\Sp_{n,p})_{p \in \cP}$ and $(A_p)_{p \in \cP} = (\E)_{p \in \cP}$ satisfy the conditions of \cref{maincor}. It is clear that $\Sp_{n,p}$ is a symmetric monoidal monogenic stable $\infty$-category. It remains to show that $\E$ has finite vanishing-degree independent of $p$.

Let $E = \E$ and let $\Mfgn(p)$ be the $p$-local moduli stack of formal groups of height less than or equal to $n$. In \cite{moravastack}, Morava shows that the stabilizer group has finite cohomological dimension $n^2$ for large enough primes. Via the chromatic spectral sequence, this implies the following result, a proof of which can be found for example in \cite[Theorem 3.4.9]{Franke_exotic}.

\begin{prop}(Franke, Morava)\label{prop:cohomdim}
The cohomological dimension of the stack $\Mfgn(p)$ is $n^2+n$ for all $p> n+1$. 
\end{prop}

The $E$-based Adams spectral sequence is a spectral sequence of the form
\[
\Ext_{E_*E}^{*}(E_*(X),E_*(Y)) \implies \pi_*\Hom_{\Sp_{E}}(X,Y).
\]
The finite cohomological dimension of $\Mfgn(p)$ implies that, for any $X$ and $Y$, there is a horizontal vanishing on the $E_2$-page of the spectral sequence with prime-independent intercept (for large enough primes) of the vanishing line. For this conclusion, see \cite[Proof of Theorem 5.4]{hoveysadofskyinv}. 

\begin{prop}\label{thm:hopkinsravenelsmash}
For any height $n$ and for $p$ large enough with respect to $n$, there exists a constant $s_0$, independent of $p$, such that the $\E$-based Adams spectral sequence for any spectrum $Y$ has a horizontal vanishing line of intercept $s_0$ at the $E_2$-page.
\end{prop}

\begin{rem} \label{rem:hoprav}
Hopkins and Ravenel (\cite[Chapter 8]{ravbook2}) proved that there is an $r \in \N$ so that a horizontal vanishing line occurs on the $E_r$-page independent of the choice of prime.
\end{rem}

\begin{lem} \label{lem:br}
For any ultrafilter $\cF$ on $\cP$, there is an equivalence of symmetric monoidal $\infty$-categories
\[
\Prod{\cF}^{\Pic}\Mod_{\E} \simeq \Prod{\cF}^{\flat}\Mod_{\E}.
\]
\end{lem}
\begin{proof}
Baker--Richter \cite{bakerrichter_inv} prove that the Picard group of $\Mod_{\E}$ is $\Z/2$ (represented by $\E$ and $\Sigma \E$). Thus the $\Pic$-filtration and the cell filtration on $\Mod_{\E}$ agree and the protoproducts are equivalent.
\end{proof}

Thus we may take $(\cC_p)_{p \in \cP} = (\Sp_{n,p})_{p \in \cP}$ and $(A_p)_{p \in \cP} = (\E)_{p \in \cP}$ in \cref{maincor} and conclude the following: 
\begin{cor} \label{topequiv}
For any ultrafilter $\cF$ on $\cP$, there is a canonical symmetric monoidal equivalence of symmetric monoidal $\infty$-categories
\[
\Prod{\cF}^{\Pic} \Sp_{n,p} \simeq \Loc \Pic(\Tot(\Prod{\cF}^{\flat}\Mod_{\E^{\otimes\bullet +1}})).
\]
\end{cor}

\subsection{The algebraic model} \label{ssec:algebraicmodel}
Let $E = \E$. We produce an $\infty$-category that is an algebraic analogue of $\Sp_{n,p}$ out of a model category of quasi-periodic complexes of $E_0E$-comodules first described in \cite{Franke_exotic}.

Let $(A, \Gamma)$ be an Adams Hopf algebroid in the sense of \cite[Section 1.4]{hoveyhcomodules}, i.e., a cogroupoid object in the category of commutative rings such that $\Gamma$ is flat over $A$ and satisfies a certain technical condition. Let $(A, \Gamma)-\comod$ be the 1-category of $(A,\Gamma)$-comodules. An introduction to this category is given in \cite{hoveyhcomodules, ravbook1}. In particular $(A, \Gamma)-\comod$ is a Grothendieck abelian category. Furthermore, $(A,\Gamma)-\Comod$ has a natural symmetric monoidal structure $\otimes_A = \otimes$ with unit $A$, which is compatible with the usual symmetric monoidal structure on $\Mod_A$. Therefore, we get a symmetric monoidal adjunction
\[
\xymatrix{\text{forget}\colon (A,\Gamma)-\Comod \ar@<0.5ex>[r] & A-\Mod \noloc -\otimes_A \Gamma. \ar@<0.5ex>[l]}
\]
An $(A, \Gamma)$-comodule $L$ is invertible if the underlying $A$-module is invertible as an $A$-module.  

Let $\cA$ be a symmetric monoidal Grothendieck abelian category and let $L$ be an invertible object in $\cA$. Associated to this data, Barnes and Roitzheim~\cite{barnesroitzheim_monoidality} construct a category of quasi-periodic chain complexes of objects in $\cA$. We recall their construction in the case $\cA = (A, \Gamma)-\Comod$. Given an invertible $(A, \Gamma)$-comodule $L$, let 
\[
C^{(L,2)}((A,\Gamma)-\comod)
\]
be the category of quasi-periodic chain complexes of $(A, \Gamma)$-comodules. The objects of this category are pairs $(X,q)$, where $X$ is an unbounded complex of $(A,\Gamma)$-comodules and $q$ is an isomorphism
\[
q \colon X[2] \cong X \otimes_{A} L.
\]
Morphisms in $C^{(L,2)}((A,\Gamma)-\comod)$ are maps of complexes that respect the fixed isomorphism.

Let $\Ch((A,\Gamma)-\comod)$ be the category of complexes of $(A,\Gamma)$-comodules. By \cite[Lemma 1.2]{barnesroitzheim_monoidality}, there is an adjunction
\[
\xymatrix{P\colon \Ch((A,\Gamma)-\comod) \ar@<0.5ex>[r] & C^{(L,2)}((A,\Gamma)-\comod) \noloc U, \ar@<0.5ex>[l]}
\]
where $U$ is the forgetful functor and $P$ is the ``periodization" functor
\[
P(Y) = \Oplus{k \in \Z} (Y\otimes L^{\otimes k}[2k]). 
\]
There is a natural symmetric monoidal structure on $C^{(L,2)}((A,\Gamma)-\comod)$, the periodized tensor product, such that $P$ is symmetric monoidal. In particular, the unit of this symmetric monoidal structure is $P(A)$. Moreover, we obtain a projection formula in this context:

\begin{lem}\label{lem:projformula}
Let $\cA$ be a symmetric monoidal Grothendieck abelian category. In the situation of the above adjunction, for any $X \in C^{(L,2)}(\cA)$ and $Y \in \Ch(\cA)$, the natural morphism
\begin{equation}\label{eq:projformula}
\xymatrix{UX \otimes Y \ar[r] &U(X \otimes P(Y)),}
\end{equation}
is an equivalence. Here, the tensor product $\otimes$ has to be interpreted in the corresponding categories.
\end{lem}
\begin{proof}
This is a general categorical fact. The canonical map is constructed as the adjoint of the composite
\[
\xymatrix{P(U(X) \otimes Y) \simeq PU(X) \otimes P(Y)  \ar[r] & X \otimes P(Y),}
\]
using the counit of the adjunction $(P,U)$ and the fact that $P$ is symmetric monoidal. To check it is an equivalence, it suffices to consider $Y$ a compact generator of $\Ch(\cA)$, as both $P$ and $U$ preserve arbitrary colimits. Such $Y$ can be taken to be comodules which are finite free $E_0$-modules, from which the claim follows immediately.
\end{proof}

Following Barnes and Roitzheim~\cite{barnesroitzheim_monoidality}, who were building on work of Franke~\cite{Franke_exotic} and Hovey~\cite{hoveyhcomodules}, we are now ready to construct the model category which gives rise to our algebraic model for the $E$-local category.

\begin{thm}\label{thm:quasiproj}
Let $L$ be an invertible object in $(E_0,E_0E)-\Comod$. There is a model structure on the category of quasi-periodic chain complexes of $E_0E$-comodules,
\[
C^{(L,2)}((E_0,E_0E)-\comod),
\]
whose weak equivalences are quasi-isomorphisms of the underlying maps of chain complexes, and which satisfies the following properties:
\begin{enumerate}
	\item The resulting model category is cofibrantly generated, proper, stable, and symmetric monoidal.
	\item There is a Quillen adjunction
	\[
	U \colon C^{(L,2)}((E_0,E_0E)-\comod) \leftrightarrows \Ch((E_0,E_0E)-\comod) \noloc P
	\]
	with symmetric monoidal left adjoint.
	\item For primes $p>n+1$, this model category is monogenic, i.e., compactly generated by its tensor unit $P(E_0)$. 
	\item For primes $p>n+1$, the object $P(E_0E)$ is faithfully flat. 
	\item For primes $p>n+1$ and for all objects $X$, $P(E_0E) \otimes X$ is fibrant.
\end{enumerate}
We will refer to this model structure as the quasi-projective model structure.
\end{thm}
\begin{proof}
To ease notation, for the remainder of this proof we will write $\cA$ for the Grothendieck abelian category $(E_0,E_0E)-\comod$ of $(E_0,E_0E)$-comodules. The required model structure was constructed by Barnes and Roitzheim \cite{barnesroitzheim_monoidality}, building on earlier work of Hovey \cite{hoveyhcomodules}. We start by recalling the data of the quasi-projective model structure. To this end, we need to introduce an auxiliary class of morphisms. 

Let $\cP$ be the set of (representatives of) comodules $M \in  \cA$ such that the underlying $E_0$-module is finitely presented and projective; note that these are precisely the compact objects of the abelian category $\cA$. A map $f\colon X \to Y$ in $C^{(L,2)}(\cA)$ is called a $\cP$-fibration or a $\cP$-equivalence if, for all $P \in \cP$,  $\cA(P,f)$ (using the notation of \cite{barnesroitzheim_monoidality}) is a degree-wise surjection or quasi-isomorphism, respectively. The class of $\cP$-cofibrations is defined to be the collection of those morphisms in $C^{(L,2)}(\cA)$ that have the left lifting property with respect to all $\cP$-fibrations which are also $\cP$-equivalences. In fact, these classes form the $\cP$-model structure. The weak equivalences, cofibrations, and fibrations of the quasi-projective model structures are obtained from the $\cP$-model structure via left Bousfield localization along the quasi-isomorphisms; explicitly:
\begin{itemize}
	\item A morphism $f\colon X \to Y$ in $C^{(L,2)}(\cA)$ is a weak equivalence if and only if the underlying morphism $U(f)$ of chain complexes is a quasi-isomorphism
	\item The cofibrations are the $\cP$-cofibrations.
	\item The fibrations are then determined by the previous two classes. 
\end{itemize}
The analogous construction can be carried out in $\Ch(\cA)$ and this gives rise to the quasi-projective model structure. In particular, the forgetful functor $U\colon C^{(L,2)}(\cA) \to \Ch(\cA)$ preserves fibrations and weak equivalences. Claim (1) is then precisely the content of \cite[Theorem 6.9]{barnesroitzheim_monoidality}. The free-forgetful adjunction 
\[
\xymatrix{P\colon \Ch(\cA) \ar@<0.5ex>[r] & C^{(L,2)}(\cA) \ar@<0.5ex>[l]\noloc U}
\]
can be promoted to a Quillen adjunction by \cite[Theorem 6.5]{barnesroitzheim_monoidality}.

In order to prove Claims (2), (3), and (4), we compare the quasi-projective model structure to Hovey's homotopy model structure on $\Ch(\cA)$. In general, this model structure sits between the $\cP$-model structure and the quasi-projective model structure on $\Ch(\cA)$, in the sense that we have inclusions
\[
(\cP-\text{equivalences}) \subseteq (\text{homotopy equivalences}) \subseteq (\text{quasi-isomorphisms}).
\]
For the details of its construction, we refer to \cite{hoveyhcomodules}. Using the finite cohomological dimension of $\Mfgn(p)$ for $p>n+1$, \Cref{prop:cohomdim}, it can be shown that the homotopy model structure and the quasi-projective model structure on $\Ch(\cA)$ coincide for $p>n+1$, as shown in the proof of Theorem 4.11 in \cite{bhacht}. This allows us to import the key properties of the homotopy model structure proven by Hovey to the quasi-projective model structure. 

From now on, assume that $p>n+1$. Since $E$ is a Landweber exact ring spectrum of finite height $n$, \cite[Corollary 6.7]{hoveyhcomodules} implies that the homotopy model category and hence the quasi-projective model category $\Ch(\cA)$ is compactly generated by its unit $E_0$. By adjunction, it follows that the same is true for the quasi-projective model structure on $C^{(L,2)}(\cA)$, with the compact generator given by the tensor unit $P(E_0)$. This proves Claim (2).

Recall that an object $M\in C^{(L,2)}(\cA)$ is faithfully flat (with respect to the quasi-projective model structure) whenever the underived endofunctor $M\otimes -$ preserves and reflects weak equivalences. Note that the tensor product $\otimes$ is balanced, so it does indeed suffice to work with the functor $M\otimes -$. Now suppose $f\colon X \to Y$ is a weak equivalence in $C^{(L,2)}(\cA)$ and consider the induced morphism
\[
\xymatrix{P(E_0E)\otimes f\colon P(E_0E) \otimes X \ar[r] & P(E_0E) \otimes Y.}
\]
Using the projection formula \eqref{eq:projformula}, this map is a weak equivalence if and only if $E_0E\otimes Uf$ is a weak equivalence in $\Ch(\cA)$. This is the case if $E_0E\otimes Uf$ is a quasi-isomorphism and this can be detected in $\Ch(E_0-\Mod)$ (with the projective model structure). Therefore, the flatness claim reduces to the analogous statement in $\Ch(E_0-\Mod)$, where it is clear. 

By a long exact sequence argument, we can reduce to showing that, for any $X \in C^{(L,2)}(\cA)$, $H_*X=0$ if $H_*(P(E_0E) \otimes X)=0$. Applying the projection formula \eqref{eq:projformula} to the assumption yields $H_*(E_0E \otimes UX) =0$. Therefore, 
\[
0 \simeq \mathrm{Rhom}_{\Ch(\cA)}(E_0, E_0E \otimes UX) \simeq \mathrm{Rhom}_{\Ch(E_0-\Mod)}(E_0, UX),
\]
thus $H_*(X) = 0$ as claimed. Taken together, this implies that $P(E_0E)$ is faithfully flat.

Finally, to show that $P(E_0E)\otimes X$ is fibrant for all $X$ (Claim (5)), we use a special case of Hovey's fibrancy criterion \cite[Theorem 5.2.3]{hoveyhcomodules}: Any complex of relative injective comodules is fibrant in the homotopy model structure on $\Ch(\cA)$. Indeed, $P(E_0E)\otimes X$ is fibrant in the quasi-projective model structure on $C^{(L,2)}(\cA)$ if $U(P(E_0E) \otimes X) = E_0E \otimes UX$ is fibrant in the homotopy model structure on $\Ch(\cA)$. But $E_0E$ is a relative injective comodule by \cite[Lemma 3.1.3]{hoveyhcomodules}.
\end{proof}

Our algebraic analogue of the $\infty$-category $\Sp_{n,p}$ is the underlying $\infty$-category of such a model category. 

\begin{defn}
Let $L = \pi_2E$. We define $\Frnp = C^{(L,2)}((E_0,E_0E)-\comod)^{c}[W^{-1}]$ to be the underlying symmetric monoidal $\infty$-category of $C^{(L,2)}((E_0,E_0E)-\comod)$ in the sense of \cite[Example 4.1.3.6]{ha}.
\end{defn}

We will also make use of 
\[
\Comod_{n,p} = \Ch((E_0,E_0E)-\comod)^{c}[W^{-1}],
\]
the $\infty$-category of $E_0E$-comodules. Let $\underline{E_0E}$ be the object in $\Comod_{n,p}$ which is the image of $E_0E \in \Ch((E_0,E_0E)-\comod)$ under the localization functor. Since the localization is lax and $E_0E$ is a commutative algebra in $\Ch((E_0,E_0E)-\comod)$, $\underline{E_0E}$ is a commutative algebra in $\Comod_{n,p}$. 

Applying these definitions, we get the following $\infty$-categorical corollary of \cref{thm:quasiproj}.

\begin{cor}\label{cor:franke}
Assume that $p> n+1$. The $\infty$-categories $\Frnp$ and $\Comod_{n,p}$ are monogenic presentably symmetric monoidal stable $\infty$-categories. Moreover, there is an induced adjunction of $\infty$-categories
\[
\xymatrix{\underline{P} \colon \Comod_{n,p} \ar@<0.5ex>[r] & \Frnp \noloc \underline{U}, \ar@<0.5ex>[l]}
\]
in which $\underline{P}$ is symmetric monoidal. Finally, the spectrum of maps between $\underline{E_0}$ and $\underline{Y} \otimes \underline{E_0E}^{\otimes k}$ for any $Y \in \Ch((E_0,E_0E)-\comod)$ can be computed as 
\[
\hom_{\Ch((E_0,E_0E)-\comod)}(E_0,Y\otimes {E_0E}^{\otimes k})
\] 
the chain complex of maps in the model category $\Ch((E_0,E_0E)-\comod)$.
\end{cor}
\begin{proof}
Presentability follows from \cite[Proposition 1.3.4.22]{ha}. \cref{thm:quasiproj} implies that the $\infty$-category $\Frnp$ is symmetric monoidal, stable, and monogenic (see \cite[Section 4.1.3]{ha}). \cite{mazelgee_adj} implies that Quillen adjunctions induce adjunctions of $\infty$-categories. \cite[Proposition 3.2.2]{Hinich} shows that since $P$ is symmetric monoidal, the induced $\underline{P}$ is symmetric monoidal. The last sentence of the statement follows from the fact that $E_0$ is cofibrant and $Y \otimes E_0E^{\otimes k}$ is fibrant.
\end{proof}

\begin{rem}
It is worth noting that Franke's category $\Frnp$ is not equivalent to $\Sp_{n,p}$ as a symmetric monoidal $\infty$-category for $n\ge 1$ and any prime $p$. One way to see this is based on the following observation: On the one hand, since $\Frnp$ is constructed from a category of chain complexes, it is an $H\Z$-linear category. On the other hand, the internal mapping objects of $\Sp_{n,p}$ are $E_{n,p}$-local, so if they were also $H\Z$-linear, then they would have to be rational, forcing $n=0$.
\end{rem}

\subsection{Descent for Franke's categories}
The goal of this subsection is to show that $(\cC_p)_{p \in \cP} = (\Frnp)_{p \in \cP}$ and 
\[
(A_p)_{p \in \cP} = (\underline{P}(\underline{\pi_0(\E \otimes \E)}))_{p \in \cP}
\]
satisfy the conditions of \cref{maincor}.

Let $(B_p, \Sigma_p) = (\pi_0(\E), \pi_0(\E \otimes \E))$ and let 
\[
A_p = \underline{P}(\underline{\pi_0(\E \otimes \E)}) = \underline{P}(\underline{\Sigma_p}).
\] 
Since $\underline{P}$ is symmetric monoidal, $A_p$ is a commutative algebra in $\Frnp$. By \cref{thm:quasiproj}, since $\underline{P}(\underline{E_0E})$ is faithfully flat, $-\otimes A_p$ is conservative.

\begin{lem}
The collection $(A_p)_{p \in \cP}$ satisfies the conditions of \cref{prop:ultraenlocalcomparisonfunctor}.
\end{lem}
\begin{proof}
By \cref{degrees}, it is enough to show that the collection $(\underline{\Sigma_p})_{p \in \cP}$, where $\underline{\Sigma_p}$ is an object in $\Comod_{n,p}$, satisfies the conditions of \cref{prop:ultraenlocalcomparisonfunctor}. By \cref{bounds}, it suffices to show that the vanishing-degree of $\underline{\Sigma_p}$ is bounded for large enough $p$. For $X$ and $Y$ objects in $\Comod_{n,p}$, consider the spectral sequence associated to the cosimplicial spectrum
\[
\Hom_{\Comod_{n,p}}(X,Y \otimes (\underline{\Sigma_p})^{\otimes \bullet +1}).
\] 
But as the cohomological dimension of the stack represented by $(B_p, \Sigma_p)$ is $n^2+n$ for large enough primes (\cref{prop:cohomdim}) there is a horizontal vanishing line at the $E_2$-page with intercept $n^2+n$ (which does not depend on $p$).
\end{proof}

\begin{lem} \label{formalcollapse}
Let $A^{\bullet}$ be a formal cosimplicial spectrum. Then the associated Bousfield--Kan spectral sequence collapses at the $E_2$-page.
\end{lem}
\begin{proof}
Formality implies that $A^{\bullet} \simeq \prod_{n \in \Z} A_{n}^{\bullet}$, where the homotopy groups of the spectra in $A_{n}^{\bullet}$ are concentrated in degree $n$. Since the spectral sequence associated to $A_{n}^{\bullet}$ collapses at the $E_2$-page, the result follows for $A^{\bullet}$.
\end{proof}

\begin{lem}
There is an equivalence of cosimplicial $\bE_{\infty}$-ring spectra
\[
\Hom_{\Fr}(1_{\Fr}, A_{p}^{\otimes \bullet +1}) \simeq (\E^{\otimes \bullet +1})_{\star}.
\]
In particular, \cref{formalcollapse} implies that the associated spectral sequence collapses at the $E_2$-page.
\end{lem}
\begin{proof}
Because $\underline{P}$ is symmetric monoidal, 
\[
1_{\Fr} \simeq \underline{P}(1_{\Comod_{n,p}}) \simeq \underline{P}(\underline{B_p}).
\] 
Thus
\begin{align*}
\Hom_{\Fr}(1_{\Fr}, A_{p}^{\otimes \bullet +1}) &\simeq \Hom_{\Fr}(\underline{P}(\underline{B_p}), \underline{P}(\underline{\Sigma_p})^{\otimes \bullet +1}) \\
&\simeq \Hom_{\Comod_{n,p}}(\underline{B_p}, \underline{U}(\underline{P}(\underline{\Sigma_p}^{\otimes \bullet +1}))) \\
&\simeq \Hom_{\Comod_{n,p}}(\underline{B_p}, \underline{UP({\Sigma_p}^{\otimes \bullet +1})}) \\
&\simeq \hom_{\Ch((B_p,\Sigma_p)-\comod)}(B_p,UP(\Sigma_{p}^{\otimes \bullet +1})) \\
&\simeq \hom_{\Ch((B_p,\Sigma_p)-\comod)}(B_p,UP(B_p) \otimes (\Sigma_{p}^{\otimes \bullet +1})) \\
&\simeq \hom_{\Ch(B_p - \mod)}(B_p,UP(B_p) \otimes (\Sigma_{p}^{\otimes \bullet})) \\
&\simeq UP(B_p) \otimes_{B_p} (\Sigma_{p}^{\otimes \bullet})\\
&\simeq UP(B_p) \otimes_{B_p} (\E^{\otimes \bullet +1})_0 \\
&\simeq (\E)_{\star} \otimes_{B_p} (\E^{\otimes \bullet +1})_0 \\
&\simeq (\E^{\otimes \bullet+1})_{\star}.
\end{align*}
From the first to the second line we use the $(\underline{U},\underline{P})$-adjunction and that $\underline{P}$ is monoidal (\cref{cor:franke}). From the second to the third line we use that $\Sigma_{p}$ is flat and that $UP$ preserves weak equivalences. From the third to fourth line we use that $UP(\Sigma_{p}^{\otimes \bullet +1})$ is a cosimplicial fibrant object and $E_0$ is cofibrant. From the fourth to fifth line we use \cref{lem:projformula}. From the fifth to the sixth line we use the free-forgetful adjunction between $(B_p,\Sigma_{p})$-comodules and $B_p$-modules. From the sixth to the seventh line we use the enrichment of $\Ch(B_p - \mod)$ in itself. From the seventh to the eighth line we use the flatness of $\Sigma_p$ over $B_p$. From the eighth to the ninth line we use the definition of $(-)_{\star}$ from \cref{defn:star} and the fact that $L = \pi_2(\E)$.
\end{proof}

\begin{lem}
The Picard group of $\Mod_{E_{\star}}$ is $\Z/2$ (independent of $p$).
\end{lem}
\begin{proof}
This follows from Baker--Richter \cite{bakerrichter_inv}. Their result makes use of the fact that $E_{\star}$ is an even periodic $\bE_{\infty}$-ring spectrum with regular local Noetherian $\pi_0$.
\end{proof}

\begin{cor} \label{algequiv}
There is a canonical symmetric monoidal equivalence of $\infty$-categories
\[
\Prod{\cF}^{\doubleflat} \Mod_{(\E)_{\star}} \lra{\simeq} \Prod{\cF}^{\Pic} \Mod_{(\E)_{\star}}.
\]
\end{cor}

Now \cref{maincor} applies to the collections $(\cC_p)_{p \in P} = (\Frnp)_{p \in P}$ and $(A_p)_{p \in P}$: 
\begin{cor} \label{cor:frequiv}
There is a canonical symmetric monoidal equivalence of symmetric monoidal $\infty$-categories
\[
\Prod{\cF}^{\Pic} \Frnp \lra{\simeq} \Loc \Pic \Tot (\Prod{\cF}^{\flat} \Mod_{(\E^{\otimes \bullet +1})_{\star}}).
\]
\end{cor}

\subsection{The proof of the main result}
We promote \cref{thm:sec4thm}, the main result of \cref{sec:formality}, to an equivalence of $\infty$-categories: 

\begin{thm} \label{cosimplicialcor}
Let $\cF$ be a non-principal ultrafilter on $\cP$. There is a  symmetric monoidal equivalence of cosimplicial compactly generated $\Q$-linear stable $\infty$-categories
\[
\Prod{\cF}^{\flat} \Mod_{\E^{\otimes \bullet+1}} \simeq \Prod{\cF}^{\flat} \Mod_{(\E^{\otimes \bullet+1})_{\star}}.
\]
\end{thm}
\begin{proof}
\cref{thm:sec4thm} produces a cosimplicial $\bE_{\infty}$-ring spectrum
\[
\Prod{\cF}\E^{\otimes \bullet +1} \simeq \Prod{\cF}(\E^{\otimes \bullet +1})_{\star}.
\] 
Applying modules $\Mod_{(-)}$ to this gives a symmetric monoidal equivalence of cosimplicial $\infty$-categories
\[
\Mod_{\Prod{\cF}\E^{\otimes \bullet +1}} \simeq \Mod_{\Prod{\cF}(\E^{\otimes \bullet +1})_{\star}}.
\]
\cref{prop:protomodules} applies to this to give a symmetric monoidal equivalence
\[
\Prod{\cF}^{\doubleflat} \Mod_{\E^{\otimes \bullet +1}} \simeq \Prod{\cF}^{\doubleflat} \Mod_{(\E^{\otimes \bullet +1})_{\star}}.
\]
Since the $\bE_{\infty}$-ring spectra are even periodic, \cref{cor:evenrings} applies to give a symmetric monoidal equivalence
\[
\Prod{\cF}^{\flat} \Mod_{\E^{\otimes \bullet +1}} \simeq \Prod{\cF}^{\flat} \Mod_{(\E^{\otimes \bullet +1})_{\star}}. \qedhere
\] \end{proof}

Finally, we may prove the main result of the paper:

\begin{thm} \label{superduper}
There is a symmetric monoidal equivalence of compactly generated $\Q$-linear stable $\infty$-categories
\[
\Prod{\cF}^{\Pic} \Sp_{n,p} \simeq \Prod{\cF}^{\Pic} \Frnp,
\]
for any non-principal ultrafilter $\cF$ on the prime numbers. 
\end{thm}
\begin{proof}
By \cref{cosimplicialcor}, there is a cosimplicial equivalence
\[
\Prod{\cF}^{\flat} \Mod_{\E^{\otimes \bullet+1}} \simeq \Prod{\cF}^{\flat} \Mod_{(\E^{\otimes \bullet+1})_{\star}}.
\]
This induces a symmetric monoidal equivalence of $\infty$-categories
\[
\Loc \Pic \Tot \Prod{\cF}^{\flat} \Mod_{\E^{\otimes \bullet+1}} \simeq \Loc \Pic \Tot  \Prod{\cF}^{\flat} \Mod_{(\E^{\otimes \bullet+1})_{\star}}.
\]
Now applying \cref{topequiv} and \cref{cor:frequiv} gives the desired equivalence
\[
\Prod{\cF}^{\Pic} \Sp_{n,p} \simeq \Prod{\cF}^{\Pic} \Frnp. \qedhere
\]
\end{proof}

\section{Applications}
In this section, we give an application of the equivalence of \cref{superduper} to local generalized Moore spectra. We show that, for large enough primes, such spectra exist and admit coherent multiplicative structure. 

\subsection{Compact $\infty$-operads}

Let $\Op$ be the $\infty$-category of $\infty$-operads, introduced in \cite[Section 2]{ha}. Overloading notation, we will write $\otimes$ for the Boardman--Vogt tensor product of $\infty$-operads constructed in \cite[Section 2.2.5]{ha}.  Let $\sCat= \CAlg(\Cat)$ be the $\infty$-category of symmetric monoidal $\infty$-categories.

The next result summarizes the salient features of $\Op$ needed in this section.

\begin{prop}\label{prop:opprop} The $\infty$-category $\Op$ has the following features:
	\begin{enumerate}
		\item $\Op$ is a presentable $\infty$-category. 
		\item	The forgetful functor $U\colon \sCat \to \Op$ admits a left adjoint, the monoidal envelope functor.
		\item The functor $U$ preserves filtered colimits.
	\end{enumerate}
\end{prop}
\begin{proof}
The first item is a consequence of \cite[Proposition 2.1.4.6]{ha}, while the second one is proven in  \cite[Section 2.2.4]{ha}. 

Recall now that a functor $F \colon \cC \to \cD$ is said to detect (filtered) colimits if for every (filtered) $\infty$-category $I$ a diagram $\rho \colon I^{\rhd} \to \cC$ is colimit cone if and only if $F\circ\rho$ is a colimit cone.
Note that if $G\colon \cD \to \cE$ detects (filtered) colimits   then $G\circ F$ detects (filtered) colimits if and only if $F$ detects (filtered) colimits. Also, if $F$ detects (filtered) colimits it preserves  (filtered) colimits. The converse is true assuming $F$ is conservative.

In order to see that $U$ preserves filtered colimits, consider the sequence of functors
\[
\xymatrix{\sCat \ar[r]^-{U} &   \Op \ar[r]^-{U_1} & \left(\Cat\right) _{/\Fins} \ar[r]^-{U_2} & \Cat.}
\]
Here $U_1$ and $U_2$ are the canonical forgetful functors. Note that $U$, $U_1$, and $U_2$ are all conservative and that $U_2$ admits a right adjoint (given by $\cC \mapsto (\cC \times \Fins \to \Fins$). We thus get that $U_2$ detects all colimits. In  \cite[Theorem 3.3.1]{gh_enrichedoperads}, it is shown that the non-symmetric analogue of the composite $U_2 \circ U_1$ detects filtered colimits and thus that $U_1$ detects filtered colimits. Their proof is quite general; replacing $\Delta$ by $\Fins$ gives a proof that the composite $U_2 \circ U_1$ detects filtered colimits and thus that $U_1$ detects filtered colimits. By \cite[Corollary 3.3.3]{gh_enrichedoperads} $U_1 \circ U$ detects  filtered colimits and thus $U$ detects filtered colimits.\end{proof}

As an immediate consequence of \Cref{prop:opprop} (2) and (3), we obtain:

\begin{cor}\label{cor:operadicultraproducts}
The forgetful functor $U\colon \sCat \to \Op$ preserves ultraproducts. 
\end{cor}

An $\infty$-operad $\cO$ is called compact if it is a compact object in $\Op$. We would like to have a recognition principle for compact $\infty$-operads. To this end we shall recall the description of  
$\Op$ as the underlying $\infty$-category of a model structure on the category of $\infty$-preoperads.
 
Let $\MS$ be the $1$-category of marked simplicial sets and let $F=(\Fins,M) \in \MS$ be the nerve of $\mathrm{Fin}_{*}$ with the inert edges marked.  
Recall from \cite[Section 2.1.4]{ha} that the category of $\infty$-preoperads $\POp$ is defined to be $\MS_{/F}$. For an object $\bar{X}= (X \to F)$ in $\POp$, we denote by $X$ the underlying simplicial set and by $U(\bar{X})$ the underlying object in $\sSet_{/\mathrm{Fin}_{*}}$ given by forgetting the marked simplices. The category $\POp$ is simplicially enriched and, given objects $\bar{X}$ and $\bar{Y}$, we will write $\map(\bar{X},\bar{Y})$ for the simplicial mapping space. 

\begin{prop} \label{popprops}
The category $\POp$ admits a left proper combinatorial simplicial model structure with the following properties:  
\begin{enumerate}
\item A morphism $f\colon \bar{X} \to \bar{Y}$ in $\POp$ is a cofibration if and only if it induces a monomorphism between the underlying simplicial sets $X$ and $Y$.
\item $\bar{X}= (X \to  F)$ is fibrant if and only if $U(\bar{X})$ is an $\infty$-operad and the marked edges in $\bar{X}$ are the inert edges.
\item Fibrant objects in $\POp$ are closed under filtered colimits. 
\item The localization functor $$R\colon \POp \to \Op $$ is essentially surjective, commutes with filtered colimits, and preserves compact objects.
\item $\POp^{\omega}$ consists of the objects $\bar{X}$ such that $X$ has finitely many non-degenerate  simplices.  
\item $\Op$ is compactly generated by the objects in $R(\POp^{\omega})$.
\end{enumerate}
\end{prop} 
\begin{proof}
The model structure on $\POp$ is constructed in \cite[Proposition 2.1.4.6]{ha}, where properties (1) and (2) are explicitly stated. Part (3) follows from \cite[Proposition B.1.6]{ha} and the observation that all generating $\mathfrak{P}$-anodyne maps have finite source and target. Since $R$ is a localization functor, it is essentially surjective. 

The model structure on $\POp$ is defined by applying \cite[A.2.6.15]{htt}. Thus by \cite[Remark A.2.6.16]{htt}, weak equivalences are preserved under filtered colimits and so the functor $R$ preserves filtered colimits. Next we show that $R$ preserves compact objects. Let $\bar{X} \in  \POp^{\omega}$ and let $F \colon I \to \Op$ be a filtered diagram together with a choice of lift $\hat{F} \colon I \to \POp$ such that $\hat{F}(i)$ is fibrant for all $i \in I$. Then we have 
\begin{align*}
\Map(R(\bar{X}),\colim_I F(i)) &\simeq \map(\bar{X}, \colim_I \hat{F}(i)) \\  
&\simeq \colim_I \map(\bar{X}, \hat{F}(i)) \\ 
&\simeq \colim_I \Map(R(\bar{X}), F(i)).
\end{align*}
Part (5) follows from the definition of $\POp$. To get Part (6) we combine the earlier parts together with the fact that $\POp$ is compactly generated as it is an over category over $\MS$.
\end{proof}

In light of Part (5) of \cref{popprops}, we call an $\infty$-operad $\mathcal{O}$ finite if it is equivalent to $R(\bar{X})$ for some $\bar{X} \in \POp^{\omega}$.
\begin{cor}\label{l:2}
An $\infty$-operad $\mathcal{O}$ is compact if and only if it is a retract of a finite $\infty$-operad.
\end{cor}

\begin{cor}\label{lem:compactoptensor}
Let $\cO$ and $\cO'$ be  compact $\infty$-operads. Then $\cO \otimes \cO'$ is a compact $\infty$-operad. 
\end{cor}
\begin{proof}
\cite[Notation 2.2.5.5 and Proposition 2.2.5.13]{ha} show that the tensor product of finite $\infty$-preoperads is finite, so the claim follows from \cref{l:2}.
\end{proof}

In \cite[Remark 4.1.4.8]{ha}, Lurie constructs an $\infty$-operad $\cO(m)$, which is a non-unital version of the associative operad. It is constructed by applying $R$ to an object in $\POp$ with finitely many non-degenerate simplices. Thus $\cO(m)$ is compact by Parts (4) and (5) of \cref{popprops}. In \cite[Corollary 2.3.1.8, Proposition 2.3.1.9]{ha}, Lurie produces a unitalization functor for $\infty$-operads and proves that it is given by $- \otimes \bE_0$. Define $\bA_m$ to be the $\infty$-operad $\cO(m) \otimes \bE_0$.

\begin{ex}\label{ex:anoperad}
For any $m$, the operad $\bA_{m}$ is compact. 
\end{ex}
\begin{proof}
By \cite[Example 2.1.4.9]{ha}, $\mathbb{E}_0$ is compact. Now $\bA_{m}$ is compact by \cref{lem:compactoptensor}.
\end{proof}

\subsection{Existence of multiplicative local generalized Moore spectra}

In \Cref{cosimplicialcor} we established a symmetric monoidal equivalence of cosimplicial compactly generated $\Q$-linear stable $\infty$-categories
\[
\xymatrix{\Phi_{\bullet+1}\colon \Prod{\cF}^{\flat}\Mod_{\E^{\otimes \bullet +1}} \ar[r]^-{\simeq} & \Prod{\cF}^{\flat}\Mod_{(\E^{\otimes \bullet +1})_{\star}}.}
\]
In this subsection, we describe the compatibility of this equivalence with the homotopy groups functor and use this to study the existence of $\E$-local generalized Moore spectra and their multiplicative structures. 

For any $p$ and $n$, there are canonical isomorphisms $\varphi_p\colon \pi_0\E \cong \pi_0(\E)_{\star}$. If $\cF$ is any ultrafilter on $\cP$, these isomorphisms assemble into an isomorphism $\varphi = \Prod{\cF}\varphi_p \colon \Prod{\cF}\pi_0\E \cong \Prod{\cF}\pi_0(\E)_{\star}$ and we denote by 
\[
\xymatrix{\Phi_1'\colon (\Prod{\cF}\pi_0\E)-\Mod^{\gr} \ar[r] & (\Prod{\cF}\pi_0(\E)_{\star})-\Mod^{\gr}}
\]
the induced equivalence between the associated (1-)categories of graded modules.

\begin{prop}\label{prop:goodformality}
For any non-principal ultrafilter $\cF$ on $\cP$, there is a commutative diagram
\begin{equation}\label{eq:twosquares}
\xymatrix{\Prod{\cF}^{\flat}\Mod_{\E} \ar[r]^-{\simeq} \ar[d]_{\Phi_1}^{\simeq} & \Mod_{\Prod{\cF}\E} \ar[r]^-{\pi_*} \ar[d]^{\simeq} & (\Prod{\cF}\pi_0\E)-\Mod^{\gr} \ar[d]_{\Phi_1'}^{\simeq} \\
\Prod{\cF}^{\flat}\Mod_{(\E)_{\star}} \ar[r]_{\simeq} & \Mod_{\Prod{\cF}(\E)_{\star}} \ar[r]_-{\pi_*} & (\Prod{\cF}\pi_0(\E)_{\star})-\Mod^{\gr}.}
\end{equation}
\end{prop}
\begin{proof}
The result follows by unwinding the proof of \Cref{cosimplicialcor}. Indeed, the middle vertical equivalence arises from an equivalence of $\mathbb{E}_{\infty}$-ring spectra
\[
\xymatrix{\phi \colon\Prod{\cF}(\E)_{\star} \ar[r]^-{\simeq} & \Prod{\cF}\E.}
\]
The equivalence $\Phi_1$ is then defined to make the left square commute, so it remains to consider the right square. By construction, $\phi$ is induced from a canonical map $\Prod{\cF}H\pi_0\E \to \Prod{\cF}\E$, which fits into a commutative diagram
\[
\xymatrix{\Prod{\cF}\pi_0\E \ar[r]^-{\varphi}_-{\cong} \ar[d]_{\cong} & \Prod{\cF}\pi_0(\E)_{\star} \ar[d]^{\cong} \\
\pi_0\Prod{\cF}\E \ar[r]^-{\pi_0\phi}_-{\cong} & \pi_0\Prod{\cF}(\E)_{\star},}
\]
where the vertical isomorphisms are witnessing the fact that $\pi_0$ commutes with ultraproducts. It follows that $\Phi_1'$ is canonically equivalent to the equivalence induced by $\pi_0\phi$, which in turn makes the right square in  \eqref{eq:twosquares} commute.
\end{proof}

Informally speaking, the above proposition shows that the effect of $\Phi_1$ on the homotopy groups of the unit is equivalent to the effect of the ultraproduct of the coordinatewise isomorphisms $\varphi_p\colon \pi_0\E \cong \pi_0(\E)_{\star}$. This is in sharp contrast to the construction of the equivalence $\phi$, which only exists after applying the ultraproduct at a non-principal ultrafilter. 

\begin{defn}
Let $I = (i_0,\ldots, i_{k-1})$ be a sequence of non-negative integers and fix a height $n$ and prime $p$. An $\E$-local generalized Moore spectrum of type $I$ is a compact $\E$-local spectrum $\LM_{n,p}(I)$ with $(\E)_*\LM_{n,p}(I) \cong (\E)_*/(p^{i_0},u_1^{i_1},\ldots,u_{k-1}^{i_{k-1}})$.
\end{defn}

\begin{thm}
Let $\cO$ be a compact $\infty$-operad and $n \ge 0$ an integer. Given a sequence $I = (i_0,\ldots, i_{k-1})$ of non-negative integers, there exists an integer $N = N(\cO,n,I)$ such that, for all primes $p \ge N$, an $\E$-local generalized Moore spectrum $\LM_{n,p}(I)$ of type $I$ exists and admits the structure of an $\cO$-algebra. 
\end{thm}
\begin{proof}
For every height $n$, prime $p$, and sequence $I$ as above, the strict commutative ring $(\E)_0/(p^{i_0},u_1^{i_1},\ldots,u_{k-1}^{i_{k-1}})$ in $((\E)_0,(\E)_0(\E))-\Comod$ lifts to an object
\[
\LM_{n,p}(I)^{\alg} = (\E)_{\star}/(p^{i_0},u_1^{i_1},\ldots,u_{k-1}^{i_{k-1}}) \in \CAlg(\Frnp).
\]
Since the objects $\LM_{n,p}(I)^{\alg}$ are compact and have the structure of $\bE_{\infty}$-algebras, they give rise to a map of $\infty$-operads 
\[
\xymatrix{[\LM_{n,p}(I)^{\alg}]_{\cF}\colon \bE_{\infty} \ar[r] & \Prod{p \in \cP}\Frnp^{\omega} \ar[r] & \Prod{\cF}\Frnp^{\omega} \ar[r] & (\Prod{\cF}^{\omega}\Frnp)^{\omega}.}
\]
Since $\LM_{n,p}(I)^{\alg}$ can be built out of $2^k$ many cells, independently of $p$, $[\LM_{n,p}(I)^{\alg}]_{\cF}$ factors through the symmetric monoidal inclusion (see \Cref{lem:nsymmmon})
\[
\xymatrix{(\Prod{\cF}^{\Pic}\Frnp)^{\omega} \ar[r] & (\Prod{\cF}^{\omega}\Frnp)^{\omega}.}
\]
Restricting along the unique map $\cO \to \bE_{\infty}$ gives rise to an $\cO$-algebra in $(\Prod{\cF}^{\Pic}\Frnp)^{\omega}$, which we will also denote by $[\LM_{n,p}(I)^{\alg}]_{\cF}$. Let $\Phi$ be the equivalence of \Cref{superduper} and consider the diagram
\begin{equation}\label{eq:oalgdiagram}
\xymatrix{\cO \ar[r] \ar@{-->}[dd] & (\Prod{\cF}^{\Pic}\Frnp)^{\omega} \ar[r] \ar[d]_{\Phi^{-1}}^{\simeq} & \Prod{\cF}^{\flat}\Mod_{(\E)_{\star}}^{\omega} \ar[r]^-{\pi_*} \ar[d]_{(\Phi_1)^{-1}}^{\simeq} &  (\Prod{\cF}\pi_0(\E)_{\star})-\Mod^{\gr} \ar[d]_{(\Phi_1')^{-1}}^{\simeq} \\
& (\Prod{\cF}^{\Pic}\Sp_{n,p})^{\omega} \ar[r] \ar@{^{(}->}[d] & \Prod{\cF}^{\flat}\Mod_{\E}^{\omega} \ar[r]^-{\pi_*} \ar@{^{(}->}[d] & (\Prod{\cF}\pi_0\E)-\Mod^{\gr} \\
\Prod{p \in U}\Sp_{n,p}^{\omega} \ar[r] & \Prod{\cF}\Sp_{n,p}^{\omega} \ar[r] & \Prod{\cF}\Mod_{\E}^{\omega}.}
\end{equation}
The top central square commutes by the construction of the symmetric monoidal equivalence $\Phi$, while the commutativity of the top right square is the content of \Cref{prop:goodformality}. The indicated arrows are symmetric monoidal inclusions by \Cref{lem:nsymmmon}, and it follows that the bottom square commutes as well. 

The ultraproduct $\Prod{\cF}\Sp_{n,p}^{\omega}$ can be computed in $\sCat$, the $\infty$-category of symmetric monoidal $\infty$-categories, and hence in $\Op$ by \Cref{cor:operadicultraproducts}, so we obtain a map of $\infty$-operads
\[
\cO \xrightarrow{[\LM_{n,p}(I)^{\alg}]_{\cF}} (\Prod{\cF}^{\Pic}\Frnp)^{\omega} \xrightarrow{\Phi^{-1}} (\Prod{\cF}^{\Pic}\Sp_{n,p})^{\omega} \longrightarrow  \Prod{\cF}\Sp_{n,p}^{\omega},
\]
where we suppress the superscript $\otimes$ for ease of readability. Since $\cO$ is a compact $\infty$-operad, there exists $U \in \cF$ and the dashed factorization in Diagram \eqref{eq:oalgdiagram}. That is, there is an $\cO$-algebra in $\Prod{p \in U}\Sp_{n,p}^{\omega}$ corresponding to $[\LM_{n,p}(I)^{\alg}]_{\cF}$. For $p \in U$, we will denote the corresponding $\cO$-algebra in $\Sp_{n,p}^{\omega}$ by $\LM_{n,p}(I)$.

It remains to identify the $\E$-homology of the underlying object of $\LM_{n,p}(I)$. To this end, it suffices to compute the composite
\[
\xymatrix{\cO \ar[r] & (\Prod{\cF}^{\Pic}\Sp_{n,p})^{\omega} \ar[r] & \Prod{\cF}^{\flat}\Mod_{(\E)}^{\omega} \ar[r]^-{\pi_*} & (\Prod{\cF}\pi_0(\E))-\Mod^{\gr}}
\]
From the commutativity of the top part of Diagram \eqref{eq:oalgdiagram}, we deduce that this is isomorphic to $\Prod{\cF}(\E)_*/(p^{i_0},u_1^{i_1},\ldots,u_{k-1}^{i_{k-1}})$. Thus there exists $V \subseteq U$ with $V \in \cF$ such that, for all $p \in V$, $\LM_{n,p}(I)$ has the desired $\E$-homology. 

Applying the above procedure to every non-principal ultrafilter, the claim follows from the fact that the intersection of all non-principal ultrafilters on $\cP$ is the Frechet filter on $\cP$. 
\end{proof}

Specializing to the associative $\infty$-operad $\bA_m$, which is compact by \cref{ex:anoperad}, we obtain the existence of associative multiplicative structure on the local generalized Moore spectra for large enough primes. 

\begin{cor}
Let $n, m \ge 0$ and let $I = (i_0,\ldots, i_{k-1})$ be a $k$-tuple of natural numbers, there is an integer $N = N(n,I,m)$ such that for all primes $p \ge N$ an $\E$-local generalized Moore spectrum $\LM_{n,p}(I)$ exists and admits the structure of an $\bA_m$-algebra spectrum. 
\end{cor}

\bibliographystyle{amsalpha}
\bibliography{bibliography}

\providecommand{\bysame}{\leavevmode\hbox to3em{\hrulefill}\thinspace}
\providecommand{\MR}{\relax\ifhmode\unskip\space\fi MR }
\providecommand{\MRhref}[2]{%
  \href{http://www.ams.org/mathscinet-getitem?mr=#1}{#2}
}
\providecommand{\href}[2]{#2}
\begin{thebibliography}{GHMR05}

\bibitem[AFT17]{AFT}
David Ayala, John Francis, and Hiro~Lee Tanaka, \emph{Factorization homology of
  stratified spaces}, Selecta Math. (N.S.) \textbf{23} (2017), no.~1, 293--362.
  \MR{3595895}

\bibitem[AK65a]{AxKoc1}
James Ax and Simon Kochen, \emph{Diophantine problems over local fields. {I}},
  Amer. J. Math. \textbf{87} (1965), 605--630. \MR{0184930}

\bibitem[AK65b]{AxKoc2}
\bysame, \emph{Diophantine problems over local fields. {II}. {A} complete set
  of axioms for {$p$}-adic number theory}, Amer. J. Math. \textbf{87} (1965),
  631--648. \MR{0184931}

\bibitem[AK66]{AxKoc3}
\bysame, \emph{Diophantine problems over local fields. {III}. {D}ecidable
  fields}, Ann. of Math. (2) \textbf{83} (1966), 437--456. \MR{0201378}

\bibitem[ALR03]{alr}
J.~Ad\'amek, F.~W. Lawvere, and J.~Rosick\'y, \emph{Continuous categories
  revisited}, Theory Appl. Categ. \textbf{11} (2003), No. 11, 252--282.
  \MR{1988399}

\bibitem[Ax68]{pseudofinite}
James Ax, \emph{The elementary theory of finite fields}, Ann. of Math. (2)
  \textbf{88} (1968), 239--271. \MR{0229613}

\bibitem[Bea17]{beaudry}
Agn\`es Beaudry, \emph{The chromatic splitting conjecture at {$n = p = 2$}},
  Geom. Topol. \textbf{21} (2017), no.~6, 3213--3230. \MR{3692966}

\bibitem[BGH17]{bgh_csc}
A.~{Beaudry}, P.~G. {Goerss}, and H.-W. {Henn}, \emph{{Chromatic splitting for
  the $K(2)$-local sphere at $p=2$}}, ArXiv e-prints (2017).

\bibitem[BH]{bhacht}
Tobias Barthel and Drew Heard, \emph{Algebraic chromatic homotopy theory for
  {$BP_*BP$}-comodules}, preprint available at
  \url{https://arxiv.org/abs/1708.09261}.

\bibitem[Bla77]{B}
Andreas Blass, \emph{A model without ultrafilters}, Bull. Acad. Polon. Sci.
  S\'er. Sci. Math. Astronom. Phys. \textbf{25} (1977), no.~4, 329--331.
  \MR{0476510}

\bibitem[Bou79]{Bousfieldlocalization}
A.~K. Bousfield, \emph{The localization of spectra with respect to homology},
  Topology \textbf{18} (1979), no.~4, 257--281. \MR{551009 (80m:55006)}

\bibitem[Bou87]{bousfield_homolss}
\bysame, \emph{On the homology spectral sequence of a cosimplicial space},
  Amer. J. Math. \textbf{109} (1987), no.~2, 361--394. \MR{882428}

\bibitem[BR05]{bakerrichter_inv}
Andrew Baker and Birgit Richter, \emph{Invertible modules for commutative
  {$\Bbb S$}-algebras with residue fields}, Manuscripta Math. \textbf{118}
  (2005), no.~1, 99--119. \MR{2171294}

\bibitem[BR11]{barnesroitzheim_monoidality}
David Barnes and Constanze Roitzheim, \emph{Monoidality of {F}ranke's exotic
  model}, Adv. Math. \textbf{228} (2011), no.~6, 3223--3248. \MR{2844942}

\bibitem[BS69]{bellslomson}
J.~L. Bell and A.~B. Slomson, \emph{Models and ultraproducts: {A}n
  introduction}, North-Holland Publishing Co., Amsterdam-London, 1969.
  \MR{0269486}

\bibitem[CK90]{changkeisler}
C.~C. Chang and H.~J. Keisler, \emph{Model theory}, third ed., Studies in Logic
  and the Foundations of Mathematics, vol.~73, North-Holland Publishing Co.,
  Amsterdam, 1990. \MR{1059055}

\bibitem[DH95]{dh_lt}
Ethan~S. Devinatz and Michael~J. Hopkins, \emph{The action of the {M}orava
  stabilizer group on the {L}ubin-{T}ate moduli space of lifts}, Amer. J. Math.
  \textbf{117} (1995), no.~3, 669--710. \MR{1333942}

\bibitem[Fra]{Franke_exotic}
Jens Franke, \emph{Uniqeness theorems for certain triangulated categories with
  an {A}dams spectral sequence}, \url{http://www.math.uiuc.edu/K-theory/0139/}.

\bibitem[GH]{ghproblems}
P.~G. Goerss and M.~J. Hopkins, \emph{Moduli problems for structured ring
  spectra.}, \url{www.math.northwestern.edu/~pgoerss/spectra}.

\bibitem[GH15]{gh_enrichedoperads}
David Gepner and Rune Haugseng, \emph{Enriched {$\infty$}-categories via
  non-symmetric {$\infty$}-operads}, Adv. Math. \textbf{279} (2015), 575--716.
  \MR{3345192}

\bibitem[GHM14]{ghm}
Paul~G. Goerss, Hans-Werner Henn, and Mark Mahowald, \emph{The rational
  homotopy of the {$K(2)$}-local sphere and the chromatic splitting conjecture
  for the prime 3 and level 2}, Doc. Math. \textbf{19} (2014), 1271--1290.
  \MR{3312144}

\bibitem[GHMR05]{ghmr}
P.~Goerss, H.-W. Henn, M.~Mahowald, and C.~Rezk, \emph{A resolution of the
  {$K(2)$}-local sphere at the prime 3}, Ann. of Math. (2) \textbf{162} (2005),
  no.~2, 777--822. \MR{2183282 (2006j:55016)}

\bibitem[Gla16]{glasman}
Saul Glasman, \emph{A spectrum-level {H}odge filtration on topological
  {H}ochschild homology}, Selecta Math. (N.S.) \textbf{22} (2016), no.~3,
  1583--1612. \MR{3518559}

\bibitem[Hen07]{henn_finiteresolutions}
Hans-Werner Henn, \emph{On finite resolutions of {$K(n)$}-local spheres},
  Elliptic cohomology, London Math. Soc. Lecture Note Ser., vol. 342, Cambridge
  Univ. Press, Cambridge, 2007, pp.~122--169. \MR{2330511}

\bibitem[Heu]{heuts_goodwillie}
Gijs Heuts, \emph{Goodwillie approximations to higher categories},
  \url{http://arxiv.org/abs/1510.03304}.

\bibitem[Hin16]{Hinich}
Vladimir Hinich, \emph{Dwyer-{K}an localization revisited}, Homology Homotopy
  Appl. \textbf{18} (2016), no.~1, 27--48. \MR{3460765}

\bibitem[HL]{ambidexterity}
Michael Hopkins and Jacob. Lurie, \emph{Ambidexterity in $k(n)$-local stable
  homotopy theory.}, \url{http://www.math.harvard.edu/~lurie/}.

\bibitem[Hop]{coctalos}
Michael Hopkins, \emph{Complex oriented cohomology theories and the language of
  stacks},
  \url{http://www.math.rochester.edu/people/faculty/doug/otherpapers/coctalos.pdf}.

\bibitem[Hov95]{hoveycsc}
Mark Hovey, \emph{Bousfield localization functors and {H}opkins' chromatic
  splitting conjecture}, The \v {C}ech centennial ({B}oston, {MA}, 1993),
  Contemp. Math., vol. 181, Amer. Math. Soc., Providence, RI, 1995,
  pp.~225--250. \MR{1320994 (96m:55010)}

\bibitem[Hov04]{hoveyhcomodules}
\bysame, \emph{Homotopy theory of comodules over a {H}opf algebroid}, Homotopy
  theory: relations with algebraic geometry, group cohomology, and algebraic
  {$K$}-theory, Contemp. Math., vol. 346, Amer. Math. Soc., Providence, RI,
  2004, pp.~261--304. \MR{2066503 (2005f:18011)}

\bibitem[HPS99]{hopkinsvanishinglines}
M.~J. Hopkins, J.~H. Palmieri, and J.~H. Smith, \emph{Vanishing lines in
  generalized {A}dams spectral sequences are generic}, Geom. Topol. \textbf{3}
  (1999), 155--165 (electronic). \MR{1697180 (2000c:55017)}

\bibitem[HS98]{nilp2}
Michael~J. Hopkins and Jeffrey~H. Smith, \emph{Nilpotence and stable homotopy
  theory. {II}}, Ann. of Math. (2) \textbf{148} (1998), no.~1, 1--49.
  \MR{1652975 (99h:55009)}

\bibitem[HS99a]{hoveysadofskyinv}
Mark Hovey and Hal Sadofsky, \emph{Invertible spectra in the {$E(n)$}-local
  stable homotopy category}, J. London Math. Soc. (2) \textbf{60} (1999),
  no.~1, 284--302. \MR{1722151 (2000h:55017)}

\bibitem[HS99b]{hoveystrickland}
Mark Hovey and Neil~P. Strickland, \emph{Morava {$K$}-theories and
  localisation}, Mem. Amer. Math. Soc. \textbf{139} (1999), no.~666, viii+100.
  \MR{MR1601906 (99b:55017)}

\bibitem[Iye07]{iyengar}
Srikanth Iyengar, \emph{Andr{\'e}-quillen homology of commutative algebras},
  Contemporary Mathematics \textbf{436} (2007), 203.

\bibitem[Lura]{ha}
Jacob Lurie, \emph{Higher algebra}, \url{http://www.math.harvard.edu/~lurie/}.

\bibitem[Lurb]{sag}
\bysame, \emph{Spectral algebraic geometry},
  \url{http://www.math.harvard.edu/~lurie/}.

\bibitem[Lur09]{htt}
\bysame, \emph{Higher topos theory}, Annals of Mathematics Studies, vol. 170,
  Princeton University Press, Princeton, NJ, 2009. \MR{2522659 (2010j:18001)}

\bibitem[Lur18]{LurE}
\bysame, \emph{Elliptic cohomology {II}: Orientations},
  \url{http://www.math.harvard.edu/~lurie/}, 2018.

\bibitem[Mat15]{mathew_thick}
Akhil Mathew, \emph{A thick subcategory theorem for modules over certain ring
  spectra}, Geom. Topol. \textbf{19} (2015), no.~4, 2359--2392. \MR{3375530}

\bibitem[Mat16]{mathew_galois}
\bysame, \emph{The {G}alois group of a stable homotopy theory}, Adv. Math.
  \textbf{291} (2016), 403--541. \MR{3459022}

\bibitem[MG]{mazelgee_adj}
Aaron Mazel-Gee, \emph{Quillen adjunctions induce adjunctions of
  quasicategories}, \url{https://arxiv.org/abs/1501.03146}.

\bibitem[Mit97]{mitchellthomasondescent}
Stephen~A. Mitchell, \emph{Hypercohomology spectra and {T}homason's descent
  theorem}, Algebraic {$K$}-theory ({T}oronto, {ON}, 1996), Fields Inst.
  Commun., vol.~16, Amer. Math. Soc., Providence, RI, 1997, pp.~221--277.
  \MR{1466977 (99f:19002)}

\bibitem[Mor85]{moravastack}
Jack Morava, \emph{Noetherian localisations of categories of cobordism
  comodules}, Ann. of Math. (2) \textbf{121} (1985), no.~1, 1--39. \MR{782555}

\bibitem[MS16]{ms_picard}
Akhil Mathew and Vesna Stojanoska, \emph{The {P}icard group of topological
  modular forms via descent theory}, Geom. Topol. \textbf{20} (2016), no.~6,
  3133--3217. \MR{3590352}

\bibitem[Nav10]{nave}
Lee~S. Nave, \emph{The {S}mith-{T}oda complex {$V((p+1)/2)$} does not exist},
  Ann. of Math. (2) \textbf{171} (2010), no.~1, 491--509. \MR{2630045}

\bibitem[Pat17]{patchkoria}
Irakli Patchkoria, \emph{On exotic equivalences and a theorem of franke},
  Bulletin of the London Mathematical Society \textbf{49} (2017), no.~6,
  1085--1099.

\bibitem[Rav86]{ravbook1}
Douglas~C. Ravenel, \emph{Complex cobordism and stable homotopy groups of
  spheres}, Pure and Applied Mathematics, vol. 121, Academic Press, Inc.,
  Orlando, FL, 1986. \MR{860042 (87j:55003)}

\bibitem[Rav92]{ravbook2}
\bysame, \emph{Nilpotence and periodicity in stable homotopy theory}, Annals of
  Mathematics Studies, vol. 128, Princeton University Press, Princeton, NJ,
  1992, Appendix C by Jeff Smith. \MR{1192553 (94b:55015)}

\bibitem[Rog08]{rognes_galois}
John Rognes, \emph{Galois extensions of structured ring spectra. {S}tably
  dualizable groups}, Mem. Amer. Math. Soc. \textbf{192} (2008), no.~898,
  viii+137. \MR{2387923}

\bibitem[Sch07]{schwede}
Stefan Schwede, \emph{The stable homotopy category is rigid}, Ann. of Math. (2)
  \textbf{166} (2007), no.~3, 837--863. \MR{2373374}

\bibitem[Sch10]{schoutensbook}
Hans Schoutens, \emph{The use of ultraproducts in commutative algebra}, Lecture
  Notes in Mathematics, vol. 1999, Springer-Verlag, Berlin, 2010. \MR{2676525}

\bibitem[SS03]{schwedeshipleymorita}
Stefan Schwede and Brooke Shipley, \emph{Stable model categories are categories
  of modules}, Topology \textbf{42} (2003), no.~1, 103--153. \MR{1928647}

\end{thebibliography}

\end{document}